\documentclass{amsart}
\usepackage{amsfonts,amsthm,amsmath}
\usepackage{amssymb}
\usepackage{amscd}
\usepackage[utf8]{inputenc}
\usepackage[a4paper]{geometry}
\usepackage{enumerate}
\usepackage[usenames,dvipsnames]{color}
\usepackage{tikz}
\usepackage{soul}
\usepackage{array}
\usepackage{array,tabularx}

\numberwithin{equation}{section}
\numberwithin{equation}{subsection}

\theoremstyle{plain}

\newtheorem{theorem}[equation]{Theorem}
\newtheorem{lemma}[equation]{Lemma}
\newtheorem{proposition}[equation]{Proposition}

\newtheorem{corollary}[equation]{Corollary}

\theoremstyle{definition}
\newtheorem{example}[equation]{Example}
\newtheorem{remark}[equation]{Remark}

\newtheorem{definition}[equation]{Definition}

\newtheorem{bekezdes}[equation]{}

\def\C{\mathbb C}
\def\Q{\mathbb Q}

\def\Z{\mathbb Z}

\def\im{{\rm im}}

\newcommand{\calb}{{\mathcal B}}

\newcommand{\calv}{{\mathcal V}}

\newcommand{\calk}{{\mathcal K}}
\newcommand{\cali}{{\mathcal I}}
\newcommand{\calj}{{\mathcal J}}

\newcommand{\calO}{{\mathcal O}}
\newcommand{\calC}{{\mathcal C}}
\newcommand{\calS}{{\mathcal S}}

\newcommand{\calL}{\mathcal{L}}

\newcommand{\tX}{\widetilde{X}}

\newcommand{\cO}{{\mathcal O}}

\newcommand{\bC}{{\mathbb C}}

\newcommand{\eca}{{\rm ECa}}
\newcommand{\pic}{{\rm Pic}}

\newcommand{\bt}{{\mathbf t}}

\newcommand{\bZ}{{\mathbb{Z}}}

\author{J\'anos Nagy}
\address{Central European University, Dept. of Mathematics,  Budapest, Hungary}
\email{nagy\textunderscore janos@phd.ceu.edu}

\author{Andr\'as N\'emethi}
\address{Alfr\'ed R\'enyi Institute of Mathematics,
Hungarian Academy of Sciences,
Re\'altanoda utca 13-15, H-1053, Budapest, Hungary \newline
 \hspace*{4mm} ELTE - University of Budapest, Dept. of Geometry, Budapest, Hungary \newline \hspace*{4mm}
BCAM - Basque Center for Applied Math.,
Mazarredo, 14 E48009 Bilbao, Basque Country – Spain}
\email{nemethi.andras@renyi.mta.hu }

\title{The Abel map for surface singularities \\
III. Elliptic germs
}

\begin{document}



\begin{abstract}
If $(\tX,E)\to (X,o)$ is the resolution of a complex normal surface singularity
and $c_1:\pic(\tX)\to H^2(\tX,\Z)$ is the Chern class map, then $\pic^{l'}(\tX):=
c_1^{-1}(l')$ has a (Brill--Noether type) stratification $W_{l', k}:=
\{\calL\in \pic^{l'}(\tX)\,:\, h^1(\calL)=k\}$. In this note we determine it for elliptic singularities together with the stratification according to the cycle of fixed components. For elliptic singularities we also characterize the End Curve Condition and Weak End Curve Condition  in terms of the Abel map, we provide several characterization of them, and finally we show that they are equivalent.
\end{abstract}

\maketitle

\linespread{1.2}


\pagestyle{myheadings} \markboth{{\normalsize  J. Nagy, A. N\'emethi}} {{\normalsize Abel maps }}


\section{Introduction}\label{s:intr}

\subsection{}
 Recall that the classical Brill--Noether problem for curve is the following.
Let $C$ be a smooth projective (complex) curve and let
$c_1:\pic(C)\to H^2(C,\Z)=\Z$ be   the first Chern
class map. Set $\pic^l(C):=c_1^{-1}(l)$. Then one
considers the stratification of $\pic^l(C)$ according to the $h^1$--value, namely,
$W_{l,k}:=\{\calL\in\pic^l(C)\,:\, h^1(\calL)=k\}$. The problem is to determine
the values $(l,k)$ for which $W_{l,k}$ is non--empty and in such cases to describe the topology of the different strata $W_{l,k}$. (This depends heavily on the analytic structure of $C$.) For details see e.g. \cite{ACGH,Flamini}.

\bekezdes For complex normal surface singularities the analogue question can be formulated as follows. Let $(X,o)$ be such a singularity  and let us fix a resolution
$\phi:(\tX,E)\to (X,o)$. We will assume that the link is a rational homology sphere,
equivalently, that the dual resolution graph $\Gamma$ is a tree of ${\mathbb P}^1$'s.
Then one has the exponential exact sequence
$$ 0\to \pic ^0(\tX)\simeq H^1(\calO_{\tX})\to \pic(\tX) \stackrel{c_1}{\longrightarrow}
L':=H^2(\tX,\Z)\to 0.$$
Here $L'$ might serve also as  the dual lattice of $L:=H_2(\tX,\Z)$ freely generated
by the irreducible exceptional divisors and endowed with the negative definite intersection form. Then for any possible Chern class $l'\in L'$ set $\pic^{l'}(\tX):=
c_1^{-1}(l')$. (Note that while $\pic^l(C)$ for a smooth curve is a {\it compact}
 complex torus,
in the surface singularity case $\pic ^{l'}(\tX)$ is a (non--compact)  affine space
$\C^{p_g}$, where $p_g$ is the geometric genus of $(X,o)$.)
Following \cite{NNI,NNII} we consider the  stratification $W_{l', k}:=\{ \calL\in
\pic^{l'}(\tX)\, :\, h^1(\calL)=k\}$. Again, the goal is to describe the spaces $W_{l',k}$.
In general, they depend in a rather arithmetical  way on the combinatorics of
the resolution graph $\Gamma$ and also on the analytic structure supported on $\Gamma$.
Usually the spaces $W_{l',k}$ are not open, nor closed, not even  quasi--projective.
They might be nonlinear, their closure might be even singular.
Though in theory of singularities several results are known for the possible
$h^1(\calL)$--values (vanishing theorems, coarse topological bounds),
a systematic study of the spaces $W_{l',k}$ was missing, in the series of article
(starting with \cite{NNI,NNII} and the present one) the authors aim to fill in this
necessity.

\bekezdes The main tool in the study of the $W$--stratification (similarly as in the curve case) is the Abel map $c^{l'}(Z):\eca^{l'}(Z)\to \pic^{l'}(Z)$,
where $Z\in L$ is a nonzero  effective cycle, and $\eca^{l'}(Z)$ is the space of
effective Cartier divisors on $Z$, cf. \cite{Groth62,Kl,Kleiman2013}. (We emphasize again
some major differences compared with the curve case: $\eca^{l'}(Z)$ is not compact, and
$c^{l'}$ is not proper, a fact which creates extra difficulties in
 the study of the fibers.) If $Z\gg 0$ then $\pic^{l'}(Z)=\pic^{l'}(\tX)$, and the
 $W$--stratification can be analysed via the fiber structure of the Abel map $c^{l'}$.
Besides the general theory, some concrete families of singularities were
already analysed,  superisolated and weighted homogeneous ones (partially)
 in \cite{NNI}, the generic analytic structure (as an extreme bound case of the theory)
 in \cite{NNII}. In this manuscript we provide a complete description for elliptic singularities.

For the theory of elliptic singularities the reader might consult
 \cite{Laufer77,MR,weakly,OkumaEll,Tomari85,Wa70,Yau5,Yau1}. The main technical machinery, which guides
 most of the properties of elliptic singularities are
 the elliptic sequences. In the numerical Gorenstein case we will use the
 sequences defined and used by Laufer and S. S.-T. Yau, however for the non--numerically
 Gorenstein case we introduce a new sequence, which mimics better the numerical Gorenstein case and it is more relevant for our purposes (for the comparison of the old and new sequences see \cite{NNtop}).

Elliptic sequences, Artin fundamental cycles on different supports and the
(anti)canonical cycle satisfies several key compatibility properties, they
are formulated (elegantly) in the minimal resolution, in any other resolution they became
uneasy. On the other hand, in \cite{NNI} we described several properties of the
space of effective Cartier divisors and the Abel map in the context of a good resolution:
in several local analysis the normal crossing property of the exceptional curves was used.
Therefore, strictly speaking, in this note we analyse only those elliptic singularities
 (with rational homology sphere links) whose minimal resolution is good.
The interested reader is invited to extend the results
 for the remaining pathological cases as well (see e.g. \cite{Laufer77} for their list
 in the minimally elliptic case). Hence, in the sequel, `elliptic' means elliptic with these restrictions.

 \bekezdes In the case of elliptic singularities the Abel map has several very pleasant properties (see Theorem \ref{th:Vk} below):

(a) the closure $\overline{ \im (c^{l'}(Z))}$
of the image of the Abel map $c^{l'}(Z)$ is an affine subspace of $\pic^{l'}(Z)$;

(b) $h^1(\calL)$ is uniform on $\overline{ \im (c^{l'}(Z))}$ (whose value is computable).

\vspace{2mm}

This solves the Brill--Noether problem on the image  $\overline{ \im (c^{l'}(Z))}
\subset \pic^{l'}(Z)$. However, usually $\overline{ \im (c^{l'}(Z))}\not=\pic^{l'}(Z)$.
Recall that the image of $c^{l'}$ consists of line bundles without fixed components.
Hence to complete the picture we need to analyse the possible cycles of fixed components, and twisting a certain  bundle with the cycle of its fixed components
  (hence  creating bundles without fixed components)
we reduce the Brill--Noether problem to the study of several Abel map images.

The fact that the closure of the Abel maps are affine subspaces
is inherited by the $W$--stratification as well: in the Gorenstein case they are organized in a flag
of linear subspaces
reflecting completely the concatenated structure of the elliptic sequence.
(In the non--Gorenstein case some additional `wandering points' also might appear.)
Different levels of generality the $W$--stratification of the Picard groups according to the $h^1$--values are completely  described in section \ref{ss:BN}. In section
\ref{s:basestrat} we provide even the finer stratification according to the
cycles of fixed components.

\bekezdes The second main goal of the article is to analyse elliptic surface
singularities satisfying the `End Curve Condition' (ECC) and
`Weak End Curve Condition' (WECC). For surface singularities (with rational homology
sphere link) singularities with ECC were introduced by Neumann and Wahl, they coincide with the very important family of splice quotient singularities of Neumann and Wahl.
The WECC imposes  a weaker condition, it appeared naturally in the generalization of
certain surgery formulae and $p_g$--formulae from the splice quotient singularities
to more general cases (e.g. in Okuma's work) and also in the study of Abel map by the
authors \cite{NNI,NNII}.
(In the WECC case the strict transform of an end--curve function can intersect the
end--exceptional curve even non--transversally and by any intersection multiplicity
contrary to the ECC case when the intersection is transversal.)
 It turns out that both ECC and WECC can be elegantly characterized by the mutual position of the `natural' line bundles and the images of the Abel maps.
(By a natural line bundle we associate in a universal way
to a given Chern class a  line bundle having  that
Chern class \cite{trieste,OkumaRat}, cf. \ref{bek:natline} here.)

Recall that a  graph $\Gamma$ is the dual graph of a resolution  of a certain
singularity (analytic type) with ECC if and only if it satisfies the
`semigroup and congruence conditions' of Neumann and Wahl
\cite{NWsq}, or equivalently, if it satisfies the  `monomial condition' of Okuma
\cite{Ok}.
In Theorem \ref{th:WECCCrit} we provide a similar combinatorial condition
 (we call it `extension criterion of the elliptic sequence') which guarantees the existence of an analytic structure with WECC (provided that the graph is elliptic). In fact, using
 this we prove in Theorem \ref{th:9.4.2}
 that for elliptic singularities the ECC and WEEC are equivalent:

 \vspace{1mm}

 (a) If an elliptic graph
 $\Gamma$ admits a WECC analytic structure then it admits an ECC as well
 (in particular, the three topological conditions ---  
 the `semigroup and congruence conditions', the `monomial condition'
 and the `extension criterion' ---  are equivalent).

 (b) If $(X,o)$ elliptic is WECC then it is ECC too  (hence it is splice quotient).

\bekezdes The structure of the article is the following. Section \ref{s:prel}
recalls the preliminary notions regarding surface singularities, Abel map, modified Abel map, and their relationships with differential forms. Section \ref{s:elliptic}
review facts regarding elliptic singularities and  defines the `new' elliptic sequence.
In section \ref{ss:lineb} we prove several identities and inequalities regarding
$h^1(\calL)$, $\calL\in \pic(\tX)$, we analyse the possible cycles of fixed components,
 and we study certain compatibilities with the elliptic
sequences. In Appendix we analyse with details an example with pathological
cycle of fixed components. Section \ref{s:AbelEll} treats the Abel map of
elliptic singularities. Several examples are listed. In section \ref{ss:BN} we describe
the  stratification of  $\pic(\tX)$ according to $h^1$, while in section \ref{s:basestrat}
the stratification according to the fixed components. The last section contains the
study of ECC and WECC elliptic singularities. We present two topological and two analytical characterizations of germs satisfying WECC.

\section{Preliminaries and notations}\label{s:prel}
\subsection{Notations regarding a resolution}  \cite{Nfive,trieste,NCL,LPhd,NNI}
Let $(X,o)$ be the germ of a complex analytic normal surface singularity.
We denote by  $p_g$ the geometric genus of $(X,o)$.
We will assume that the link $M$ of $(X,o)$ is a rational homology sphere.

Let $\phi:\widetilde{X}\to X$ be a   resolution   of $(X,o)$ with
 exceptional curve $E:=\phi^{-1}(0)$,  and  let $\cup_{v\in\calv}E_v$ be
the irreducible decomposition of $E$. Define  $E_I:=\sum_{v\in I}E_v$ for any subset $I\subset \calv$.

$L:=H_2(\widetilde{X},\mathbb{Z})$, endowed
with a negative definite intersection form  $(\,,\,)$, is a lattice. It is
freely generated by the classes of  $\{E_v\}_{v\in\mathcal{V}}$.
 The dual lattice is $L'={\rm Hom}_\Z(L,\Z)=\{
l'\in L\otimes \Q\,:\, (l',L)\in\Z\}$. It  is generated
by the (anti)dual classes $\{E^*_v\}_{v\in\mathcal{V}}$ defined
by $(E^{*}_{v},E_{w})=-\delta_{vw}$ (where $\delta_{vw}$ stays for the  Kronecker symbol).
$L'$ is also  identified with $H^2(\tX,\Z)$. 

All the $E_v$--coordinates of any $E^*_u$ are strict positive.
We define the Lipman cone as $\calS':=\{l'\in L'\,:\, (l', E_v)\leq 0 \ \mbox{for all $v$}\}$.
As a monoid it is generated over $\bZ_{\geq 0}$ by $\{E^*_v\}_v$.
Write also $\calS:=\calS'\cap L$.

$L$ embeds into $L'$ with
 $ L'/L\simeq H_1(M,\mathbb{Z})$, which is abridged by $H$.
 The class of $l'$ in $H$ is denoted by $[l']$.

There is a natural partial ordering of $L'$ and $L$: we write $l_1'\geq l_2'$ if
$l_1'-l_2'=\sum _v r_vE_v$ with all $r_v\geq 0$. We set $L_{\geq 0}=\{l\in L\,:\, l\geq 0\}$ and
$L_{>0}=L_{\geq 0}\setminus \{0\}$.

The support of a cycle $l=\sum n_vE_v$ is defined as  $|l|=\cup_{n_v\not=0}E_v$.

Since $H_1(M,\Q)=0$,
each $E_v$ is rational, and the dual graph of any good resolution is a tree.

The geometric genus $h^1(\tX, \calO_{\tX})$ of $(X,o)$ is denoted by $p_g$.

\bekezdes\label{ss:mincyc} {\bf Minimal cycles in $L'_{\geq 0}$ and in $\calS'$.}
Consider the semi-open cube $\{\sum_v l'_v E_v\in L' \ | \ 0\leq l'_v<1\}$.
 It contains a unique representative $r_h$ for every $h\in H$ so that $[r_h]=h$.
Similarly,   for any $h\in H$ there is a
 unique minimal element of $\{l'\in L' \ | \ [l']=h\}\cap \mathcal{S}'$,
 which will be denoted by $s_h$.
One has  $s_h\geq r_h$;  in general, $s_h \neq r_h$ (see e.g. Example \ref{ex:new}).

\bekezdes\label{sss:s} {\bf A `Laufer--type' computation sequence targeting $\mathcal{S}'$.}
Recall the following fact:
\begin{lemma}\ \cite[Lemma 7.4]{NOSZ}\label{lem:cs2} \
Fix any $l'\in L'$.
\begin{enumerate}
 \item[(1)] There exists a unique minimal element $s(l')$ of $(l'+L_{\geq 0})\cap \mathcal{S}'$.
 \item[(2)] $s(l')$ can be found via the following
  computation sequence $\{z_i\}_i$ connecting $l'$ and $s(l')$:
   set $z_0:=l'$, and assume that $z_i$ ($i\geq 0$) is already constructed. If
   $(z_i,E_{v(i)})>0$ for some $v(i)\in\mathcal{V}$ then  set $z_{i+1}=z_i+E_{v(i)}$. Otherwise
   $z_i\in \mathcal{S}'$ and necessarily  $z_i=s(l')$.
\end{enumerate}
\end{lemma}
In general the choice of the individual vertex $v(i)$ might not be unique, nevertheless the final
output $s(l')$ is unique. From definitions, if $l'=r_h$ then $s(l')=s_h$.

The {\it (anti)canonical cycle} $Z_K\in L'$ is defined by the
{\it adjunction formulae}
$(Z_K, E_v)=(E_v,E_v)+2$ for all $v\in\mathcal{V}$.
(It is the first Chern class of the dual of the line bundle $\Omega^2_{\tX}$.)
We write $\chi:L'\to \Q$ for the (Riemann--Roch) expression $\chi(l'):= -(l', l'-Z_K)/2$.

The singularity (or, its topological type) is called numerically Gorenstein if $Z_K\in L$.
(Since $Z_K\in L$ if and only if the line bundle $\Omega^2_{X\setminus \{o\}}$ of holomorphic 2--forms
on $X\setminus \{o\}$ is topologically trivial, see e.g. \cite{Du}, the $Z_K\in L$ property
is independent of the resolution). $(X,o)$ is called Gorenstein if $Z_K\in L$ and
$\Omega^2_{\tX}$ (the sheaf of holomorphic 2--forms) is isomorphic to $ \calO_{\tX}(-Z_K)$ (or,
equivalently, if the line bundle $\Omega^2_{X\setminus \{o\}}$ is holomorphically trivial).
If $\tX$ is a minimal resolution then (by the adjunction formulae) $Z_K\in \calS'$.
In particular, $Z_K-s_{[Z_K]}\in L_{\geq 0}$.

\begin{lemma}\label{lem:szk}  $p_g=0$  whenever $Z_K=s_{[Z_K]}$.
If $Z_K>s_{[Z_K]}$ then $p_g=h^1(\calO_{Z_K-s_{[Z_K]}})$.
More generally,
$h^1(\tX, \calL)=h^1(Z_K-s_{[Z_K]}, \calL)$  for any $\calL\in \pic(\tX)$ with $c_1(\calL)\in -\calS'$.
\end{lemma}
\begin{proof} By generalized Kodaira or Grauert--Riemenschneider type vanishing
$h^1(\widetilde{X}, \mathcal{O}_{\widetilde{X}}(-\lfloor Z_K\rfloor))=0$).
Hence, if $\lfloor Z_K\rfloor=0$ then $p_g=0$. Otherwise, using the exact sequence
$0\to \mathcal{O}_{\widetilde{X}}(-\lfloor Z_K\rfloor))\to \mathcal{O}_{\widetilde{X}}
\to \mathcal{O}_{\lfloor Z_K\rfloor}\to 0$ we get $h^1(\calO_{\lfloor Z_K\rfloor})=p_g$.
Next, consider the computation sequence from Lemma \ref{lem:cs2}
applied for $l'=r_{[Z_K]}$. By induction we prove that
$h^1(\calO_{Z_K-z_i})=p_g$. For $i=0$ we just verified it, since $Z_K-r_{[Z_K]}=\lfloor Z_K\rfloor$.
Then use the cohomological exact sequence associated with
$0\to \calO_{E_{v(i)}}(-Z_K+z_{i+1})\to \calO_{Z_K-z_i}\to \calO_{Z_K-z_{i+1}}\to 0$ and the vanishing
$h^1(\calO_{E_{v(i)}}(-Z_K+z_{i+1}))=0$.

More generally, $h^1(\tX,\calL)=h^1(Z_K-z_i,\calL)$ for any $i$ by similar argument.
\end{proof}

\bekezdes \label{bek:natline} {\bf Natural line bundles.} Let
 $\phi:(\tX,E)\to (X,o)$ be as above.
Consider the `exponential'  cohomology exact sequence (with $H^1(\tX, \calO_{\tX}^*)=\pic(\tX)$ and
$H^1(\tX, \calO_{\tX})=\pic^0(\tX)$)
\begin{equation}\label{eq:exp}
0\to {\rm Pic}^0(\tX)\longrightarrow
{\rm Pic}(\tX)\stackrel{c_1}{\longrightarrow}H^2(\tX,\Z)\to 0.
\end{equation}
Here $c_1(\calL)\in H^2(\tX,\Z)=L' $ is the first Chern class of $\calL\in \pic(\tX)$.
Then, see e.g. \cite{OkumaRat,trieste},
there exists a unique homomorphism (split)
 $s_1:L'\to {\rm Pic}(\tX)$  of $c_1$, that is  $c_1\circ s_1=id$, such that
$s_1$ restricted to $L$ is $l\mapsto \calO_{\tX}(l)$.
The line bundles $s_1(l')$ are called {\it natural line
bundles } of $\tX$. For several definitions of
them see  \cite{trieste}.
E.g., $\calL$ is natural if and only if one of its power has the form $\calO_{\tX}(l)$
for some {\it integral} cycle $l\in L$ supported on $E$.
In order to have a uniform notation we write $\calO_{\tX}(l')$ for $s_1(l')$ for any $l'\in L'$.

\subsection{The Abel map \cite{NNI}}\label{ss:AbelMap}
As above,  let $\pic(\tX)=H^1(\tX,\calO_{\tX}^*)$ be the group of isomorphic
classes of holomorphic line bundles on $\tX$.
The first Chern class map $c_1:\pic(\tX)\to L'$ is surjective; write
$\pic^{l'}(\tX)=c_1^{-1}(l')$. Since $H^1(M,\Q)=0$, $\pic^0(\tX)\simeq H^1(\tX,\calO_{\tX})\simeq \C^{p_g}$.

Similarly, if $Z\in L_{>0}$ is an effective non--zero integral cycle supported by $E$, then ${\rm Pic}(Z)=H^1(Z,\calO_Z^*)$ denotes   the group of isomorphism classes
of invertible sheaves on $Z$. Again, it appears in the exact sequence $
0\to {\rm Pic}^0(Z)\to {\rm Pic}(Z)\stackrel{c_1}
{\longrightarrow} L'(|Z|)\to 0$,
where ${\rm Pic}^0(Z)=H^1(Z,\calO_Z)$.
Here  $L(|Z|)$ denotes the sublattice of $L$ generated by
the  base element $E_v\subset |Z|$, and $L'(|Z|)$ is its dual.

Though for any effective cycle $Z$ the Abel map might have its own peculiar properties, in this manuscript we always assume that all the $E_v$--coefficients of $Z$ are sufficiently large, denoted by $Z\gg 0$.
Under this assumption one has several stability properties, e.g.  $L'(|Z|)=L'$,
$\pic(Z)=\pic(\tX)$, or $h^1(Z,\calL)=h^1(\tX,\calL)$ for any $\calL\in \pic(\tX)$.

For any $Z\gg 0$ let
$\eca(Z)$  be the space of (analytic) effective Cartier divisors on 
$Z$. Their supports are zero--dimensional in $E$.
Taking the line bundle  of a Cartier divisor provides  the {\it Abel map}
$c=c(Z):\eca(Z)\to \pic(Z)$.
Let
$\eca^{l'}(Z)$ be the set of effective Cartier divisors with
Chern class $l' \in L'$, i.e.
$\eca^{l' }(Z):=c^{-1}(\pic^{l'}(Z))$.
The restriction of $c$ is denoted by   $c^{l'}:\eca^{l'}(Z)\to \pic^{l'}(Z)$.


A line bundle $\calL\in \pic^{l'}(Z)$ is in the image
 ${\rm im}(c^{l'})$ if and only if it has a section without fixed components, that is,
 if $H^0(Z,\calL)_{reg}\not=\emptyset $, where
$H^0(Z,\calL)_{reg}:=H^0(Z,\calL)\setminus \cup_v H^0(Z-E_v, \calL(-E_v))$.
By this definition (see (3.1.5) of \cite{NNI}) $\eca^{l'}(Z)\not=\emptyset$ if and only if
$-l'\in \calS'\setminus \{0\}$. It is advantageous to have a similar statement for
$l'=0$ too, hence we redefine  $\eca^0(Z)$ as $\{\emptyset\}$, a set/space with one element
(the empty divisor), and $c^0:\eca ^0(Z)\to \pic^0(Z)$ by $c^0(\emptyset)=\calO_Z$.
In particular,
\begin{equation}\label{eq:Chernzero1}
H^0(Z,\calL)_{reg}\not=\emptyset\ \Leftrightarrow\ \calL=\calO_Z\
\Leftrightarrow\ \calL\in {\rm im}(c^0) \ \ \mbox{ whenever $c_1(\calL)=0$}.\end{equation}
Hence, the previous equivalence extends to this $l'=0$ case too:
\begin{equation}\label{eq:Chernzero}
 \eca^{l'}(Z)\not=\emptyset\ \Leftrightarrow \ -l'\in \calS'.
\end{equation}
Sometimes (e.g. in section \ref{saturated}) even for $\calL\in\pic^{l'}(\tX)$ we write
$\calL\in \im(c^{l'})$ whenever $\calL|_Z\in  \im(c^{l'}(Z))$ for some $Z\gg0$. This fact is equivalent
with the fact that
$\calL\in \pic(\tX) $ has no fixed components.

It turns out that
$\eca^{l'}(Z)$ is a smooth complex algebraic variety  of dimension $(l',Z)$ and
the Abel map  is an algebraic regular map. For more properties
and applications see \cite{NNI,NNII}.

\bekezdes\label{bek:modAbel} {\bf The modified Abel map.}  For any $Z\gg 0$ let $\calO_Z(l')$ be the
restriction of the natural line bundle $\calO_{\tX}(l')$ to $Z$. (In fact, $\calO_Z(l')$
can be defined in an identical way as $\calO_{\tX}(l')$ starting from
the exponential cohomological sequence
$0\to \pic^0(Z)\to  \pic(Z)\to H^2(\tX,\Z)\to 0$ as well.) Multiplication by $\calO_{Z}(-l')$  gives
an  isomorphism of the affine spaces $\pic^{l'}(Z)\to \pic^0(Z)$. Furthermore, we identify
(via the exponential exact sequence) $\pic^0(Z)$ with the vector space
$H^1(Z, \calO_Z)=H^1(\tX,\calO_{\tX})$.

It is convenient to replace the Abel map $c^{l'}$ with the composition
$$\widetilde{c}^{l'}:\eca^{l'}(Z)\stackrel{c^{l'}}{\longrightarrow} \pic^{l'}(Z)\stackrel{\calO_Z(-l')}
{\longrightarrow} \pic^0(Z)\stackrel{\simeq}{\longrightarrow} H^1(\calO_Z).$$
The advantage of this new set of maps is that all the images sit in the same
vector  space $H^1(\calO_Z)$.

\bekezdes \label{bek:addDiv} {\bf The monoid structure of divisors and of the modified Abel map.}
Consider the natural additive structure
$s^{l'_1,l'_2}(Z):\eca^{l'_1}(Z)\times \eca^{l'_2}(Z)\to \eca^{l'_1+l'_2}(Z)$ ($l'_1,l'_2\in-\calS'$)
provided by the sum of the divisors. One verifies (see e.g. \cite[Lemma 6.1.1]{NNI}) that
$s^{l'_1,l'_2}(Z)$ is dominant and quasi--finite. 
There is a parallel multiplication $\pic^{l'_1}(Z)\times \pic^{l'_2}(Z)\to \pic^{l'_1+l'_2}(Z)$,
$(\calL_1,\calL_2)\mapsto \calL_1\otimes\calL_2$, which satisfies
 $c^{l_1'+l'_2}\circ s^{l'_1,l'_2}= c^{l'_1} \otimes c^{l'_2}$ in $\pic^{l'_1+l'_2}$.
This, in the modified case, using
$\calO_Z(l'_1+l'_2)=\calO_Z(l'_1)\otimes \calO_Z(l'_2)$,
 reads as
 $\widetilde{c}^{l_1'+l'_2}\circ s^{l'_1,l'_2}= \widetilde{c}^{l'_1}
+ \widetilde{c}^{l'_2}$ in $H^1(\calO_Z)$. The above properties   imply
\begin{equation}\label{eq:addclose}
\im (\widetilde{c}^{l_1'})+ \im (\widetilde{c}^{l_2'})\subset
\im (\widetilde{c}^{l_1'+l_2'})\subset \overline{
\im (\widetilde{c}^{l_1'})+ \im (\widetilde{c}^{l_2'})}.
\end{equation}
\begin{definition}\label{def:VZ}
For any  $l'\in-\calS'$  let $A_Z(l')$ --- or just  $A(l')$ ---
be the smallest dimensional affine subspace of
$H^1(\calO_Z)$ which contains $\im (\widetilde{c}^{l'})$. Let $V_Z(l')$, or $V(l')$,
 be the parallel vector subspace
of $H^1(\calO_Z)$,
the translation of  $A_Z(l')$ to the origin. 

For any $I\subset \calv$, $I\not=\emptyset$,
let $(X_I,o_I)$ be the multigerm $\tX/_{\cup_{v \in I}E_v}$
at its  singular points,
obtained by contracting the connected components of
$\cup_{v \in I}E_v$ in $\tX$. If $I=\emptyset$ then by convention $(X_I,o_I)$ is a smooth germ.
\end{definition}

\begin{theorem}\label{prop:AZ} \ \cite[Prop. 5.6.1, Lemma 6.1.6 and Th. 6.1.9]{NNI}
Assume that  $Z\gg 0$.

(a)   For any $-l'=\sum_va_vE^*_v\in \calS'$ let the $E^*$--support of $l'$ be
$I(l'):=\{v\,:\, a_v\not=0\}$. Then  $V(l')$ depends only on $I(l')$.
(This motivates to write $V(l')$ as $V(I)$ where $I=I(l')$.)

(b)  $V(I_1\cup I_2)=V(I_1)+V(I_2)$ and $A(l_1'+l_2')= A(l_1')+A(l_2')$.

(c)  $\dim V(I)=h^1(\calO_Z)-h^1(\calO_{Z|_{\calv\setminus I}})=
 p_g(X,o)-p_g(X_{\calv\setminus I},o_{\calv\setminus I})$.

(d)  If $\calL^{im}_{gen}$ is a generic bundle of \,$\im (c^{l'})$
then $h^1(Z,\calL^{im}_{gen})=p_g(X,o)-\dim(\im (c^{l'}))$.

(e)  For $n\gg 1$ one has $\im(\widetilde{c}^{nl'})=A(nl')$, and $h^1(Z,\calL)=
p_g(X_{\calv\setminus I(l')},o_{\calv\setminus I(l')})$    
for any $\calL\in \im (c^{nl'})$.

\end{theorem}

 For different geometric reinterpretations of $\dim V_Z(I)$ see  also \cite[\S 9]{NNI}.

\bekezdes \label{bek:diff forms} {\bf The linear subspace arrangement $\{V_Z(I)\}_I\subset
\C^{p_g}$ and differential forms.}
 The arrangement $\{V(I)\}_I$ transforms into a
  linear subspace arrangement of $H^0(\Omega^2_{\tX}(Z))/H^0(\Omega^2_{\tX})$ via the (Laufer)
  non--degenerate pairing $H^1(\calO_Z)\otimes   H^0(\Omega^2_{\tX}(Z))/H^0(\Omega^2_{\tX})\to\C$
  (cf. \cite[7.3]{NNI}) as follows. Let $\Omega(I)$ be the subspace
  $H^0(\Omega^2_{\tX}(Z|_{\calv\setminus I}))/H^0(\Omega^2_{\tX})$ in
  $H^0(\Omega^2_{\tX}(Z))/H^0(\Omega^2_{\tX})$ (that is, the subspace generated by those forms which
  have no poles along generic points of any $E_v$, $v\in I$).
  \begin{proposition}\label{prop:HDIZ} \ \cite[8.3]{NNI} Via Laufer duality
$V(I)= \Omega(I)^*.$
\end{proposition}

\bekezdes \label{bek:diff forms2} {\bf The $\dim \im (c^{l'})$  and differential forms.}
Next  we recall a statement from \cite[\S 10]{NNI}. For simplicity we will assume that
$l'=-E^*_v$ for some $v\in \calv$. This means that any  divisor $D\in \eca^{l'}(\tX)$ with Chern
class $l'$ is a transversal cut of $E_v$ at some point $p\in E_v\setminus \cup_{u\not=v}E_u$.
Let us fix some local coordinates $(u,v)$ in some neighbourhood $U$ of $p$ such that
$\{u=0\}=E_v\cap U$, while $D$ has local equation $v$.
Any local section of $\Omega^2_{\tX}(Z)$ ($Z\gg 0$ as above)
near $p$ has local form $\omega=\sum_{i\in\Z, j\in \Z_{\geq 0}} a_{i,j} u^i v^jdu\wedge dv$.
We define the residue ${\rm Res}_D(\omega)=(w/dv)|_{v=0}:= \sum_{i}a_{i,0}u^i du$.

\begin{proposition}\label{prop:Res} \ \cite[Corollary 10.1.2]{NNI}
 Assume that $\{\omega_1,\ldots, \omega_{p_g}\}$ are fixed representatives of  a
 basis of $H^0(\tX,\Omega^2_{\tX}(Z))/H^0(\tX,\Omega^2 _{\tX})$ (where  $Z\gg 0$). Set
\begin{equation*}
{\mathcal H}:=\{(a_1,\ldots,a_{p_g})\in \bC^{p_g}: {\rm Res}_{D}(\textstyle{\sum}_\alpha a_{\alpha}\omega_{\alpha}) \ \mbox{has no pole along $D$}\}.
\end{equation*}
Then $h^1(Z,\calO_{Z}(D))=\dim ({\mathcal H})$,
and the number of independent relations between $(a_1,\ldots, a_{p_g})$, $p_g-h^1(Z,\calO_{Z}(D))$,
is the dimension of the image of the tangent map  $\im T_Dc^{l'}(T_D\eca^{l'}(Z))$.

In particular, $\dim (\im (c^{l'}(Z)))$ is the number of independent
relations for $D$ generic.
\end{proposition}
\begin{corollary}\label{cor:dim} In the situation of Proposition \ref{prop:Res} assume that
the forms $\{\omega_j\}_{j=d+1}^{p_g}$ have no poles along $E_v$, while the non--trivial poles of
$\{\omega_j\}_{j=1}^{d}$ along $E_v$ are all distinct. Then $\dim \im (c^{l'})=d$.
\end{corollary}
\begin{proof}
For generic $D$ (or, generic $v$), $\{{\rm Res}_D(\omega_j)\}_{j=1}^d$ have different non-trivial poles
at $u=0$.
\end{proof}

\section{Elliptic singularities. The elliptic sequence.}\label{s:elliptic}

\subsection{Elliptic singularities} \label{ss:ell}
 Let $Z_{min}\in L$ be the minimal (or fundamental) cycle of the resolution $\phi$, that is,
$\min\{\calS\setminus 0\}$ \cite{Artin62,Artin66}. Recall that $(X,o)$ is called elliptic if
$\chi(Z_{min})=0$, or equivalently, $\min _{l\in L_{>0}}\chi(l)=0$ \cite{Laufer77,Wa70}.
It is known that if we decrease the decorations (Euler numbers), or we take a full subgraph
 of an elliptic graph, then we get either an elliptic or a rational graph.

Let $C$ be the minimally elliptic cycle \cite{Laufer77,weakly}, that is, $\chi(C)=0$ and
$\chi(l)>0$ for any $0<l<C$. There is a unique cycle with this property, and if
$\chi(D)=0$ ($D\in L$) then necessarily $C\leq D$. In particular, $C\leq Z_{min}$.
In the sequel we assume that the resolution is minimal.  Then $Z_K\in \calS'\setminus 0$, hence in the numerically
Gorenstein case $Z_{min}\leq Z_K$ by the minimality of $Z_{min}$ in $\calS\setminus 0$.

The minimally elliptic singularities were introduced by Laufer in \cite{Laufer77}.  In a minimal
resolution they are characterized (topologically) by $Z_{min}=Z_K=C$.
Moreover, $(X,o)$ is minimally elliptic if and only if $p_g(X,o)=1$ and $(X,o)$
is Gorenstein.  For details see  \cite{Laufer77,weakly,Nfive}.

\subsection{Elliptic sequences}
One of the most important tools in the study of elliptic singularities are the {\it elliptic sequences}.
An elliptic sequence  constitute of a sequence of integral cycle associated with the topological type (graph).
They were introduced by Laufer and S. S.-T. Yau, for the definition in the general
(non numerically Gorenstein) case see \cite{Yau5,Yau1}. In the numerically Gorenstein case the construction is
simpler, see also \cite{weakly,Nfive,OkumaEll}. This second
case will be recalled below. In fact, we will use
an elliptic sequence even in the non numerically Gorenstein case, but not the `classical' one defined
by Laufer and Yau:  we define a new one, whose structure is much closer to the structure of sequences
associated with numerically Gorenstein graphs.  In fact, after the first step of the construction (which produces
a rational cycle)
we hit a  numerically Gorenstein support, and the continuation of the sequence
 is the one imposed by the numerically Gorenstein case.

\bekezdes {\bf The construction of the elliptic sequence; numerically Gorenstein case
\cite{Yau5,Yau1,weakly,Nfive,OkumaEll}.}
The elliptic sequence consists of a sequence of integral cycles
$\{Z_{B_j}\}_{j=0}^m$, where $Z_{B_j}$ is the minimal
 cycle supported on  the connected  reduced cycle $B_j$.  $\{B_j\}_{j=0}^m$ are defined inductively as follows.  For $j=0$ one takes $B_0=E$, hence $Z_{B_0}=Z_{min}$. Then $C\leq Z_{min}=Z_{B_0}\leq Z_K$.
 If $Z_{B_0}=Z_K$ then we stop, $m=0$, this situation corresponds to the minimally elliptic case.

 Otherwise one takes $B_1:=|Z_K-Z_{B_0}|$. One  verifies that   $|C|\subseteq B_1\varsubsetneq B_0$,
 $B_1$ is connected,
 and it supports a numerically Gorenstein elliptic
 topological type with canonical cycle $Z_K-Z_{B_0}$.
 (Furthermore, $(E_v,Z_{B_0})=0$ for any $E_v\subset B_1$. The proof of all these facts are
  similar to the proof of Lemma \ref{lem:b0} below.)
  In particular, $C\leq Z_{B_1}\leq
 Z_K-Z_{B_0}$. Then we repeat the inductive argument. If $Z_{B_1}=  Z_K-Z_{B_0}$, then we stop,
 $m=1$. Otherwise,  we define $B_2:=|Z_K-Z_{B_0}-Z_{B_1}|$. $B_2$ again is connected, $|C|\subseteq B_2
 \varsubsetneq B_1$, and supports a  numerically Gorenstein elliptic
 topological type with canonical cycle $Z_K-Z_{B_0}-Z_{B_1}$. After finite steps we get
 $Z_{B_m}=Z_K-Z_{B_0}-\cdots-Z_{B_{m-1}}$, hence the minimal cycle and the canonical cycle on
 $B_m$ coincide. This means that $B_m$ supports a minimally elliptic singularity with $Z_{B_m}=C$.

 We say that the length of the elliptic sequence $\{Z_{B_j}\}_{j=0}^m$ is $m+1$.

\bekezdes {\bf The  construction of the (new)
elliptic sequence; non--numerically Gorenstein case.}
Assume that $Z_K\not\in L$, that is, $r_{[Z_K]}\not=0$. Since the resolution is minimal,
$Z_K\in \calS'$, hence $Z_K\geq s_{[Z_K]}$. By Lemma \ref{lem:szk} $Z_K> s_{[Z_K]}$.
We will use the following notations: $B_{-1}:=E$, $Z_{B_{-1}}:=s_{[Z_K]}$ and $B_0:=|Z_K-s_{[Z_K]}|$.
(Note that $Z_{B_{-1}}\in L'\setminus L$.)
\begin{lemma}\label{lem:b0} (a) $\chi(s_{[Z_K]})=0$.

(b)  $B_0$ is connected, $C\subseteq B_0\varsubsetneq E$,  and
$(E_v,Z_{B_{-1}})=0$ for any $E_v\subset B_0$.

(c) $B_0$ supports a numerically Gorenstein
elliptic topological type with canonical cycle $Z_K-s_{[Z_K]}$.
\end{lemma}
\begin{proof}
{\it (a)--(b)}
Write $l:=Z_K-s_{[Z_K]}$. Then $\chi(s_{[Z_K]})=\chi( Z_K-l)=\chi(l)$.  Since $(X,o)$ is elliptic
$\chi(s_{[Z_K]})=\chi(l)\geq 0$ $(\dag)$. Also, $(s_{[Z_K]},l)\leq 0$
since $s_{[Z_K]}\in\calS'$ $(\ddag)$.
On the other hand, $0=\chi(Z_K)=\chi(l+s_{[Z_K]})=\chi(l)+\chi(s_{[Z_K]})-(l,s_{[Z_K]})$.
Then by $(\dag)$ and $(\ddag)$ the expressions from the right hand side are $\geq 0$, hence
necessarily $\chi(s_{[Z_K]})=\chi(l)=(l,s_{[Z_K]})=0$. If $l$ has more connected components, say
$\cup_il_i$, then $\chi(l_i)=0$ for all $i$, hence each $l_i$  contains/dominates  a minimally
elliptic cycle (cf. \cite{Laufer77}), a fact which contradicts  the uniqueness of the minimally
elliptic cycle. Hence $|l|=B_0$ is connected and $|C|\subset B_0$. Furthermore,
$(l, s_{[Z_K]})=0$ shows that $|l|\not=E$.

{\it (c)} $C\subseteq B_0\varsubsetneq E$ shows that $\min_{|l|\subset B_0, \, l>0}
 \chi(l)=0$, hence $B_0$ supports an elliptic topological type. Moreover, from   $(l, s_{[Z_K]})=0$
 we read that for any $E_v$ from the support of $l$ one has
$(E_v,s_{[Z_K]})=0$, hence $(E_v,Z_K-s_{[Z_K]})=(E_v,Z_K)$, hence $Z_K-s_{[Z_K]}\in L$ is
the canonical cycle on $B_0$.
\end{proof}
Then, as a continuation of the sequence, starting from $B_0$ and its integral canonical class $Z_K-s_{[Z_K]}$  we construct
the sequence $\{Z_{B_j}\}_{j=0}^m$ as in the numerically Gorenstein case.

We say that the elliptic sequence $\{Z_{B_j}\}_{j=-1}^m$ has length $m+1$ and `pre--term'
$Z_{B_{-1}}=s_{[Z_k]}\in L'$.

In order to have a uniform notation, in the numerically Gorenstein case we set
$Z_{B_{-1}}:=0$ (which, in fact, it is $s_{[Z_k]}$).
In both cases, from  above (see also \cite[2.11]{weakly}), for latter references,
\begin{equation}\label{eq:orthogonal}
(E_v,Z_{B_j})=0 \ \ \mbox{for any $E_v\subset B_{j+1}$} \ \ \ (-1\leq j< m).
\end{equation}

\begin{remark}
The construction of $\{B_j\}_j$  can be handled uniformly as follows.
For any connected support $B$ let $Z_K(B)$ be the canonical cycle associated
with the graph supported by $B$ and let $s^*(B)$ be the smallest nonzero element of
$\calS'(B)$ with $[s^*(B)]=[Z_K(B)]$ in $L'(B)/L(B)$. Then we proceed
inductively:  the first support is $E$, and
once $B_j$ is known then  one sets $B_{j+1}:=|Z_K(B_j)-s^*(B_j)|$.
We prefer to index them in such a way  that $B_0$ is the first numerically Gorenstein support.
\end{remark}
\begin{example}\label{ex:new}
Consider the next elliptic graph from the left

\begin{picture}(200,50)(-120,0)
\put(70,30){\circle*{4}}\put(90,30){\circle*{4}}\put(110,30){\circle*{4}}\put(130,30){\circle*{4}}
\put(50,30){\circle*{4}}\put(130,30){\circle*{4}}\put(150,30){\circle*{4}}
\put(-30,30){\circle*{4}}
\put(110,40){\makebox(0,0){\small{$-3$}}}\put(130,40){\makebox(0,0){\small{$-3$}}}
\put(150,20){\makebox(0,0){\small{$E_1$}}}
\put(130,20){\makebox(0,0){\small{$E_2$}}}
\put(-10,30){\circle*{4}}
\put(10,30){\circle*{4}}
\put(30,30){\circle*{4}}
\put(10,10){\circle*{4}}
\put(-30,30){\line(1,0){180}}\put(10,10){\line(0,1){20}}


\end{picture}

\noindent where the $(-2)$--vertices are  unmarked.
 $Z_K$ and $s_{[Z_K]}$ are

 \begin{picture}(500,45)(70,0)

\put(70,30){\makebox(0,0){\tiny{14/3}}}
\put(90,30){\makebox(0,0){\tiny{28/3}}}
\put(110,30){\makebox(0,0){\tiny{42/3}}}
\put(110,10){\makebox(0,0){\tiny{21/3}}}
\put(130,30){\makebox(0,0){\tiny{35/3}}}
\put(150,30){\makebox(0,0){\tiny{28/3}}}
\put(170,30){\makebox(0,0){\tiny{21/3}}}
\put(190,30){\makebox(0,0){\tiny{14/3}}}
\put(210,30){\makebox(0,0){\tiny{7/3}}}
\put(230,30){\makebox(0,0){\tiny{4/3}}}
\put(250,30){\makebox(0,0){\tiny{2/3}}}

\put(300,30){\makebox(0,0){\tiny{2/3}}}
\put(320,30){\makebox(0,0){\tiny{4/3}}}
\put(340,30){\makebox(0,0){\tiny{6/3}}}
\put(340,10){\makebox(0,0){\tiny{3/3}}}
\put(360,30){\makebox(0,0){\tiny{5/3}}}
\put(380,30){\makebox(0,0){\tiny{4/3}}}
\put(400,30){\makebox(0,0){\tiny{3/3}}}
\put(420,30){\makebox(0,0){\tiny{2/3}}}
\put(440,30){\makebox(0,0){\tiny{1/3}}}
\put(460,30){\makebox(0,0){\tiny{1/3}}}
\put(480,30){\makebox(0,0){\tiny{2/3}}}
\end{picture}

\noindent $B_0$ is obtained by deleting $E_1$ from $E$, while
$B_1$ by deleting $E_1$ and $E_2$.
The length is  $m+1=2$.
\end{example}
The elliptic sequence imposes some kind of `linearity' of the  structure of the graph. E.g.,
the following statement holds (probably parts of it already known in the literature).

\begin{lemma}\label{lem:glue}
Consider an elliptic graph $\Gamma$ with elliptic sequence supports
$B_{-1},B_0, \ldots , B_m$ and vertices $\calv$. Assume that we can glue to the graph a new vertex $v_{new}$
by an edge $(v,v_{new})$, $v\in \calv$, such that
 the new graph is still elliptic.  Then the $E_v$--multiplicity of the
 fundamental cycle $Z_{min}$ is 1 and  $v \not\in B_1$.
 Furthermore, if $\Gamma$ is numerically Gorenstein, then $v$ is necessarily an end--vertex, and the $E_v$--multiplicity of $Z_K$ is 1 too (and the multiplicity of the adjacent vertex is 2).
\end{lemma}
\begin{proof}
Suppose  that $v \in B_1$.  Then we have that the multiplicity of $E_v$ in $Z_{B_0}  + Z_{B_1}$ is at least $2$.
 But  $\chi(Z_{B_0}  + Z_{B_1}) = 0$, cf. (\ref{eq:orthogonal}), hence
$\chi(Z_{B_0}  + Z_{B_1}+E_w) < 0$, which contradicts the ellipticity of the large graph.
By Laufer's algorithm \cite{Laufer72} there exists a computation sequence of the fundamental cycle of the large graph such that one of its terms is $Z_{min}=Z_{min}(\Gamma)$ while the next one is $Z_{min}+E_{new}$. Since
$\chi(Z_{min})=\chi(Z_{min}+E_{new})=0$, we get that the coefficient $m_{E_v}(Z_{min})$ of $E_v$ in $Z_{min}$ is 1.
In the numerically Gorenstein case, since $v\not \in B_1$ we get that that  $m_{E_v}(Z_{K})=1$ too. By the adjunction formula then $v$ is either an end--vertex (as in the statement) or it has two neighbours both with multiplicity 1. But this last case would generate (by repeating the argument for the two neighbours)
an infinite string, all with multiplicity one, which cannot happen.
\end{proof}

\begin{example}\label{ex:nagygraf} The next graphs are both numerically Gorenstein, with  $m=3$.
The dash--boxes show the supports $B_3\subset B_2\subset B_1\subset B_0$.
In both cases
 $B_1\setminus B_2 $ has more   connected components. However,
in the first case the two components
are  adjacent with different vertices of $B_1$, while in the second case pairs of components
are adjacent with the same vertex of $B_1$.
These adjacent properties will be crucial in Theorem \ref{th:WECCCrit}.
 In both examples we create nodes in the zones $B_i\setminus B_{i+1}$.
Both graphs can be continued as  (infinite) series of numerically
Gorenstein graphs by adding pairs of vertices to each
string (similarly as their last extension).

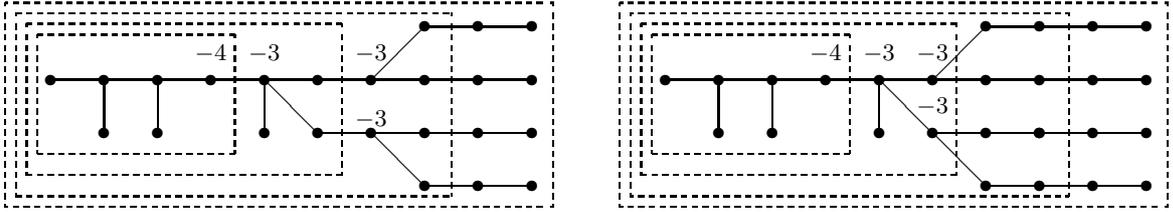
\begin{picture}(200,90)(-35,-20)
\put(-10,10){\circle*{4}}
\put(30,40){\makebox(0,0){\small{$-4$}}}
\put(50,40){\makebox(0,0){\small{$-3$}}}
\put(70,30){\circle*{4}}\put(90,30){\circle*{4}}\put(110,30){\circle*{4}}
\put(50,30){\circle*{4}}\put(110,50){\circle*{4}}\put(110,-10){\circle*{4}}
\put(-30,30){\circle*{4}}
\put(50,10){\circle*{4}}\put(70,10){\circle*{4}}\put(90,10){\circle*{4}}\put(110,10){\circle*{4}}
\put(90,40){\makebox(0,0){\small{$-3$}}}
\put(90,15){\makebox(0,0){\small{$-3$}}}
\put(-10,30){\circle*{4}}
\put(10,30){\circle*{4}}
\put(30,30){\circle*{4}}
\put(10,10){\circle*{4}}
\put(-30,30){\line(1,0){140}}\put(10,10){\line(0,1){20}}\put(-10,10){\line(0,1){20}}
\put(50,10){\line(0,1){20}}\put(50,30){\line(1,-1){20}}
\put(70,10){\line(1,0){40}}\put(90,10){\line(1,-1){20}}
\put(90,30){\line(1,1){20}}
\put(-35,2){\dashbox{2}(74,45)}
\put(-39,-6){\dashbox{2}(118,57)}
\put(-43,-14){\dashbox{2}(163,69)}\put(-47,-18){\dashbox{2}(205,77)}

\put(130,50){\circle*{4}}\put(130,30){\circle*{4}}\put(130,10){\circle*{4}}
\put(130,-10){\circle*{4}}
\put(150,50){\circle*{4}}\put(150,30){\circle*{4}}\put(150,10){\circle*{4}}
\put(150,-10){\circle*{4}}
\put(110,50){\line(1,0){40}}
\put(110,30){\line(1,0){40}}
\put(110,10){\line(1,0){40}}
\put(110,-10){\line(1,0){40}}

\put(220,10){\circle*{4}}\put(240,10){\circle*{4}}
\put(260,40){\makebox(0,0){\small{$-4$}}}
\put(280,40){\makebox(0,0){\small{$-3$}}}
\put(300,40){\makebox(0,0){\small{$-3$}}}
\put(300,20){\makebox(0,0){\small{$-3$}}}
\put(300,30){\circle*{4}}\put(320,30){\circle*{4}}\put(340,30){\circle*{4}}
\put(280,30){\circle*{4}}
\put(320,50){\circle*{4}}\put(340,50){\circle*{4}}
\put(200,30){\circle*{4}}
\put(280,10){\circle*{4}}\put(300,10){\circle*{4}}\put(320,10){\circle*{4}}\put(340,10){\circle*{4}}
\put(220,30){\circle*{4}}
\put(240,30){\circle*{4}}
\put(260,30){\circle*{4}}
\put(320,-10){\circle*{4}}\put(340,-10){\circle*{4}}
\put(340,10){\circle*{4}}
\put(200,30){\line(1,0){140}}\put(240,10){\line(0,1){20}}\put(220,10){\line(0,1){20}}
\put(280,10){\line(0,1){20}}\put(280,30){\line(1,-1){40}}
\put(300,10){\line(1,0){40}}
\put(300,30){\line(1,1){20}}\put(320,50){\line(1,0){20}}
\put(320,-10){\line(1,0){20}}
\put(195,2){\dashbox{2}(74,45)}
\put(191,-6){\dashbox{2}(118,57)}
\put(187,-14){\dashbox{2}(164,69)}\put(183,-18){\dashbox{2}(205,77)}

\put(360,50){\circle*{4}}\put(360,30){\circle*{4}}\put(360,10){\circle*{4}}
\put(360,-10){\circle*{4}}
\put(380,50){\circle*{4}}\put(380,30){\circle*{4}}\put(380,10){\circle*{4}}
\put(380,-10){\circle*{4}}
\put(340,50){\line(1,0){40}}
\put(340,30){\line(1,0){40}}
\put(340,10){\line(1,0){40}}
\put(340,-10){\line(1,0){40}}

\end{picture}

%

\end{example}

\subsection{The cycles $C_t$ and $C_t'$}\label{bek:Ct} Set
$C_t:=\sum_{i=-1}^t Z_{B_i}$ and $C_t':=\sum_{i=t}^m Z_{B_i}$, $-1\leq t\leq m$.
E.g. $C_m=Z_K$ and, in general,  $C'_j$ is the canonical cycle  of $B_j$.
Furthermore, $\chi(Z_{B_j})=\chi(C_j)=\chi(C_j')=0$.

The next Lemma generalizes   \cite[Lemma 2.13]{weakly} valid in the numerically Gorenstein  case.

\begin{lemma}\label{lem:Ci}
 Assume that $l' \in \calS'$,  $[l'] = [Z_K]$ and
 $l' \leq Z_K$.   Then $l'\in \{C_{-1}, C_0, \cdots,  C_m \}$.
\end{lemma}
\begin{proof}
Since  $0\leq l'-Z_{B_{-1}}\leq Z_K-Z_{B_{-1}}=Z_K(B_0)$ and by (\ref{eq:orthogonal})
$l'-Z_{B_{-1}}\in \calS(B_0)$, the statement reduces to the numerically Gorenstein case.
(Alternatively, using the analogue of the
previous line as inductive step one can proceed also by induction on $m$.
The first step is as follows. If $l'-Z_{B_{-1}}\not=0$ then $l'-Z_{B_{-1}}\in \calS(B_0)
\setminus \{0\}$, hence $l'-Z_{B_{-1}}\geq Z_{B_{0}}$. Then
$0\leq l'-Z_{B_{-1}}-Z_{B_{0}}\leq Z_K-Z_{B_{-1}}-Z_{B_{0}}=C'_1$, that is,
$l'-Z_{B_{-1}}-Z_{B_{0}}$ is supported on $B_1 $ and it belongs to $\calS(B_1)$.
Then the induction runs.)
\end{proof}
\begin{remark}\label{rem:univ} The cycles $\{C_i\}_{i=0}^m$, or their supports
$\{B_i\}_{i=0}^m$, satisfy several other universal properties as well.
E.g., assume that the graph is numerically Gorenstein, and let $I\subset
\calv$, $I\not=\emptyset$, such that $I$ supports a numerically Gorenstein (connected) subgraph.
Then $I$ is one of the supports $\{B_i\}_{i=0}^m$.

Indeed, suppose, that $I \neq B_0$. Then, by induction, it is enough to prove $I \subset B_1$.
Let the canonical cycle on $I$ be $Z \in L$. Then   $(Z_K- Z , E_v) = 0$ for all $E_v \subset |Z|$.
Else,  if  $E_v \not\subset  |Z|$,  we have $(Z_K, E_v) \leq 0$ and $(Z, E_v) \geq 0$,
so $(Z_K - Z, E_v) \leq 0$.
This means, that $Z_K - Z \in \calS'$, $Z_K - Z > 0$, $Z_K - Z \in L$.
These imply that  $Z_K - Z \geq Z_{min}$, hence  $Z \leq C'_1$ and $I=|Z|\subset B_1$.

Assume next that the graph is not numerically  Gorenstein.
Then we claim that the support of  any   numerically Gorenstein   (connected) subgraph  belongs again to
 $\{B_i\}_{i=0}^m$. First we show that the largest numerically Gorenstein subgraph is supported by $B_0$.
 (Then the rest follows from the previous paragraph.) Indeed, if $I$ is its support and $Z$ is the canonical cycle on this support, then similarly as above, $Z_K-Z\in \calS'$, hence
 $Z_K-Z\geq s_{[Z_K]}$. This reads as $Z\leq Z_K-s_{[Z_K]}$, or $|Z|\subset B_0$.
\end{remark}

\begin{remark}
Even if the graph is numerically Gorenstein, the list of antinef cycles $l\in \calS$ with
$l\not \geq Z_K$ is much larger than the list given in Lemma \ref{lem:Ci}. Indeed, take e.g
$l=2Z_{min}$, which usually is $\not\leq Z_K$ and $\not\geq  Z_K$.
\end{remark}

\bekezdes\label{bek:xtildej}
Let $\tX_j$ be a small 
\marginpar{{\bf ??}}
neighbourhood of
 $\cup_{E_v\subset B_j}E_v$ in $\tX$, and consider the singularities
$(X_j,o_j):=(\tX_j/B_j,o_j)$ obtained by contraction of $B_j$, $-1\leq j\leq m$. In particular,
$(X_{-1},o_{-1})$ is defined in the non--numerically Gorenstein case, it is  $(X,o)$, and
$(X_0,o_0)$ is the first numerically Gorenstein germ in the sequence.
The last one, $(X_m,o_m)$,  is a minimally elliptic (hence automatically Gorenstein).

Before we recall the next characterization of Gorenstein elliptic singularities we
mention that any numerically Gorenstein topological type admits Gorenstein structure
\cite{PPP}.
(But the generic analytic structure is Gorenstein only in the Klein and minimally elliptic case \cite[Th. 4.3]{Laufer77},  see also  \cite[Prop. 5.9.1]{NNII}.)

We recall the following facts from \cite[Statements 2.10, 3.5 and 4.11]{weakly}.

\begin{theorem}\label{e220}
Assume that $(X,o)$ is a  numerically Gorenstein elliptic singularity. Then
the following facts are equivalent:

(a)\ $p_g=m+1$;

(b)\   $ h^1(\calO_{C'_j})=m-j+1$, $h^1(\calO_{C_j})= j+1$
and $h^1(\tX,\calO(-C_j))=m-j$ \  for all \ $0\leq j\leq m$;



(c)\ For any $0\leq j\leq m-1$, there exists $f_j\in H^0(\tilde{X},\calO(-C_j))$, such
that for any $E_v\subset  B_{j+1}$ the vanishing order of $f_j$ on $E_v$ is
exactly 
the multiplicity of $C_j$ at $E_v$;


(d)\ The line bundles $\calO_{C'_{j+1}}(-C_j)$ are trivial for \ $0\leq j\leq m-1$;

(d')\ The line bundles $\calO_{C'_{j+1}}(-Z_{B_j})$ are trivial for \ $0\leq j\leq m-1$.

(e) The singularities $(X_j,o_j)$ are Gorenstein for all \ $0\leq j\leq m-1$;

(f)  The singularity $(X,o)$ is Gorenstein.
\end{theorem}

The implication {\it (f)$\Rightarrow$(e)} from above says that if the top singularity is
 Gorenstein then all the others (automatically
 with smaller support)  are necessarily Gorenstein. This fact
 applied for a fixed $(X_j,o_j)$ says that if one of the singularities $(X_j,o_j)$
 is Gorenstein, then all the others $\{(X_i,o_i)\}_{i>j}$ with smaller support
 are Gorenstein too. In fact, one has the following statement of Okuma. Set
\begin{equation}\label{eq:alpha}
{\mathcal A}_{gor}:=\{\, j\,|\, 0\leq j\leq m, \ (X_j,o_j)\ \mbox{is Gorenstein}\,\}
\ \ \mbox{and} \ \ \alpha:=\min\{{\mathcal A}_{gor}\}.
\end{equation}
Then by the above discussion ${\mathcal A}_{gor}=\{j\,|\, \alpha\leq j\leq m\}$ and the following facts
hold as well.
\begin{theorem}\label{th:Okuma}\ \cite[Corollary 2.15]{OkumaEll}
$p_g(X,o)=p_g(X_\alpha,o_\alpha)=\#{\mathcal A}_{gor}=m+1-\alpha$.
Furthermore,
the line bundles $\calO_{C'_{j+1}}(-C_j)$ are trivial for \ $\alpha \leq j\leq m-1$.
\end{theorem}

\bekezdes\label{bek:discussion} {\bf Discussion.} Assume that $(X,o)$ is Gorenstein, and
let us consider the function $f_j$ from Teorem \ref{e220}{\it (c)}.
Let ${\rm  div}_E(f_j)$ be the part of its divisor supported on $E$. Write $ {\rm  div}_E(f_j)$
as $C_j+x_j$. Then the support of $x_j$ contains no $E_v$ from $B_{j+1}$. Therefore, for such an $E_v$
one has $(E_v,C_j+x_j)\leq 0$ (since $ {\rm  div}_E(f_j)\in \calS$), $(E_v,C_j)=0$
(by (\ref{eq:orthogonal})) and $(E_v,x_j)\geq 0$ (by the above support condition), hence
necessarily $(E_v,x_j)=0$. Therefore, in the support of $x_j$ there is no $E_v$, which intersects
$B_{j+1}$ nontrivially.

\begin{remark}\label{bek:discussion2}
The support condition of functions $f_j$ can be improved slightly more, but not too much.
Indeed,  in the Gorenstein case $Z_{max}=Z_{min}$ \cite[\S 5]{weakly}, that is, $\calO_{\tX}(-C_0)$
has no fixed components (so, $f_0$ can be chosen such that $x_0=0$).
However, in general,  a similar choice for $f_j$ (with $x_j=0$) is not possible.
See e.g. the elliptic singularity $\{x^2+y^3+z^{6m+7}=0\}$.
Nevertheless, using Theorem \ref{th:fixed}, if $C^2\not=-1$ then there exists a function
$F$ with ${\rm  div}_E(F)=Z_K$, hence a general combination of $f_j$ and $F$ has the property that
${\rm  div}_E(f_j+\alpha F)=C_j$.
(In general, the maximum what
one can get via inductive steps using Gorenstein property is that
 the vanishing order of $f_j$ on $E_v$ is exactly
the multiplicity of $C_j$ at $E_v$ for  any $E_v\subset  B_{j}$. See again $\{x^2+y^3+z^{6m+7}=0\}$.)
\end{remark}
%
%
%
%
%


\subsection{The space
 $H^0(\tX,\Omega^2_{\tX}(Z))/H^0(\tX,\Omega^2_{\tX})$ for elliptic singularities}\label{bek:diff}

 Assume that $(X,o)$ is Gorenstein. Then each $(X_j,o_j)$ is Gorenstein, and,  in fact,  their
 Gorenstein forms are related. Indeed, let  $\omega_0\in H^0( \Omega^2_{\tX}(Z_K))$
 be the  Gorenstein form of $(X,o)$ (that is, the section which trivializes
  $\Omega^2_{\tX\setminus E}$) and  consider the
  function $f_j\in H^0(\tX, \calO(-C_j))$ given by  Theorem \ref{e220}{\it (c)}, $0\leq j\leq m-1$.
Then $\omega_{j+1}:=f_j\omega_0 $ has pole $C'_{j+1}$, and
(by the discussion from \ref{bek:discussion})
its restriction to a small neighbourhood of
$\cup_{E_v\subset B_{j+1}}E_v$ is a Gorenstein form of $(X_{j+1},o_{j+1})$. Furthermore, the classes of
$\{\omega_j\}_{j=0}^m$  generate $H^0(\Omega^2_{\tX}(Z))/H^0(\Omega^2_{\tX})$.

Next, assume an arbitrary numerically Gorenstein singularity with $\alpha=\min \{{\mathcal A}_{gor}\}$
as in Theorem \ref{th:Okuma}. Then the statement of the previous paragraph can be applied for
$(X_\alpha,o_\alpha)$. In this way we get forms $\{\omega_j'\}_{j=\alpha}^m$ in
$H^0(\tX_{\alpha}, \Omega^2_{\tX_{\alpha}}(C'_\alpha))$, whose classes in
$H^0(\tX_{\alpha}, \Omega^2_{\tX_{\alpha}}(C'_\alpha))/H^0(\tX_{\alpha}, \Omega^2_{\tX_{\alpha}})$
generate this vector space of dimension $p_g(\tX_\alpha)=m+1-\alpha$.

We claim that these forms (more precisely, some representatives of their classes
modulo $H^0(\tX,\Omega^2_{\tX})$)
can be extended to forms $\{\omega_j\}_{j=\alpha}^m$ in
$H^0(\tX, \Omega^2_{\tX}(C'_\alpha))$, such that their classes generate the vector space
 $H^0(\tX,\Omega^2_{\tX}(C'_0))/H^0(\tX,\Omega^2_{\tX})$ of dimension $p_g(X,o)=p_g(\tX_\alpha)$.
The pole of $\omega_j$ is  $C_j'$ for each $j$.

Indeed, set $I:=\calv\setminus B_\alpha$. Then, by Theorem \ref{th:Okuma},
$p_g(X,o)=p_g(X_\alpha,o_\alpha)$, hence part {\it (c)} of Theorem \ref{prop:AZ} reads as
$V(I)=0$. But this via Proposition \ref{prop:HDIZ} implies that $\Omega(I)$ is the total space
$H^0(\tX,\Omega^2_{\tX}(C'_0))/H^0(\tX,\Omega^2_{\tX})$. On the other hand, from
$\Omega(I)$ there is a well--defined restriction to
$H^0(\tX_{\alpha}, \Omega^2_{\tX_{\alpha}}(C'_\alpha))/H^0(\tX_{\alpha}, \Omega^2_{\tX_{\alpha}})$,
which is a priori injective; but since the two dimensions agree, it is necessarily bijective.

\begin{corollary}\label{cor:omegaflag}
  The linear subspace arrangement $\{\Omega_I\}_{I\subset \calv}$ in
   $H^0(\Omega^2_{\tX}(Z))/H^0(\Omega^2_{\tX})\simeq\C^{p_g }$
  reduces to the   flag  consisting of subspaces (where by $\omega$'s
   we denote their classes as well):
$$  0\subset  \C\langle \omega_m\rangle\subset \cdots
\subset \C\langle \omega_{\alpha+1},  \ldots,  \omega_m\rangle
\subset \C\langle
\omega_\alpha, \ldots, \omega_m\rangle =
\C^{m+1-\alpha}.$$
In fact, for any $I$ the subspace $\Omega_I$ is determined uniquely
 by $i_I:=\min\{i\,|\, I\cap B_i=\emptyset\}=
 \max\{i\,|\, I\cap B_{i-1}\not=\emptyset\} $ as
  $\Omega_I=\C\langle \omega_{\max\{\alpha, i_I\}}, \ldots , \omega_m\rangle$.
\end{corollary}

\section{Line bundles on $\tX$. Preliminary cohomological statements}\label{ss:lineb} Fix an elliptic singularity and its minimal
resolution $\tX$ as in the previous section.

\subsection{Cohomology of the line bundles}\label{ss:cohlineb}

Assume first that $(X,o)$ is numerically Gorenstein and fix some $j\in\{0,\ldots, m+1\}$.

\begin{lemma}\label{lem:INEQj}  Let $l$ be the cycle of fixed components of some
$\calL\in \pic^0(\tX)$.

(a) If  $l\geq C_{j-1}$
then $h^1(\calL)\leq p_g(\tX_{j})$. (This for $j=0$ reads as follows: if $l\geq 0$ then $h^1(\calL)
\leq p_g(\tX)$, while for $j=m+1$ says that  if $l\geq Z_K$ then $h^1(\calL)=0$.)

(b) Assume that  $\min\{l,Z_K\}=C_{j-1}$ (cf. Lemma \ref{lem:Ci}).
Then $\calL(-C_{j-1})|_{C_j'}\in \pic(C_j')$ is trivial and
$h^1(\tX,\calL)=h^1(\tX,\calL(-C_{j-1}))=h^1(C_j',\calL(-C_{j-1})|_{C_j'})=p_g(\tX_j)$.
Furthermore, if  $(X,o)$ is Gorenstein, then
$\calL|_{C_j'}\in \pic(C_j')$ is trivial too.
\end{lemma}
\begin{proof} {\it (a)}
If $l=0$ then a section  trivializes $\calL$ and
 $\calL=\calO_{\tX}$. Hence $h^1(\calL)=p_g(\tX)$. Otherwise, since $l\in\calS$, $l\geq Z_{min}$.
If $l\geq Z_K$ then in the cohomological exact sequence of $0\to \calL(-Z_K)\to \calL\to \calL|_{Z_K}\to 0$ we have
$H^0(\calL(-Z_K))=H^0(\calL)$ and $h^1(\calL(-Z_K))=0$, hence $h^1(\calL)=\chi(\calL|_{Z_K})=0$.
Hence, the statement holds for $m=0$. Then we proceed by induction.

Since $l\geq Z_{min}$ we can assume $j\geq 1$.
In the cohomology exact sequence of $0\to \calL(-C_{j-1})\to \calL\to \calL|_{C_{j-1}}\to 0$
we have $H^0(\calL(-C_{j-1}))=H^0(\calL(-l))=H^0(\calL)$, hence
$\chi(\calL|_{C_{j-1}})-h^1(\calL(-C_{j-1}))+h^1(\calL)=0$. Since $\chi(\calL|_{C_{j-1}})=0$ we get
$h^1(\calL(-C_{j-1}))=h^1(\calL)$. Next, from
$0\to \calL(-Z_K)\to \calL(-C_{j-1})\to \calL(-C_{j-1})|_{C' _{j}}\to 0$
and vanishing $h^1(\calL(-Z_K))=0$,  we also have
 $h^1(\calL(-C_{j-1}))=h^1(\calL(-C_{j-1})|_{C'_{j}})$. Since
 $\calL(-C_{j-1})|_{C'_{j}}\in \pic^0(C'_{j})$, by induction,
 $h^1(\calL(-C_{j-1})|_{C'_{j}})\leq p_g(\tX_{j})$.

{\it (b)} By the proof of {\it (a)} we have that
$h^1(\calL)=h^1(\calL(-C_{j-1}))=h^1(\calL(-C_{j-1})|_{C_j'})$.

Write $l$ as $C_{j-1}+x$. Then $|x|$ contains no $E_v$ from $B_j$. Furthermore,
for such $E_v$, $(E_v,l)\leq 0$, $(E_v,C_{j-1})=0$, $(E_v,x)\geq 0$ (as in \ref{bek:discussion}),
hence necessarily $(E_v,x)=0$, that is $B_j\cap|x|=\emptyset$. This shows that
$\calL(-C_{j-1})|_{C_j'}$ is trivialized by the restriction of the generic section of
$\calL(-C_{j-1})$.

In the Gorenstein case use the fact that
$\calO(-C_{j-1})|_{C_j'}$ is trivial (cf. Theorem \ref{e220}).
\end{proof}

Fix $\calL\in\pic(\tX)$ such that $c_1(\calL)\in -\calS'$.
Recall that by Lemma \ref{lem:szk} the computation of $h^1(\tX,\calL)$ reduces to the numerically Gorenstein case: $h^1(\calL)=h^1(\calL|_{Z_K-s_{[Z_K]}})$.

\begin{theorem}\label{prop:vanishing}
Let $I$ be the $E^*$--support of $c_1(\calL)$  and assume
 that $I\cap B_{j-1}\not=\emptyset$ for some $j>0$.  Then

 (a)
$h^1(\tX,\calL)=h^1(C'_j,\calL|_{C'_j})=h^1(\tX_j,\calL|_{\tX_j})$.
(For $j-1=m$ this reads as follows: if $I\cap B_m\not=\emptyset$ then $h^1(\tX,\calL)=0$.)

(b)  $h^1(\tX,\calL)\leq p_g(\tX_j)$.
\end{theorem}

\begin{proof}  Lemma \ref{lem:szk} reduces the statements to
the numerically Gorenstein case.

{\it (a)}
Similar reductions were used in \cite{Laufer77,weakly,Nfive}.
For the convenience of the reader  we provide the details.
By  Lemma  \ref{lem:szk} the second equality follows, and also
$h^1(\tX, \calL)=h^1(C'_0,\calL)$ (since $C'_0=Z_K-s_{[Z_K]}$). Hence we need to
show $h^1(C'_0,\calL)=h^1(C'_j,\calL)$.

Chose $u\in I\cap B_{j-1}$.
We construct a computation sequence which connects 0 with $\sum_{k=0}^{j-1} Z_{B_k}$.
This is
 a sequence of cycles $\{z_i\}_{i=0}^t $ with $z_0=0$ and $z_t=\sum_{k=0}^{j-1} Z_{B_k}$, such that
$z_{i+1}=z_i+E_{v(i)}$, where $v(i)\in\calv$  is conveniently chosen. We construct the
 sequence as concatenated  of several ones, each one being the (Laufer)
 computation sequence of  a minimal cycle (cf. \cite[5.8]{Nfive}).
Indeed, for each $0\leq k\leq j-1$ let $\{z_{k,i}\}_i $ be a computation sequence starting with
$z_{k,0}=E_u$ and ending with $z_{k,t_k}=Z_{B_k}$, such that at every step $z_{k,i+1}=z_{k,i}+E_{v(i)}$
one has
$(E_{v(i)},z_{k,i})>0$, cf. \cite{Laufer72}. Then we glue these sequences as follows.
The first element is 0. Then we list all the elements  of the sequence $\{z_{0,i}\}_i$.
This ends with $Z_{B_0}$. The next element is $Z_{B_0}+E_u=Z_{B_0}+z_{1,0}$.
Then we continue with $Z_{B_0}+z_{1,i}$ adding all elements of $\{z_{1,i}\}_i$.
This ends with $Z_{B_0}+Z_{B_1}$.  Then we repeat the procedure and continue with
$Z_{B_0}+Z_{B_1}+E_u$ and all $Z_{B_0}+Z_{B_1}+z_{2,i}$. We call the steps $0\rightsquigarrow
E_u$, $Z_{B_0}\rightsquigarrow Z_{B_0}+E_u$, $Z_{B_0}+Z_{B_1}\rightsquigarrow
Z_{B_0}+Z_{B_1}+E_u$, etc.,  `gluing steps', all the other $z_i\rightsquigarrow z_{i+1}$ are
`normal steps'. Using (\ref{eq:orthogonal}) one verifies that
along a  normal step  $(E_{v(i)},z_{i})>0$, while along a
gluing step $(E_{v(i)},z_{i})=0$. Note that $\{C'_0-z_i\}_i$ is a decreasing sequence
connecting $C'_0$ with $C'_j$. We claim that along the sequence the integer
$h^1(C'_0-z_i,\calL)$ stays constant. Indeed, in
$$H^1(E_{v(i)}, \calL(-C'_0+z_{i+1}))\to H^1(C'_0-z_i,\calL)\to H^1(C'_0-z_{i+1},\calL)\to 0$$
one has $h^1(E_{v(i)}, \calL(-C'_0+z_{i+1}))=h^0(E_{v(i)}, \calL^*(-z_i))$. But analysing both
cases (normal and gluing steps) we realize that
 $(E_{v(i)}, -c_1(\calL)-z_i)<0$, hence this last  cohomology group vanishes indeed.

 {\it (b)} By part {\it (a)}  we have to show that $h^1(C'_j,\calL)\leq p_g(\tX_j)$.
 Set $i_I:=\min\{i\,|\, I\cap B_i=\emptyset\}$. Since $i_I\geq j$, hence $p_g(\tX_{i_I})\leq p_g(\tX_j)$,
 and $h^1(\tX,\calL)=h^1(\tX_{i_I},\calL|_{\tX_{i_I}})$ by {\it (a)},
 it is enough to verify that $h^1(\tX_{i_I},\calL|_{\tX_{i_I}})\leq p_g(\tX_{i_I})$. Note also that
 $\calL|_{\tX_{i_I}}\in\pic^0(\tX_{i_I})$.
Hence we need to show that
 for a numerically Gorenstein elliptic singularity
 if $\calL\in \pic^0(\tX)$ then $h^1(\calL)\leq p_g(\tX)$.
This  follows from Lemma \ref{lem:INEQj} {\it (a)}.
\end{proof}
\begin{remark}
In general, for arbitrary (non--elliptic) singularity, it is not true that
$h^1(\tX,\calL)\leq p_g(X,o)$ for any line bundle $\calL\in \pic(\tX)$, cf.
\cite[Remark 9.2.3 and Example 9.2.4]{NNII}, see also \cite[Prop. 5.7.1]{NNI}.
\end{remark}

\begin{remark}\label{rem:reduction}
In any situation $h^1(\tX,\calL)$ equals with some $h^1(Z, \calL)$, e.g. for   $Z=\lfloor Z_K\rfloor$ or even $Z_K-s_{[Z_K]}$, cf. Lemma
\ref{lem:szk}. Furthermore, if one wishes a reduction to a smaller
supported cycle, say to $Z|_B$ (as in theorem \ref{prop:vanishing}{\it (a)}), then the existence of an  isomorphism of type
$H^1(Z,\calL)\to H^1(Z|_B,\calL|_{Z|_B})$ usually is obtained  using the vanishing
of $H^1( Z|_{\calv\setminus B}, \calL(-Z|_B))$, which is guaranteed whenever $\calL$ is sufficiently
positive along $\calv\setminus B$. Note that in the above theorem, in the elliptic case,
this reduction can be done with a `minimal positivity requirement' of $\calL$ along $\calv\setminus B$.
See also the statement and the proof of Theorem \ref{th:fixed} as well.

The fact that only such `minimal positivity' is needed is a key additional property of elliptic singularities,
 which makes them special. (For another key special property, the `distinct pole property'
 see Remarks \ref{rem:8points}--\ref{rem:pole}.)
\end{remark}

\subsection{The cycle of fixed components of the line bundles}\label{ss:cyclefixed}

Assume that $(X,o)$ is numerically Gorenstein  and we fix $\calL\in \pic^0(\tX)$.
We denote the cycle of fixed components of $\calL$  by $l$.

\begin{theorem}\label{th:fixed} If $(X,o)$ is
either  minimally elliptic or $C^2\not=-1$ then  the following facts hold.

(a) $\calL(-Z_K)$ has no fixed components.

(b) $l$  belongs to $\{0,C_0,C_1, \cdots,  C_m\}$.
\end{theorem}
\begin{proof} Assume first that $(X,o)$ is minimally elliptic (and the resolution is minimal,
 hence with $Z_K=Z_{min}$).
We recall the following facts, valid in this situation, cf. \cite[Lemma 3.3]{Laufer77}.

 Fix any pair $E_v$ and $E_u$ ($E_v\not=E_u$) of irreducible exceptional divisors.
Then there exists a computation sequence for $Z_{min}$ which starts with
$E_v$ (i.e. $z_1=E_v$) and ends with $E_u$ (i.e. $E_{v(t-1)}=E_u$.
Recall that necessarily $(z_{v(t-1)},E_u)=2$), cf. \cite{Laufer77}.
Moreover, let $E_v$ be an irreducible component whose coefficient  in $Z_{min}$ is
strict greater than one,
then there exists a computation sequence for $Z_{min}$ which starts  and ends with $E_v$.

Next we prove that for any $E_v$ one has
 $h^1(\calL(-Z_K-E_v))<h^0(E_v, \calL(-Z_K))$. Note that this implies that
$H^0(\calL(-Z_K-E_v))\hookrightarrow H^0(\calL(-Z_K))$ is not onto. We use similar arguments as in
\cite[Lemma 3.12]{Laufer77}.

Assume that there exists  a computation sequence $\{z_i\}_{i=1}^t$ with $z_1=E_v$ and ends at some
$E_u$ such that $(E_u,Z_{min})<0$. Then consider the  infinite sequence $\{x_i\}_{i}$:
$Z_{min}+z_1,\ldots, Z_{min}+z_i,\ldots, Z_{min}+Z_{min},2 Z_{min}+z_1,\ldots, 3Z_{min},  3Z_{min}+z_1,\ldots.$
Then $H^1(\calL(-x_{i+1}))\to H^1(\calL(-x_i))$ is onto, hence
$\alpha:H^1(\calL(-(n+1)Z_{min}-E_v))\to H^1(\calL(-Z_{min}-E_v))$ is onto for any
$n\geq 0$. Compose this with $\beta:H^1(\calL(-Z_{min}-E_v))\to H^1(nZ_{min},\calL(-Z_{min}-E_v))$
to get an exact sequence.
But, by formal neighbourhood theorem, $\beta$ is an isomorphism for $n\gg 0$, hence $\alpha=0$, or
$h^1(\calL(-Z_{min}-E_v))=0$.

 If such a computation sequence does not  exist, then $E_v$ is the
 only component with $(Z_{min},E_v)<0$, and the coefficient of $E_v$ in $Z_{min}$ is 1.
 In this case we consider a computation sequence $\{z_i\}_i$
which starts with $E_v$ and ends at some other $E_u$, and a sequence $\{y_i\}_i$ which starts with
$E_u$ and ends at $E_v$. Take the infinite sequence $\{x_i\}_{i}$:
$Z_{min}+z_1,\ldots, Z_{min}+z_i,\ldots, 2Z_{min},2Z_{min}+y_1,\ldots, 2Z_{min}+y_i,\ldots, 3Z_{min},  3Z_{min}+y_1,\ldots.$ Then $H^1(\calL(-x_{i+1}))\to H^1(\calL(-x_i))$ is onto, except when we pass from
$2Z_{min}-E_u$
to $2Z_{min}$, in which case the corank is 1. Hence,
$\alpha$ has corank at most one, $h^1(\calL(-Z_{min}-E_v))\leq 1$.
But $h^0(E_v, \calL(-Z_{min}))\geq 2$.

Next, consider the case when $(X,o)$ is numerically Gorenstein with $m>0$. We will generalize
 the first argument presented above. Fix any $v\in \calv$. Let $\calv_m$ be the set of vertices
 $\{v\,:\, E_v\subset B_m,\ (E_v, Z_{B_m})<0\}$. Clearly, it is nonempty. Moreover, if
 $v\in\calv_m$ then by (\ref{eq:orthogonal}) $(E_v,Z_K)<0$ too. Recall that along any
  computation sequence   of $Z_{min}$ one has $(E_{v(i)},z_i)=1$ except one step when
  it `jumps', that is,  $(E_{v(i)},z_i)=2$. We claim that for any $v\in \calv$ there exists $u\in\calv_m$
and a computation sequence $\{z_i\}_i$ which starts with $E_v$ and jumps at $E_u$.
Indeed, we construct  the computation sequence as follows: it starts with $E_v$ and then we add
consecutively the shortest string of $E_w$'s connecting  $E_v$ with $B_m$. Let the last element
of the string (the first one  which is supported in $B_m$) be $E_{v'}$.
(If $E_v\subset B_m$ then $E_{v'}$ is just $E_v$; and also, it can happen that
$v'$ belongs to $\calv_m$, or not.) Then we continue starting from $E_{v'}$ to
construct the computation sequence of $Z_{B_m}$ which jumps at $E_u$. If $\calv_m\not=\{v'\}$
this is possible. If $\calv_m=\{v'\}$ and the multiplicity of  $Z_{B_m}$ at $E_u$ is $\geq 2$
then again is possible (for both cases see above). Otherwise $Z_{B_m}^2=-1$ which case is excluded.
Then,  finally, after we completed $Z_{B_m}$,  we continue (in an arbitrary way) Laufer's
algorithm to complete $Z_{min}$.
If   we concatenate this computation sequence as in the first part of minimally elliptic situation
(that is,
$Z_K+z_1, \ldots, Z_K+Z_{min},  Z_K+Z_{min}+z_1, \ldots$),
we obtain (by the very same argument) that $h^1(\calL(-Z_{K}-E_v))=0$.

{\it  (b)} $l\in \calS$ and by  {\it (a)} $l\leq Z_K$ too.
Hence the statement follows from Lemma \ref{lem:Ci}.
\end{proof}

It is instructive to compare this last theorem with the example from the Appendix, which
shows that without the required  assumptions of the theorem $l>Z_K$ might happen.
\bekezdes
{\bf Question:} Does the statement of Theorem \ref{th:fixed}{\it (b)} hold under Gorenstein assumption
(even if $C^2=-1$) ?
(Compare  with Appendix.)

\section{The Abel map of elliptic singularities}\label{s:AbelEll}

\subsection{The subspace arrangement $\{V(I\}_I$.} \label{ss:VI}
Let us consider the minimal resolution of an elliptic singularity whose link is
rational homology sphere. For  $Z\gg 0$ and
 $l'\in-\calS'$ we consider  the Abel map $c^{l'}:\eca^{l'}(Z)\to \pic^{l'}(Z)\simeq\C^{p_g}$. If $I=\{u\}$ for
$u\in\calv$ then  we write $V(u)$ instead of $V(I)$.

\begin{theorem}\label{th:Vk} With the above notations the following facts hold.

(1) $V(u) =0$ if $u\in B_{-1}\setminus B_0$ (in the $Z_K\not\in L$ case). In both cases,
$V(u) =0$ if $u\in B_{0}\setminus B_\alpha$.

(2) $ V(u) = V(v)$, if $u, v \in B_j \setminus B_{j+1}$ for some $ 0 \geq j \geq m$ (with the notation $B_{m+1} = \emptyset$).

(3) $ V(u) \subset V(v)$, if $u \in  B_i \setminus B_{i+1}$ and $v \in  B_j \setminus B_{j+1}$ for some $ m \geq j \geq i \geq 0$.

(4)   $\dim V(u) = k+1$ whenever $u\in B_{\alpha+k}\setminus B_{\alpha+k+1}$,
where $0\leq k\leq m-\alpha$.

(5) For any  $I \subset \calv$  let $i$ be the maximal number, such that there exists a vertex
$u \in I$ with $u \in  B_i \setminus B_{i+1}$. Then $ V(I)=V(u) $. Hence,
the linear subspace arrangement $\{V(I)\}_{I\subset \calv}$
is the flag $\{V(B_{\alpha+k}\setminus B_{\alpha+k+1})\}_{0\leq k\leq m-\alpha }$ in
$\C^{p_g}=\C^{m-\alpha+1}$, where $\dim (V(B_{\alpha+k}\setminus B_{\alpha+k+1}))=k+1$.

(6) $\dim\im (c^{-E_{u}^*}(Z)) = \dim V(u)$ for every vertex $u\in \calv$.

(7) $\dim\im (c^{l'}(Z)) = \dim V(I(l'))$, where $l'\in -\calS'$
(hence $\dim\im (c^{l'}(Z)) =\dim\im (c^{nl'}(Z))$ for $n\geq 1$).
In particular,    $\overline{\im (c^{l'}(Z))} = A(l')$.

(8)   $h^1$ is uniform on $\overline{ \im (c^{l'}(Z))} \subset \pic^{l'}(Z)$:
$h^1(\calL) =  p_g - \dim V(I(l') )=p_g(X_{\calv\setminus I(l')},o_{\calv\setminus I(l')})$
for every $\calL \in  \overline{\im (c^{l'}(Z))}$.

(9) If $\calL\in \overline{ \im (c^{l'}(Z))}$ (e.g., if $\calL$
has no fixed components) then $h^1 (\calL^{\otimes n})=h^1(\calL)$
for  $n\geq 1$.

(10) All the fibers of $c^{l'}(Z):\eca^{l'}(Z)\to \pic^{l'}(Z)$ have the same dimension
$(l',Z)-\dim V(I(l'))$.
\end{theorem}

Property {\it (9)} for bundles of type $\calL=\calO_{\tX}(-Z)$ $(Z\in L_{\geq 0})$ was proved by
Okuma in \cite{Okuma18} too.

\begin{proof} Parts {\it (1)--(5)} follow from the duality statement of Proposition \ref{prop:HDIZ}
and the structure of the $\{\Omega_I\}_I$ linear subspace arrangement from Corollary \ref{cor:omegaflag}. The interested reader might verify directly the duality through the Laufer integration
detailed in \cite[\S 7]{NNI}.

{\it (6)} In general, cf. \ref{def:VZ},  $\dim\im (c^{l'}(Z)) \leq  \dim V(l')$ ($\dag$).
Assume that $u\in B_i\setminus B_{i+1}$. If $i<\alpha $ then by  {\it (1)}  and ($\dag$)
we are done.  Otherwise, by the general statement of Corollary \ref{cor:dim} and from the structure of the poles of differential forms constructed in \ref{bek:diff} follows that
$\dim\im (c^{-E_{u}^*}(Z))=i+1-\alpha$. On the other hand, by  {\it (4)},
$\dim V(u)=i+1-\alpha$ as well.

{\it (7)} We reduce the statement to {\it (6)}
via  the multiplicative structure from  subsection \ref{bek:addDiv} and Theorem \ref{prop:AZ}.
Firstly,
$s^{l'_1,l'_2}(Z)$ is dominant and quasi--finite. Therefore, for  $l'=\sum a_vE^*_v$, one has
$\dim V(l')=\dim (\sum_{a_v\not=0} V(v))=  \dim (\sum_{a_v\not=0} \im (c^{a_vE^*_v}))=
\dim \im (c^{l'})$.

{\it (8)} If $\calL^{im}_{gen}$ is a generic element of $\im (c^{l'})$ then $h^1(\calL^{im}_{gen})=
p_g-\dim \im (c^{l'})$ (cf. Theorem \ref{prop:AZ}{\it (d)}), which equals $p_g-\dim V(l')$ by
 {\it (7)}. Hence, by semicontinuity (see e.g. \cite[Lemma 5.2.1]{NNI}),
$h^1(\calL)\geq p_g-\dim V(l')$ for any
$\calL \in  \overline{\im (c^{l'}(Z))}$. On the other hand, by Proposition \ref{prop:vanishing}
$h^1(\calL)\leq  p_g-\dim V(I(l'))$. (See also Theorem \ref{prop:AZ}{\it (c)}.)

For {\it (9)} use {\it (8)} and
{\it (10)} follows from  \cite[Lemma 3.1.7]{NNI}.
\end{proof}

\begin{remark}\label{rem:8points}
(a) Parts {\it (6)--(7)} can be compared with  Theorem \ref{prop:AZ}{\it (e)}.
Theorem \ref{th:Vk} says that in the case of elliptic singularities there is no need to take any
multiple $nl'$ in order to obtain the maximal stabilized  dimension of $\im (c^{nl'})$, that is,
$\im (c^{l'})=\im (c^{nl'})$ for  any $l'\in-\calS'$ and $n\geq 1$. As a consequence,
the closure of any $\im (c^{l'})$ is an affine space.

Similarly, parts {\it (8)--(9)} can be compared with Theorem \ref{prop:AZ}{\it (e)}:
in order to have a uniform $h^1$--behaviour along the (closure of the image), no stabilization
is needed either.

The main property of elliptic singularities, which is responsible for the fact that the
stabilization take place from the very first term, cf. {\it (6)--(7)--(9)} above,
is the existence of forms $\{\omega_j\}_{j=1}^{p_g}$, which form a basis of
$H^0(\Omega^2_{\tX}(Z))/H^0(\Omega^2_{\tX})$ ($Z\gg0$), and which satisfies the `distinct
pole property' discussed in Corollary \ref{cor:dim}. This reads as follows: For any $v\in\calv$ let
$\calj_v$ be the index set of those forms $\omega_j$ (from this list),
 which have nontrivial pole along $E_v$.
Then the poles along $E_v$ of all forms $\{\omega_j\}_{j\in \calj_v}$ are pairwise distinct.
For elliptic singularities this property is guaranteed by Corollary \ref{cor:omegaflag}, since
the pole of each
$\omega_j$ is $ C'_j$.

The point is that if a normal surface singularity (with rational homology sphere link)
admits  a set of $p_g$ independent forms with the `distinct pole property' then
the above stabilization properties  {\it (6)--(7)--(9)} hold. This follows from Propositions
\ref{prop:HDIZ} and \ref{prop:Res} proved in \cite{NNI}.

It is natural to ask if the `distinct pole property' is an idiosyncrasy  merely of  elliptic
singularities. The answer is no, there are many germs with this property, see e.g. the next example.
\end{remark}

\begin{example}\label{rem:pole}
{\bf A singularity with `distinct  pole property'.}
Consider the following resolution graph (the associated minimal one can be obtained by blowing down the
two `cusps'.) The graph is not elliptic, $\min\chi=-1$ (and it  has two distinct candidates for the
elliptic cyle).

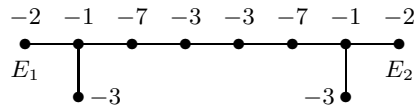
\begin{picture}(200,50)(-20,0)
\put(230,40){\makebox(0,0){\small{$-2$}}}
\put(90,40){\makebox(0,0){\small{$-2$}}}
\put(110,40){\makebox(0,0){\small{$-1$}}}
\put(170,30){\circle*{4}}\put(190,30){\circle*{4}}\put(210,30){\circle*{4}}\put(230,30){\circle*{4}}
\put(150,30){\circle*{4}}
\put(210,10){\circle*{4}}
\put(130,40){\makebox(0,0){\small{$-7$}}}
\put(150,40){\makebox(0,0){\small{$-3$}}}
\put(170,40){\makebox(0,0){\small{$-3$}}}\put(190,40){\makebox(0,0){\small{$-7$}}}
\put(210,40){\makebox(0,0){\small{$-1$}}}
\put(200,10){\makebox(0,0){\small{$-3$}}}\put(120,10){\makebox(0,0){\small{$-3$}}}
\put(90,20){\makebox(0,0){\small{$E_1$}}}\put(230,20){\makebox(0,0){\small{$E_2$}}}
\put(90,30){\circle*{4}}
\put(110,30){\circle*{4}}
\put(130,30){\circle*{4}}
\put(110,10){\circle*{4}}
\put(90,30){\line(1,0){140}}\put(110,10){\line(0,1){20}}\put(210,10){\line(0,1){20}}
\end{picture}

It is realized e.g. by the hypersurface singularity with non--degenerate Newton boundary
$\{z^3+x^{13}+y^{13}+x^2y^2=0\}$. This analytic structure has $p_g=5$ and
it is clearly Gorenstein. Let
$\omega $ be the Gorenstein form (with pole $Z_K$). Then the classes of the five forms
$\omega, \, \omega x, \, \omega x^2, \,  \omega y, \, \omega y^2$ constitute a basis of
$H^0(\Omega^2_{\tX}(Z))/H^0(\Omega^2_{\tX})$ ($Z\gg0$), and they satisfy the
`distinct  pole property' (the verification is left to the reader; the divisor of $x$ is $E_1^*$, while the divisor of $y$ is $E^*_2$).

In fact, even if we take the generic analytic structure on this graph (cf. \cite{NNII}), the property survives. Indeed, in this case
 $p_g=2$ \cite{NNII} and
 the cycle of poles of the corresponding two differential forms have even distinct support.
They are supported on the two minimally elliptic subgraphs obtained by deleting the two central
$(-3)$ vertices. Hence, again they satisfy the `distinct  pole property'.

(In fact we expect that the `distinct  pole property' is true for any analytic type supported on this
graph. It is really amazing that  for such graphs, when for any analytic type supported on them the
`distinct  pole property'  holds, the `stability'  analytic property
 `{\it $h^1 (\calL^{\otimes n})=h^1(\calL)$ for  $n\geq 1$ and $\calL$ without fixed components}'
 is imposed by the combinatorics of the graph.)
\end{example}

\subsection{The WECC and ECC terminology regarding the set $\{\im (\widetilde{c}^{l'})\}_{l'}$ \cite[\S 9]{NNI}.}\label{ss:terminology}
The mutual position of  the natural line bundle $\calO_Z(l')$ and $\im (c^{l'})$ (or,
equivalently, of $0$ and $\im (\widetilde{c}^{l'})$) is codified in the
following submonoid of $\calS'$. We set $\calS_{im}':= \{-l'\,:\, \calO_Z(l')\in \im (c^{l'})\}=
\{-l'\,:\, 0\in \im(\widetilde{c}^{l'})\}$.
In other words, $l'\in\calS'_{im}$ if and only if $\calO_Z(-l')$ has no fixed components.

As usual, we define the saturation of a submonoid ${\mathcal M}\subset \calS'$ as
$\overline{{\mathcal M}}:=\{l'\in\calS'\,: nl'\in {\mathcal M} \ \mbox{for some $ n\geq 1$}\}$.
Accordingly, $\overline {\calS'_{im}}=\{-l'\in\calS'\,: \,
0\in \im(\widetilde{c}^{nl'}) \ \mbox{for some $ n\geq 1$}\}$.

Recall also that we say that a resolution $\tX$ satisfies the `End Curve Condition' (ECC)
if $E^*_v\in \calS'_{im}$ for any end vertex $v$. The terminology was introduced by Neumann and Wahl
 in the context of splice quotient singularities \cite{NWsq}.
By the `End Curve Theorem' \cite{NWECTh} $\tX$ satisfies  ECC if and only the analytic type is
splice quotient associated with the dual graph of $\tX$. Furthermore, given a  resolution graph $\Gamma$,
a singularity resolution $\tX$ with dual graph $\Gamma$ and which satisfies ECC exists
if and only if  the graph satisfies the
`semigroup and congruence conditions'  of Neumann--Wahl \cite{NWsq}, or, equivalently, the `monomial condition' of
Okuma \cite{Ok}.

We say that  $\tX$ satisfies the
`Weak End Curve Condition' (WECC) if $E^*_v\in \overline{\calS'_{im}}$ for any end vertex $v$.
In fact, by  \cite[Proposition 9.2.2]{NNI}, $\tX$ satisfies the WECC if and only if
$\overline{\calS'_{im}}=\calS'$. In general, $\overline{\calS'_{im}}\not=\calS'$,
for concrete examples see \cite{NNI} (or below in Example \ref{ex:NOECC}).

\subsection{The set   $\{\overline{\im (\widetilde{c}^{l'})}\}_{l'}\subset H^1(\calO_Z)$}\label{ss:images}
By Theorem \ref{th:Vk}{\it (7)} $\overline{ \im (c^{l'})}=A(l')$ is an affine space of dimension
$\dim V(l')$. Furthermore, by the general result Theorem \ref{prop:AZ}{\it (e)}
$\im (c^{nl'})$ is automatically closed, hence it equals
$\overline{ \im (c^{nl'})}=A(nl')$ ($n\gg 1$),
a parallel affine space with $A(l')$ of the same dimension.
(Indeed, use {\it (7)} for $nl'$ and $V(nl')=V(l')$ from Theorem \ref{prop:AZ}{\it (a)}.)

\begin{remark}\label{rem:17} (a) Even for elliptic singularities it can happen that
$\im (c^{l'})\not=\overline {\im (c^{l'})}$. For a non--Gorenstein example see e.g.
part \ref{bek:A9} of Appendix. For a Gorenstein example  we can take even the very same graph as in Appendix, and its hypersurface realization
$\{x^2+y^3+z^{17}=0\}$ with $p_g=2$ and $Z_K=E^*_2$.
 Then by Theorem \ref{th:Vk}{\it (7)--(8)} $c^{-Z_K}=c^{-E^*_2}$ is dominant. We claim that
%
%
%
%
$\calO(-E^*_2)\not\in \im(c^{-E^*_2})$ (though $\tX$ satisfies the ECC, i.e. $\calO(-E^*_v)\in \im(c^{-E^*_v})$ for any end--vertex $v$, and also $\calO(-nE^*_2)\not\in \im(c^{-nE^*_2})$
for $n\gg 1$).

This can be proved as follows. Note that $E_1^*$ and $E^*_2$ cannot be realized
as divisors of functions (restricted to $E$) simultaneously, since the linking number of their
arrows is one, or equivalently, if $f_1,f_2$ are some realizations then the degree of
$(f_1,f_2):(X,o)\to (\C^2,0)$ would be one. Since in this hypersurface case (or any Gorenstein case)
$E^*_1$ is realized, $E^*_2$ cannot be realized.

(b) Note that WECC says that $0\in \im (\widetilde{c}^{-nE^*_v})$  for all end vertices $v$
(and $n\gg 1$). This in the elliptic case
implies that $\im (\widetilde{c}^{-nE^*_v})=V(-nE^*_v)=V(-E^*_v)$, hence
$\overline{ \im (\widetilde{c}^{-E^*_v})}=V(-E^*_v)$, or
$0\in\overline{\im (\widetilde{c}^{-E^*_v})}$.  Nevertheless, from
$0\in\overline{\im (\widetilde{c}^{-E^*_v})}$  the inclusion
  $0\in \im (\widetilde{c}^{-E^*_v})$ (that is, the ECC property)
in general
 cannot be guaranteed yet (by a general argument valid for any normal surface singularity).
 However, for elliptic germs it works, and it will
 be proved later in Theorem \ref{th:9.4.2}.
\end{remark}
\bekezdes\label{bek:shifted}
{\bf The structure of the affine subspaces
$\{\overline{\im (\widetilde{c}^{l'})}\}_{l'}$} of $\pic^0$ is the following.
First, recall that
 the structure of the linear subspace arrangement $\{V(I)\}_I$ is very simple, it is a flag.
 Associated with a fixed $V(I)$ there are several (in general,  infinitely many) associated
 parallel affine subspaces of type $\overline{\im (\widetilde{c}^{l'})}$. Indeed,
 as in Theorem \ref{th:Vk}{\it (5)},
 for any  $I \subset \calv$  let $i$ be the maximal number, such that there exists a vertex
$u \in I$ with $u \in  B_i \setminus B_{i+1}$. For all $l'$,   such that $I(l')$  has
 the same $i$, all the affine spaces  $\overline{\im (\widetilde{c}^{l'})}$ have the same dimension, and are parallel
 to the same $V(I(l'))$. (Their shifts have even an additional `semigroup structure' in the sense that
 if $A(l')=a(l')+V(l')$ then $A(nl')=n\cdot a(l')+V(l')$.)

In particular, if two subspaces of type $\overline{\im (\widetilde{c}^{l'})}$
intersect each other nontrivially, then one of them should contain the other one.

\begin{example}\label{ex:NOECC}
 It can really happen
  that these parallel affine subspaces do not collapse into one vector space (namely into $V(I)$),
 see e.g. any elliptic singularity which does not satisfy ECC.
 For example, the points $\{\calL_n\}_{n\geq 1}$ in \ref{bek:A10} are all parallel affine subspaces associated with  $V(Z_{min})=0$.

 Next we present a Gorenstein case as well.
 Take  $\{z^2=(u^2-w)(u^{11}-w^3)\}$
 with the next  graph (where the $(-2)$--vertices are unmarked):

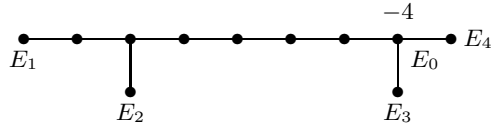
\begin{picture}(200,50)(-20,0)
\put(170,30){\circle*{4}}\put(190,30){\circle*{4}}\put(210,30){\circle*{4}}\put(230,30){\circle*{4}}
\put(150,30){\circle*{4}}\put(210,10){\circle*{4}}
\put(70,30){\circle*{4}}
\put(210,40){\makebox(0,0){\small{$-4$}}}
\put(70,22){\makebox(0,0){\small{$E_1$}}}\put(110,2){\makebox(0,0){\small{$E_2$}}}
\put(210,2){\makebox(0,0){\small{$E_3$}}}
\put(240,30){\makebox(0,0){\small{$E_4$}}}
\put(220,22){\makebox(0,0){\small{$E_0$}}}
\put(90,30){\circle*{4}}
\put(110,30){\circle*{4}}
\put(130,30){\circle*{4}}
\put(110,10){\circle*{4}}
\put(70,30){\line(1,0){160}}\put(110,10){\line(0,1){20}}\put(210,10){\line(0,1){20}}
\end{picture}

One verifies that $m=1$, hence $p_g=2$. Let the flag of $\{V(I)\}_I$'s be denoted by
 $0\subset V\subset \C^2$, where
$\dim(V)=1$, cf. \ref{th:Vk}{\it (5)}.

This topological type does not support any any analytic structure with ECC (it does not satisfies
the semigroup or the monomial condition at the $(-4)$--node). In fact, later
we will show that this graph does not admit any analytic type with WECC either.
This will follow  either from Theorem \ref{th:extension} directly,
or from Theorem \ref{th:9.4.2} using the nonexistence of
ECC structure.   Hence,
for the present Gorenstein structure WECC fails
at least  at one of the end--vertices.

On the other hand, one verifies that ${\rm div}_E(w)=E_1^*$,
${\rm div}_E(z)=E_0^*+E^*_2$,
${\rm div}_E(u^2-w)=2E_0^*$,
${\rm div}_E(u^{11}-w^3)=2E_2^*$.
(In fact if we denote ${\rm div}(z)$ by $E_0^*+E^*_2+D_0+D_2$, where $D_0$ and $D_2$ are two transversal cuts of $E_0$ and $E_2$ respectively, then ${\rm div}(u^2-w)=2E_0^*+2D_0$ and
${\rm div}(u^{11}-w^3)=2E_2^*+2D_2$.)
Hence
$E_0^*,\, E_1^*,\, E^*_2\in \overline{ \calS'_{im}}$. Therefore,  the only obstruction for WEEC can be
caused by $E_3$ or $E_4$.
But, ${\rm div}_E(u)=E_3^*+E^*_4$.
(The strict tranform is $\{u=z^2-w^4=0\}$, whose two components are permuted by the $\Z_2$--Galois action of the double covering $u\mapsto u,\ w\mapsto w,\, z\mapsto -z$.)
Hence there exists
$D_j\in \eca^{-E^*_j}$ ($j=3,4$), so that $\calO_{\tX}(D_3+D_4+E_3^*+E^*_4)=0$.
This means that the two points $\calL_j:=\calO_{\tX}(D_j+E^*_j)\in \im ( \widetilde{c}^{-E^*_j})$
($j=3,4$)
satisfy $\calL_3+\calL_4=0$ in $\pic^0$. Note also that
$V(E^*_3)=V(E^*_4)=V$ (use \ref{th:Vk}); hence $\calL_3+\calL_4=0$ implies that WEEC $\Leftrightarrow \,
0\in\overline{\im (\widetilde{c}^{-E^*_3})} \,
\Leftrightarrow\,
0\in\overline{\im (\widetilde{c}^{-E^*_4})}$. Since WECC does not hold,  we get that
$0 \not \in\overline{\im (\widetilde{c}^{-E^*_3})}$  and
$0\not \in\overline{\im (\widetilde{c}^{-E^*_4})}$. Furthermore,
$\overline{\im (\widetilde{c}^{-nE^*_3})}=n\calL_3+V$ and
$\overline{\im (\widetilde{c}^{-nE^*_4})}=-n\calL_3+V$ $(n\geq 1)$. All these 1--dimensional
affine subspaces are
distinct parallel ones in $\pic^0=\C^2$, all associated with $V$. (The Galois action is $n\mapsto -n$.)

$V$ can also be  realized as some $\overline{\im (\widetilde{c}^{l'})}$. Indeed,
since $Z_{min}=E_3^*+E_4^*={\rm div}_E(u)$, $0\in \overline{\im (\widetilde{c}^{-Z_{min}})}$, hence
 $V=\overline{\im (\widetilde{c}^{-Z_{min}})}$.
\end{example}

\section{The stratification of $\pic^{l'}$ according to $h^1$.}\label{ss:BN}

\subsection{Definition of the strata $W_{l',k}$}\label{ss:strata}
\begin{definition}\label{def:BN} \ \cite[5.8]{NNI} We fix any singularity, one of its resolutions $\tX$,
  and $l'\in -\calS'$. We  define
 $W_{l', k} = \{\calL \in \pic^{l'}(\tX) : h^1(\tX, \calL) = k\}$. Its closure in
 $\pic^{l'}(\tX)$ will be denoted by
$\overline{W_{l', k}}$.

If it is necessary, when we handle several resolution spaces, we might also write
$W_{l', k}(\tX)$.
\end{definition}

E.g., by Theorem \ref{prop:AZ}{\it (d)} and semicontinuity, if
 $c: = p_g - \dim\im (c^{l'}(Z))$, then $ \overline{\im (c^{l'}(Z))}\subset
 \overline{W_{l', c}}$.
Hence, for each $l'\in-\calS'$,  $\pic^{l'}(\tX)$ has a stratification
into constructible subsets according to $\calL\mapsto h^1(\calL)$.
We will describe the stratification in several steps.

\subsection{The general reduction to $l'=0$}\label{ss:Redto}
For any fixed  $l'\in -\calS'$ let $I=I(l')$ be the $E^*$--support of $l'$, and let $0\leq i\leq m+1$
be the maximal index with $I\cap B_{i-1}\not=\emptyset$. ($i=0$ happens when $I=\emptyset$.)

Let $K$ be the kernel of $\pi_i:H^1(\calO_{\tX})\to H^1(\calO_{\tX_i})$, reinterpreted as
the kernel of $\pic^0(\calO_{\tX})\to \pic^0(\calO_{\tX_i})$. Note that if $\calL\in \pic^{l'}(\tX)$
then $\calL|_{\tX_i}$ has trivial Chern class, hence the restriction induces  a well--defined
affine map $\pi_i^{l'}:\pic^{l'}(\tX)\to \pic^0(\tX_i)$. The group $K$ acts on $\pic^{l'}(\tX)$
via $(\calL^0,\calL)\in K\times \pic^{l'}(\tX)$, $(\calL^0,\calL)\mapsto \calL^0\otimes \calL$. Thus,
 the orbits are exactly the affine fibers of $\pi_i^{l'}$.

\begin{proposition}\label{prop:restrict}
(a) $\calL\mapsto h^1(\calL)$ is constant  along the fibers of $\pi_i^{l'}$.

(b) $W_{l',k}(\tX)=(\pi_i^{l'})^{-1} (W_{0,k}(\tX_i))$.
\end{proposition}
\begin{proof}
If  $(\calL^0,\calL)\in K\times \pic^{l'}(\tX)$, then $c_1(\calL^0\otimes \calL)=c_1(\calL)=l'$,
 hence by
Theorem \ref{prop:vanishing} one gets $h^1(\calL^0\otimes \calL)=h^1(\calL^0\otimes \calL|_{\tX_i})=
h^1(\calL|_{\tX_i})$.
\end{proof}

In particular, the $h^1$--stratification of $\pic^{l'}(\tX)$ is completely determined (as a pull--back
via an affine map) by the $h^1$--stratification of $\pic^0(\tX_i)$. This reduces its study
to the $l'=0$ case.

In the next subsections \ref{ss:Gor0} and \ref{ss:nonGor0}
we clarify the $l'=0$ case. Though the statements for the
Gorenstein and non--Gorenstein cases can be formulated uniformly, we still decided to separate the two cases; in this way we can emphasize better the peculiarities of both situations.

\subsection{The case $(X,o)$ Gorenstein and $l'=0$}\label{ss:Gor0}
For any $j\in\{0,\ldots, m+1\}$ we denote the natural linear projection
$H^1(\calO_{\tX})\to H^1(\calO_{\tX_j})$ by $\pi_j$, and we also interpret it as the restriction
$\pic^0(\tX)\to \pic^0(\tX_j)=\pic^0(C'_j)$. (Here and below, by convention, $C_{-1}=
  C'_{m+1}=0$ and $H^1(\calO_{\tX_{m+1}})=0$.)

We write $W_{0,k}$  for $W_{0,k}(\tX)$. Recall that $m+1=p_g$.

\begin{theorem}\label{th:W1}  Assume that $(X,o)$ is Gorenstein, and we fix some
$\calL\in \pic^0(\tX)$. Let $l$ be the cycle of
fixed components of $\calL$ and write $\widetilde{l}:=\min\{l,Z_K\}$.
($\widetilde{l}\in\{C_{-1}, C_0,\ldots, C_m\}$ by Lemma \ref{lem:Ci}.)

\vspace{1mm}

\noindent {\bf (Coarse version:)}  The following facts are equivalent:

\vspace{1mm}

(a) $\widetilde{l}\in \{C_{-1}, C_0,\ldots,  C_{j-1}\}$;

(b) $\calL\in \ker(\pi_j)$;

(b') $\calL(-C_{j-1})|_{C'_j}\in \pic^0(C'_j)$ is trivial;

(c) $\calL\in \overline{ W_{0, p_g-j}}$;

(d) $\calL\in \overline{ \im (\widetilde{c}^{-C_{j-1}}(Z)) }$ (via identification
$\pic^0(Z)=\pic^0(\tX)$, where $Z\gg 0$).

In particular, each  $\overline{ W_{0, p_g-j}}$ is irreducible, it is
a $j$--dimensional linear subspace of $\pic^0(\tX)$, and, in fact, it equals both
 $\overline{ \im (\widetilde{c}^{-C_{j-1}} (Z)})$ and $\ker(\pi_j)$. E.g.,
 $W_{0,p_g}=\{\calO_{\tX}\}$ and $\overline{ W_{0,0}}=\pic^0(\tX)$.
 Furthermore,
 $W_{0,p_g-j+1}\subset \overline{ W_{0, p_g-j}}$ whenever $j>0$
 and $W_{0, p_g-j}= \overline{ W_{0, p_g-j}}\setminus \overline{ W_{0, p_g-j+1}}$.

\vspace{2mm}

\noindent {\bf (Fine version:)}  The following facts are equivalent:

\vspace{1mm}

(f-a) $\widetilde{l}= C_{j-1}$;

(f-b) $\calL\in \ker(\pi_j)\setminus \ker(\pi_{j-1})$;


(f-c) $\calL\in  W_{0, p_g-j}$.



\end{theorem}

\begin{proof}
First we prove the `Coarse version'.
Notice that
{\it (b)} reads as $\calL|_{C'_j}$ is trivial, hence
{\it (b)$\Leftrightarrow$(b')}
follows from the triviality of $\calO_{C'_j}(-C_{j-1})$, cf. Theorem \ref{e220}.

{\it (a)$\Rightarrow$(b)} follows from Lemma \ref{lem:INEQj}{\it (b)}:
if $\widetilde{l}=C_{i-1}$ $(i\leq j)$ then $\calL|_{C_i'}$ is trivial, hence
$\calL|_{C_j'}$ is trivial too.

{\it (b)$\Leftrightarrow$(d)} By restriction, $\im (c^{-C_{j-1}})\subset \{\widetilde{\calL}
\in \pic^{-C_{j-1}}(\tX)\,:\, \widetilde{\calL}|_{C_j'}=\calO_{C_j'}\}$.
That is,  via {\it (b)$\Leftrightarrow$ (b')},
 $\im (\widetilde{c}^{-C_{j-1}})\subset \ker(\pi_j)$. This implies
 $\overline{\im (\widetilde{c}^{-C_{j-1}})} \subset \ker(\pi_j)$ too.
Note that  $\overline{\im (\widetilde{c}^{-C_{j-1}})}$ is an affine subspace of dimension $j$
(cf. Theorem \ref{th:Vk}) and $\dim \ker(\pi_j)=j$ too (cf. Theorem \ref{e220}). Therefore,
 $\overline{\im (\widetilde{c}^{-C_{j-1}})} = \ker(\pi_j)$.

{\it (c)$\Rightarrow$(a)} Fix some $\calL\in  W_{0, p_g-j}\subset  \pic^0(\tX)$.
This means that $h^1(\calL)=p_g-j$ ($\dag$).
Now, we know that $\widetilde{l}$ associated with $\calL$ is $C_{i-1}$ for some $i$, cf. Lemma \ref{lem:Ci}. By Lemma
\ref{lem:INEQj}{\it (b)} $h^1(\calL)=p_g-i$, hence by  ($\dag$) one gets $i=j$.
If $\calL\in \overline{ W_{0, p_g-j}}\setminus  W_{0, p_g-j}$, then by
semicontinuity of $h^1$ one has $h^1(\calL)=p_g-j'$ for some $j'< j$,
hence by the very same argument
$\widetilde{l}=C_{j'-1}$.

{\it (b)$\Rightarrow$(c)} Assume that $\calL|_{C_j'}$ is trivial.
Then $h^1(\calL)\geq h^1(\calL|_{C'_j})=
h^1(\calO_{C'_j})=p_g-j$ (cf. Theorem \ref{e220}). If $h^1 (\calL)=p_g-j$, then we are done.
Next assume that $h^1 (\calL)=p_g-j'>p_g-j$ for some $j'<j$. Then by the implication
{\it (c)$\Rightarrow$(b)} already proved, from $h^1(\calL)=p_g-j'$
 we get  $\calL\in \ker(\pi_{j'})$.
Consider a convergent sequence of line bundles $\{\calL_n\}_n$ in $\ker(\pi_j)\setminus
\ker(\pi_{j-1})$ with $\lim_{n\to \infty}\calL_n=\calL$. As above, but now for $\calL_n$,
 $h^1(\calL_n)\geq h^1(\calL_n|_{C'_j})=p_g-j$. However, here necessarily we should
have  equality (otherwise, if $h^1(\calL_n)> p_g-j$ then by the  implication
{\it (c)$\Rightarrow$(b)}
$\calL_n\in \ker(\pi_{j-1})$, which leads to a contradiction.) Hence
$\calL_n\in W_{0,p_g-j}$ and $\calL=\lim_{n\to\infty}\calL_n\in \overline{W_{0,p_g-j}}$.

The fine version follows directly from the coarse one.
\end{proof}

\begin{remark}\label{rem:d} The `fine version' cannot be completed with
{\it (f-d)}   \ $ \calL\in \im (\widetilde{c}^{-C_{j-1}}(Z)) $, as a forth condition equivalent
with {\it (f-a)$\Leftrightarrow$(f-b)$\Leftrightarrow$(f-c)}.

Indeed, {\it (f-d)$\Rightarrow$ (f-a,f-b,f-c)} does not hold. Take for example a
Gorenstein singularity with $m=1$, and $\calL:=\calO$. Then $\calO(-Z_{min})$ has no fixed components
\cite[\S 5]{weakly}, that is,  $\calO\in \im (\widetilde{c}^{-Z_{min}})$,
hence {\it (f-d)} holds for $j=1$.
On the other hand, $\calO\not\in W_{0,1}$.

However, the opposite implication
{\it (f-a,f-b,f-c)$\Rightarrow$(f-d)}
holds whenever $l\leq Z_K$ (e.g. when either  $C^2\not=-1$, or when $(X,o)$ is minimally
 elliptic, for details
see Theorem \ref{th:fixed}). In such case $\widetilde{l}=l$. Hence,
if $\widetilde{l}=C_{j-1}$ then in fact $l=C_{j-1}$, or $\calL(-C_{j-1})$ has no fixed components,
$\calL(-C_{j-1})\in \im (c^{-C_{j-1}})$ and $\calL\in \im (\widetilde{c}^{-C_{j-1}})$.
\end{remark}

\bekezdes {\bf Question:} Does
 $W_{0,p_g-j}\subset  \im (\widetilde{c}^{-C_{j-1}})$
hold for any Gorenstein singularity?

\subsection{The case $(X,o)$ non--Gorenstein and $l'=0$}\label{ss:nonGor0}
Fix an elliptic numerically Gorenstein singularity with length $m+1$ and  minimal resolution
$\tX$. Let $\alpha$ be defined as in (\ref{eq:alpha}).
Recall that $p_g(\tX)=p_g(\tX_j)$  and $\pic^0(\tX)=\pic^0(\tX_j)=\pic^0(C'_j)$ for any $0\leq j\leq \alpha$.

\begin{theorem}\label{th:W1b}  We fix some
$\calL\in \pic^0(\tX)$ and we denote by  $l$ the cycle of
fixed components of $\calL$. We also set $\widetilde{l}:=\min\{l,Z_K\}$.

\vspace{1mm}

Fix  any $j\in \{0, \ldots , \alpha\}$. Then
 for any such $j$, $\im (\widetilde{c}^{-C_{j-1}})$  consists of a single point and
$$ \Big\{\calL\,:\, \calL(-C_{j-1})|_{C'_j} \ \mbox{is trivial in }
\ \pic^0(C' _j)\Big\}=
\im (\widetilde{c}^{-C_{j-1}}).$$
  Moreover,  one has the next inclusions as well:
$$\{\calL\,:\, \widetilde{l}=C_{j-1}\}\subset
\im (\widetilde{c}^{-C_{j-1}}) \ \ \mbox{and} \ \
\im (\widetilde{c}^{-C_{j-1}})\subset \cup_{j'\leq j} \{\calL\,:\, \widetilde{l}=C_{j'-1}\}.$$
 Furthermore,
$W_{0,p_g(\tX)}=\cup_{j=0}^\alpha \im (\widetilde{c}^{-C_{j-1}})$.
Corresponding to $j=0$ (when $C_{-1}=\widetilde{l}=l=0$), $\im (\widetilde{c}^{-C_{-1}})=\{\calL:\, \widetilde{l}=C_{-1}\}=\calO$,
the origin $0\in \pic^0(\tX)$. (About the position of the other points we claim nothing.)

Corresponding to the indices $j\in \{0, \alpha+1,\alpha+2,\ldots, m+1\}$ we have
similar statements as in
Theorem \ref{th:W1}. Namely:

\noindent {\bf (Coarse version:)}  For any $j\in \{0, \alpha+1, \alpha+2,\ldots, m+1\}$
the following facts are equivalent:

\vspace{1mm}

(a) $\widetilde{l}\in \{C_{-1}, C_\alpha, C_{\alpha+1}, \ldots, C_{m+1}\}$;

(b) $\calL\in \ker(\pi_j)$;

(b') $\calL(-C_{j-1})|_{C'_j}\in \pic^0(C'_j)$ is trivial;

(c) $\calL\in \overline{ W_{0, p_g(\tX_j)}}$;

(d) $\calL\in \overline{ \im (\widetilde{c}^{-C_{j-1}}(Z)) }$ (via identification
$\pic^0(Z)=\pic^0(\tX)$, where $Z\gg 0$).

In particular, for any $j>\alpha $,
each  $\overline{ W_{0, p_g(\tX_j)}}$ is irreducible, it is
a $(p_g(\tX)-p_g(\tX_j))$--dimensional linear subspace of $\pic^0(\tX)$, and, in fact, it equals both
 $\overline{ \im (\widetilde{c}^{-C_{j-1}} (Z)})$ and $\ker(\pi_j)$.
 E.g.,
 $\overline{ W_{0,0}}=\pic^0(\tX)$.
 Furthermore,
 $W_{0,p_g(\tX_{j-1})}\subset \overline{ W_{0, p_g(\tX_j)}}$ whenever $j\geq \alpha+2$
and the same is true for $j=\alpha +1$ too if we disregard the points
$\cup_{j=1}^\alpha \im (\widetilde{c}^{-C_{j-1}})\subset W_{0,p_g(\tX)}$.

\vspace{2mm}

\noindent {\bf (Fine version:)}
If we disregard the points
$\cup_{j=1}^\alpha \im (\widetilde{c}^{-C_{j-1}})\subset W_{0,p_g(\tX)}$, then for any
 $j\in \{0, \alpha+1, \alpha+2,\ldots, m+1\}$
the following facts are equivalent:

\vspace{1mm}

(f-a) $\widetilde{l}= C_{j-1}$;

(f-b) $\calL\in \ker(\pi_j)\setminus \ker(\pi_{j-1})$;


(f-c) $\calL\in  W_{0, p_g(\tX_j)}$.



\end{theorem}

\begin{proof}
The proof of Theorem \ref{th:W1} can be addapted. Let us prove the statements valid for $l\leq \alpha$.

Since the restriction $\pic^{-C_{j-1}}(\tX)\to \pic^0(C'_j)$ is an affine isomorphism,
$\{\calL: \calL(-C_{j-1})|_{C'_j}=\calO_{C'_j}\}$ is a point. Since
$\im (\widetilde{c}^{-C_{j-1}})\subset
\{\calL: \calL(-C_{j-1})|_{C'_j}=\calO_{C'_j}\}$,
the two (one cardinality) sets must agree. Furthermore,
$\{\calL; \widetilde{l}=C_{j-1}\}\subset \{\calL: \calL(-C_{j-1})|_{C'_j}=\calO_{C'_j}\}$ by Lemma
\ref{lem:INEQj}{\it (b)}. However, if $\calL\in \im (\widetilde{c}^{-C_{j-1}})$ then $\calL(-C_{j-1})\in
\im (c^{-C_{j-1}})$, hence $\widetilde{l}\leq C_{j-1}$.

Next, $\cup_{j=0}^\alpha \{\calL; \widetilde{l}=C_{j-1}\} \subset W_{0,p_g(\tX)}$ by
Lemma \ref{lem:INEQj}{\it (b)} and Theorem \ref{th:Okuma}. Finally, assume that $\calL\in W_{0,p_g}$.
Then, by definition, $h^1(\calL)=p_g$. On the other hand, assume that
its $\widetilde{l}$ is some $C_{i-1}$. Then by Lemma \ref{lem:INEQj}{\it (b)} $h^1(\calL)=p_g(\tX_i)$, hence $i\leq \alpha$ and $\calL\in \{\calL: \widetilde{l}\leq C_{\alpha-1}\}$.

The second case $j\in \{0, \alpha+1,\ldots, m+1\}$ follows analogously as the proof of the Gorenstein case, once we replace the statement {\it (d)} of Theorem \ref{e220} with
Theorem \ref{th:Okuma}.
\end{proof}
\begin{remark}
(a) If $l'=0$ then the $\{\overline{W_{0,k}}\}_k$ stratification is determined by a flag $\{0\}\cup \{
 \overline{ \im (\widetilde{c}^{-C_{j-1}}(Z)) }\}_{j\geq \alpha+1}$ of linear vector spaces
 whose dimensions are increasing one-by-one from 0 to $p_g$, and also by several
 {\it `wandering points'} $\cup_{j=1}^\alpha \im (\widetilde{c}^{-C_{j-1}})$, all of them being in
 $ W_{0,p_g(\tX)}$.
 All the irreducible components of $\overline{W_{0,k}}$ are affine subspaces of type
 $\overline{ \im (\widetilde{c}^{-C_{j-1}}(Z)) }$ ($0\leq j\leq m+1$), if $k<p_g$ then they are even linear subspaces.
 The non--linear ones are all points.

 (b) See Appendix for a non--Gorenstein singularity with $m=1$,
 $\alpha =1$, where there is a wandering point $\calL_1\not=0$ (cf. \ref{bek:A8}--\ref{bek:A10}).
\end{remark}

\bekezdes {\bf Questions:}
(1) Are the wandering points all distinct? Are they all different than 0 ?
(That is, is the cardinality of $\cup_{j=0}^\alpha \im (\widetilde{c}^{-C_{j-1}})$ exactly $\alpha+1$?)

(2) Are the wandering points in $\overline { W_{0, p_g-1}}$? Or, do they wander `even more' ?

(3) How the position of the wandering points reflect the variation of the analytic structure
of $(X,o)$ ?
\begin{remark}
For an index subset $I\subset \calv$ consider the set of all Chern classes $l'$ with
$I(l')=I$, and also the corresponding affine subspaces
$\overline{ \im (\widetilde{c}^{l' })}$ (indexed by $l'$ with $I(l')=I$).
Among them only a few might serve as irreducible component of some $\overline{W_{0,k}}$.
See for example in \ref{bek:A10} the set of points $\{\calL_n\}_{n\geq 0}$, among which only
$\{0,\calL_1\}$ are components of $W_{0,1}$. (Note that in this case $h^1(\calL_n)=0$ for all $n\geq 2$.)
For further role of closures of Abel images $\overline{ \im (\widetilde{c}^{l' })}$ in even  finer stratifications see \S\ref{s:basestrat}.
\end{remark}
\bekezdes {\bf Question:} Can those members  of $\overline{ \im (\widetilde{c}^{l' })}$,
which serve as components of some $\overline{W_{0,k}}$, be characterized by some universal property?
Are they characterized by the maximality of $h^1$ ?

\subsection{The case $(X,o)$ arbitrary elliptic  and $l'\in -\calS'$ arbitrary}\label{ss:Gorl'}

For any $\tX$ and $\calL$ with $c_1(\calL)\in -\calS' $, by Lemma \ref{lem:szk},
$h^1(\tX, \calL)=h^1(Z_K-s_{[Z_K]}, \calL)$. In this way the $h^1$--stratification
of an arbitrary elliptic singularity is reduced to the case of numerically Gorenstein ones.

Furthermore, if $(X,o)$ is numerically Gorenstein, then by Proposition
\ref{prop:restrict} the $h^1$--stratification is reduced to the $l'=0$ case.
If $\tX_i$ is Gorestein then one has to combine Proposition \ref{prop:restrict} with
Theorem \ref{th:W1}, otherwise Theorem \ref{th:W1} should be replaced by the more general
Theorem \ref{th:W1b}.

We invite the reader to complete the details writing down the corresponding set--identities.

\begin{remark}  Recall in the numerically Gorenstein case the identity
$W_{l',k}(\tX)=(\pi_i^{l'})^{-1}(W_{0,k}(\tX_i))$ (for the notation and the statement see \ref{ss:Redto} and
Proposition \ref{prop:restrict}). On the other hand, by Theorem \ref{th:W1b},
$\overline{ W_{0,k}(\tX_i)}$ equals $\overline{ \im (\widetilde{c}^{-C_{j-1}(\tX_i)})}$
for some cycle $C_{j-1}(\tX_i)\in L(\tX_i)$ associated with the singularity $(X_i,o_i)$. We show that
some similar structure statement is valid for $\overline{ W_{l',k}(\tX)}$ too, that is,
{\bf $\overline{ W_{l',k}(\tX)}$ is the closure of the image of a certain Abel map} (at the level of $\tX$).

First, let us shift $W_{l,k}(\tX)$ into $\pic^0(\tX)$ (where the images of modified Abel maps
$\widetilde{c}$ live). That is,
for each $l'$, via identification $\pic^{l'}\stackrel{\calO(-l')}{\longrightarrow} \pic^0$, we transport
$W_{l',k}$ into $W^0_{l',k}:=\calO(-l')\cdot W_{l',k}\subset \pic^0(\tX)$. In other words,
\begin{equation}\label{eq:Wshifted}
W^0_{l',k}:=\{\calL^0\in \pic^0(\tX)\,:\, h^1(\calL^0(l'))=k\}\subset \pic^0(\tX).\end{equation}
Its closure in $\pic^0(\tX)$ will be denoted by $\overline{W_{l', k}^0}$. Note also that
$W^0_{0,k}=W_{0,k}$.

Consider first the notations and situation from Proposition \ref{prop:restrict}.
First we analyze the Abel map $\im ( c^{-l'}):\eca ^{l'}(\tX)\to \pic^{l'}(\tX)$.  By Theorem
\ref{th:Vk} one gets $\dim \im (c^{l'})=\dim V(I(l'))$. On the other hand,
the dimension of $(\pi_i^{l'})^{-1}(0)$ in $\pic^{l'}(\tX)$ is $p_g(\tX)-p_g(\tX_i)=\dim V(I(l'))$ too,
cf. the same Theorem \ref{th:Vk}. In particular, $\overline {\im (c^{l'})}=(\pi_i^{l'})^{-1}(0)$
in  $\pic^{l'}(\tX)$.

Next, write $C_{j-1}(\tX_i)\in L(\tX_i)\subset L'(\tX_i)$ as $\sum _{v} a_vE^*_v(\tX_i)$ in
$\calS(\tX_i)\subset L'(\tX_i)$, and define its lift $C_{j-1}^{lift}:=\sum _{v} a_vE^*_v(\tX)$
into $\calS'(\tX)\subset L'(\tX)$. One sees that the restriction of $C_{j-1}^{lift}$ is exactly
$C_{j-1}(\tX_i)$, hence $\overline { \im ( c^{-C_{j-1}^{lift}})}$ restricted projects exactly onto
$\overline { \im ( c^{-C_{j-1}(\tX_i)})}$. This shows that
$\overline { \im ( c^{l'-C_{j-1}^{lift}})}$, as a subspace of $\pic^{l'-C_{j-1}^{lift}}(\tX)$,
shifted by $\calO_{\tX}(C_{j-1}^{lift})$ into $\pic^{l'}(\tX)$, is exactly
$\overline {W_{l',k}(\tX)}$ (cf. (\ref{eq:addclose})). Therefore,
\begin{equation}\label{eq:imctilde}
\overline { \im ( \widetilde{c}^{l'-C_{j-1}^{lift}})}=\overline { W^0_{l',k}(\tX)}.\end{equation}
For another, more `theoretical' presentation of $\overline { W^0_{l',k}(\tX)}$ as
$\overline { \im ( \widetilde{c}^{l'-l})}$ (with certain additional properties)
 see Theorem \ref{th:structure}.
\end{remark}

\section{The stratification of $\pic^{l'}$ according
to the base components}\label{s:basestrat}
\subsection{Notations} Fix $l'\in -\calS'$ as above.
In the previous section we considered the stratification of $\pic^{l'}(\tX)$
 provided by the value $\calL\mapsto h^1(\calL)$. Now we will consider a much `finer' stratification.
 Again, it is convenient to shift the structure into $\pic^0(\tX)$, these has
 the advantage that the strata  can be compared more
naturally with  subspaces of type $\overline{ \im (\widetilde{c}^{l''})}\subset \pic^0(Z)=
\pic ^0(\tX)$ $(Z\gg 0)$.

The strata are defined as follows (for the definition of $W_{l',k}$ see \ref{def:BN}):
$$F_{l',k}(l):=\{ \calL \in W_{l',k}\,:\, \mbox{the cycle of fixed components of} \
\calL \ \mbox{is $l\in L_{\geq 0}$}\}.$$
This shifted (via identity $\calL=\calL^0(l')$) provides a
stratification of $W^0_{l',k}$ (cf. (\ref{eq:Wshifted}))  as well:
$$F^0_{l',k}(l):=\{ \calL^0 \in W^0_{l',k}\,:\, \mbox{the cycle of fixed components of} \
\calL^0(l') \ \mbox{is $l\in L_{\geq 0}$}\}\subset \pic^0(\tX).$$

We denote by $\mathfrak{I}$ the set $\{\overline{ \im (\widetilde{c}^{l''})}\}_{l''}$
indexed by all possible $l''\in - \calS'$.
Our goal is to describe for the fixed  $l'\in -\calS'$ the sets
$\{F^0_{l',k}(l)\}_{k\in\Z_{\geq 0}, l\in L_{\geq 0}}$ in terms of certain elements of $\mathfrak{I}$.
This automatically will provide a new characterization of the
sets $\{W^0_{l',k}\}_{k\in \Z_{\geq 0}}$  as well, besides the one provided in the previous section.
Though the next theorem has some overlaps with statements form the previous section
regarding the $W$--stratification, we prefer this presentation since it provide a uniform presentation of the two type of stratification showing their interactions. (Even more, we deliberately use a
formulation and proof independent from section \ref{ss:BN} with the hope that this version
 can serve as a prototype for arbitrary cycle $Z$, not necessarily
 $Z\gg 0$, or  for more general singularities.)

For the  fixed $l'\in-\calS'$ and
$k\in\Z_{\geq 0}$ we define
$\mathfrak{I}_{l',k}$ by decreasing induction as follows.
For $k>p_g$ we set $\mathfrak{I}_{l',k}=\emptyset$ (note that by Theorem \ref{prop:vanishing}
we know that $W_{l',k}=\emptyset$ for $k>p_g$; check also that  the identity (\ref{eq:compat}) from below
has no solution in these cases). Assume next that $\mathfrak{I}_{l',k' }$  is already defined for any $k'>k$. Then, by definition,  $\mathfrak{I}_{l',k}$
consists  of all set of subspaces of type $\overline{ \im (\widetilde{c}^{l'-l})}$
of $\mathfrak{I}$ indexed by
\begin{equation}\label{eq:compat}\left\{\begin{array}{ll}
 \ &(i) \ \ l\in L_{\geq 0} \ \ \mbox{such that} \ \  l'-l\in -\calS',\\
 \ &(ii) \ \ k+\chi(l)+(l,l')=p_g-\dim V(I(l'-l)), \\
 \ &(iii) \ \  \mbox{$\overline{ \im (\widetilde{c}^{l'-l})}$ is not
 included in any subspace indexed
from $\cup_{k'>k}\mathfrak{I}_{l',k'}$.} \end{array}\right.
\end{equation}

Though the set $\mathfrak{I}$ can be infinite (see e.g. Example \ref{ex:NOECC}), each  set
$\mathfrak{I}_{l',k}$ is finite. Indeed, by (\ref{eq:compat}){\it (ii)} $\chi(l)+(l,l')$ is bounded,
hence by the negative definiteness of the intersection from, all the possible $l$ cycles
 are sitting in a finite ellipsoid and constitute a finite set.

\begin{theorem}\label{th:structure} {\bf (Structure Theorem of the $h^1$-- and
`base--component'--stratification)} \
Fix $Z\gg 0$ and $l'$, $k$ as above.

(a) Each irreducible component of  $\overline{W^0_{l',k}}$ and  $\overline{F^0_{l',k}(l)}$
is an affine subspace of $\pic^0$, in fact, it is an element of $\mathfrak{I}$.

(b)  The  irreducible component of $\overline{W^0_{l',k}}$ are the maximal elements of
$\mathfrak{I}_{l',k}$ (with respect to the inclusion).  Moreover,
$$W^0_{l',k}=\overline{W^0_{l',k}}\ \setminus \ \cup_{k'>k} \overline{W^0_{l',k'}}.$$

(c) An irreducible component of $\overline{W^0_{l',k}}$ might contain several subspaces  of type  $\overline{F^0_{l',k}(l)}$, there is a maximal one which equals it.
Any subspace of type  $\overline{F^0_{l',k}(l)}$ is nonempty if and only if
$\overline{ \im (\widetilde{c}^{l'-l})}$ belongs to $\mathfrak{I}_{l',k}$,
and in such a  case $\overline{F^0_{l',k}(l)}$ is irreducible and equals
$\overline{ \im (\widetilde{c}^{l'-l})}$ of dimension $\dim V(I(l'-l))$.
(Note that for each $l$ there exists exactly  one such subspace.)

 In particular, the collection of subspaces
 $\overline{F^0_{l',k}(l)}$ in $\overline{W^0_{l',k}}$
coincide with the set of affine  subspaces $\overline{ \im (\widetilde{c}^{l'-l})}$ indexed by $\mathfrak{I}_{l',k}$. The maximal ones fill in the irreducible components of $\overline{W^0_{l',k}}$,
the other ones are proper affine subspaces of these irreducible components.

If $\overline{F^0_{l',k}(l)}=\overline{ \im (\widetilde{c}^{l'-l})}$ for such an $l$ then
\begin{equation}\label{eq:F1}
F^0_{l',k}(l)= W^0_{l',k}\cap \im (\widetilde{c}^{l'-l}). \end{equation}
\end{theorem}

\begin{proof}
Parts {\it (b)} and {\it (c)} imply {\it (a)}. We start to prove {\it (b)}.
Note also that during the proof all the appeared cycles $l$ sit  in the bounded ellipsoid
$\{l:\ \chi(l)+(l,l')\leq p_g\}$, hence we can assume that not only $Z\gg0$ but all the possible
cycles of type $Z-l$ are also `large' (so, both $Z$ and $Z-l$ can be replaced by $\tX$ in
$h^1$--computations, if we wish).
Hence, sometimes we will omit $Z$ or $Z-l$.

{\bf (I)}
Let $S$ be an irreducible component of $\overline{W^0_{l',k}}$, and choose
 $\calL^0\in S\cap W^0_{l',k}$. Then
 $\calL:=\calL^0(l')\in W_{l',k}$  and it satisfies
 $h^1(\calL)=k$. Assume that $l\in L_{\geq 0}$ is the cycle of fixed components of
$\calL$, hence $H^0(Z,\calL)=H^0(Z-l,\calL(-l))$ and $H^0(Z-l,\calL(-l))_{reg}\not=\emptyset$. Then, by (\ref{eq:Chernzero}) necessarily $l'-l\in-\calS'$ and $\calL(-l)\in \im (c^{l'-l})$ (or,
$\calL^0\in \im(\widetilde{c}^{l'-l})$). From the exact sequence
(whenever $l>0$)
$0\to \calL(-l)|_{Z-l}\to \calL|_Z\to \calL|_l\to 0$ we get that $\chi(\calL|_l)-h^1(\calL(-l))
+h^1(\calL)=0$, or $h^1(\calL(-l))=k+\chi(l)+(l',l)$. On the other hand, by Theorem
\ref{th:Vk}{\it (8)} we have $h^1(\calL(-l))= p_g-\dim V(I(l'-l))$, hence $l$ satisfies
(\ref{eq:compat}){\it (ii)} as well.

Since (\ref{eq:compat}){\it (ii)}
 has no solution for $k>p_g$, we get that in such cases
$\overline{W^0_{l',k}}=\emptyset$,
and the choice of $\mathfrak{I}_{l',k}=\emptyset$ is also supported.
Then we prove {\it (b)} by decreasing induction on $k$. Fix again $k\leq p_g$ and
assume that the statement is already proved
for all $k'$ with $k'>k$. Consider again the situation from the previous paragraph:
$S$ be an irreducible component of $\overline{W^0_{l',k}}$, $\calL^0\in S\cap W^0_{l',k}$ and $\calL=\calL^0(l')$, $h^1(\calL)=k$.
Then we verified that there exists $l\in L_{\geq 0}$
so that
$\calL^0\in  \im(\widetilde{c}^{l'-l})$,
$l'-l\in- \calS'$ and  it satisfies  (\ref{eq:compat}){\it (ii)}. Note that the subspace
$\overline{ \im (\widetilde{c}^{l'-l})}$ cannot be included in any subspace index by
any $\mathfrak{I}_{l',k'}$ with $k'>k$ since by inductive step all the subspaces
indexed by $\mathfrak{I}_{l',k'}$ belong to $\cup_{k'>k}\overline{W^0_{l',k'}}$, hence all
their elements ${\mathcal K}^{0}$ satisfy  $h^1({\mathcal K}^{0}(l'))>k$; however $\calL^0
\in \overline{ \im (\widetilde{c}^{l'-l})}$ with $h^1(\calL^0(l'))=k$.
  Therefore, $l$ belongs to $\mathfrak{I}_{l',k}$.

{\bf (II)}
Now, by taking $\calL^0$ generic in $S$, the inclusion $\calL^0\in\im (\widetilde{c}^{l'-l})$ implies
$S\subset \overline{ \im (\widetilde{c}^{l'-l})}$, where the subspace
$\overline{ \im (\widetilde{c}^{l'-l})}$
is indexed from $\mathfrak{I}_{l',k}$.

{\bf (III)}
Conversely, consider some $\bar{l}\in L_{\geq 0}$ such that the subspace
$\overline{ \im (\widetilde{c}^{l'-\bar{l}})}$ is indexed from $\mathfrak{I}_{l',k}$
and $S\subset \overline{ \im (\widetilde{c}^{l'-\bar{l}})}$.
Let $\calk^0$ be a generic bundle  of $\overline{ \im (\widetilde{c}^{l'-\bar{l}})}$, and write
$\calk:=\calk^0(l')$. From Theorem \ref{th:Vk}{\it (8)}
\begin{equation}\label{eq:K}
h^1(\calk(-\bar{l}))=p_g-\dim V(I(l'-\bar{l})).\end{equation}
By a computation $\chi(\bar{l})+(l',\bar{l})=\chi(Z,\calk)-\chi(Z-\bar{l},\calk(-\bar{l}))$, and the right hand side also equals
$$h^0(Z, \calk)-h^1(Z,\calk)-
h^0(Z-\bar{l}, \calk(-\bar{l}))+h^1(Z-\bar{l},\calk(-\bar{l})).$$
This combined with (\ref{eq:compat}) and (\ref{eq:K}) give
\begin{equation}\label{eq:K2}
h^1(\calk^0(l'))=h^1(\calk)=h^0(Z, \calk)-
h^0(Z-\bar{l}, \calk(-\bar{l}))+k\geq k.
\end{equation}
We claim that necessarily $h^1(\calk^0(l'))= k$.
Indeed, since $\overline{ \im (\widetilde{c}^{l'-\bar{l}})}$  contains the  bundle
$\calL^0$ with $h^1(\calL^0(l'))=k$ (cf. {\bf (I)-(II)}), its generic bundle $\calk^0$ cannot satisfy
$h^1(\calk^0(l'))>k$ by the semicontinuity of $h^1$.
Hence, the generic element of $\overline{ \im (\widetilde{c}^{l'-\bar{l}})}$
belongs to $W^0_{l',k}$, which implies that $\overline{ \im (\widetilde{c}^{l'-\bar{l}})}= S$.
This in particular also shows, cf. (\ref{eq:K2}), that the cycle of fixed components of $\calk$ is $\bar{l}$.
Finally note
that $l$ is maximal in $\mathfrak{I}_{l',k}$. Indeed, assume that
there exists an overset of type $\overline{ \im (\widetilde{c}^{l'-\bar{l}})}$, then by the above discussion  $\overline{ \im (\widetilde{c}^{l'-\bar{l}})}$
equals $S$ too, hence must equal  $\overline{ \im (\widetilde{c}^{l'-l})}$ as well.

This ends the proof of part {\it (b)}. Next we prove {\it (c)}.
We will repeat several steps of the proof of {\it (b)}, but now applied for
the irreducible components of $\overline{F^0_{l',k}(l)}$.

{\bf (IV)}
Let $S$ be an irreducible component of $\overline{F^0_{l',k}(l)}$, and choose
 $\calL^0\in S\cap F^0_{l',k}(l)$. Set
 $\calL:=\calL^0(l')\in W_{l',k}$, hence $h^1(\calL)=k$.
 Next,   assume that $l\in L_{\geq 0}$ is the cycle of fixed components of
$\calL$. Then,  similarly as in {\bf (I)},
$\calL^0\in \im (\widetilde{c}^{l'-l})$ ($\dag$),
 $l'-l\in-\calS'$, $l$ satisfies (\ref{eq:compat}), and
 $\overline{\im (\widetilde{c}^{l'-l})}$ belongs to $\mathfrak{I}_{l',k}$.

By taking $\calL^0$ generic in $S\cap F^0_{l',k}(l)$ we get $S\subset
\overline{\im (\widetilde{c}^{l'-l})}$.

Conversely, as in {\bf (III)} for $\bar{l}=l$, one shows that $\overline{\im (\widetilde{c}^{l'-l})}
\subset S$ too, hence necessarily $S= \overline{\im (\widetilde{c}^{l'-l})}$.
Since for fixed $l',k$ and $l$ there is a unique affine subspace of type
$\overline{\im (\widetilde{c}^{l'-l})}$ with these data,
 $\overline{F^0_{l',k}(l)}$ should have only  one irreducible component, which equals
 $\overline{\im (\widetilde{c}^{l'-l})}$.

 Note that from ($\dag$) we also have
$ F^0_{l',k}(l)\subset W^0_{l',k}\cap \im (\widetilde{c}^{l'-l})$.
 The opposite inequality also follows as above (or as in {\bf (III)})
 since in the presence of $W^0_{l',k}$  we automatically have $h^1(\calk^0(l'))=k$.
 This shows (\ref{eq:F1}) as well.

{\bf (V)} In {\bf (IV)} we proved that each  $\overline{F^0_{l',k}(l)}$ is irreducible and equals
some $\overline{\im (\widetilde{c}^{l'-l})}$ from $\mathfrak{I}_{l',k}$.
Next we plan to show that any subspace from $\mathfrak{I}_{l',k}$ is realized in this way
by some  $\overline{F^0_{l',k}(l)}$.
(In fact, in this step we really exploite the `support condition' {\it (iii)}
from (\ref{eq:compat}).)  We proceed as in {\bf (III)}.

Fix $\overline{\im (\widetilde{c}^{l'-l})}$ from $\mathfrak{I}_{l',k}$.
Let $\calk^0$ be a generic bundle  from  $\im (\widetilde{c}^{l'-l})$, set
$\calk:=\calk^0(l')\in \pic^{l'}$.
Then (\ref{eq:K}) is still valid, and as in {\bf (III)} one also has
$$h^1(\calk)=h^0(Z, \calk)-
h^0(Z-l, \calk(-l))+k\geq k.$$
 We claim that $h^1(\calk)= k$. Assume that this is not the case, that is,
 $k':=h^1(\calk)>k$.    Let $\bar{l}$ be the cycle of fixed components of $\calk$.
 Then, as above, $\calk\in \overline{\im (\widetilde{c}^{l'-\bar{l}})}$ with $h^1(\calk)=k'$.
 By the inductive step, we can assume that $ \overline{\im (\widetilde{c}^{l'-\bar{l}})}$
  is indexed from $\mathfrak{I}_{l',k'}$. Since $\calk^0$ was chosen generically from
  $\overline{\im (\widetilde{c}^{l'-l})}$, we get that $\overline{\im (\widetilde{c}^{l'-l})}$
  is included in some space of type $\overline{\im (\widetilde{c}^{l'-\bar{l}})}$
   from $\mathfrak{I}_{l',k'}$, a contradiction.

Hence $h^1(\calk)=k$, $h^0(Z, \calk)=h^0(Z-l, \calk(-l)) $  and $\calk^0(l'-l)\in \im (c^{l'-l})$.
That is,  $\calk^0\in F^0_{l',k}(l)$ for a generic bundle  $\calk^0$
of  $\im (\widetilde{c}^{l'-l})$.
\end{proof}

\begin{remark}\label{rem:explicit}
The explicit determination of the index set $\mathfrak{I}_{l',k}$ --- even in concrete
examples ---  is not trivial at all.
The system  (\ref{eq:compat})  is not totally
combinatorial, it depends on the analytic structure (on the choice of $p_g$). But, even if we fix
$p_g=m-\alpha+1$ (between the possible  topological values $1$ and $m+1$), a fact which makes
$\dim V(I(l'-l))$ topological as well (cf. Theorem \ref{th:Vk}{\it (5)}), the
 list of solutions of the combinatorial
equation $\chi(l)+(l,l')+ \dim V(I(l'-l))=c$ is still  hard. We consider it as a real challenge
(see also subsection \ref{ss:challenge}  and the two examples after it).

Furthermore, the
description/characterization of the non--closed sets
$\im (\widetilde{c}^{l'-l})$ (indexed by $\mathfrak{I}_{l',k}$), respectively of
$F^0_{l',k}(l)$, is even  harder.

\end{remark}

\bekezdes{\bf Problem.}  Is it true that $F^0_{l',k}(l)= \im (\widetilde{c}^{l'-l})$ (with the notations of Theorem
\ref{th:structure}) ?

\bekezdes\label{ss:challenge}
 In Example \ref{ex:Fstrata} we show that the $F$--stratification of a certain $W$ can be non--trivial, while
 Example \ref{ex:trivi} presents a case when the $F$--stratification
is trivial (based  an additional geometric argument).

\begin{example}\label{ex:Fstrata} Consider the elliptic graph from Appendix. It has $m=1$,
hence $p_g\leq 2$. The maximal value $p_g=2$ can be realized e.g. by the hypersurface singularity
$\{x^2+y^3+z^{17}=0\}$; see also Remark \ref{rem:17}.

Assume in the sequel that $p_g=2$. Furthermore, assume also that $l'=-Z_K$.
In this case by Kodaira type or Grauert--Riemenschneider vanishing  $h^1(\calL)=0$ for any
$\calL\in \pic^{-Z_K}(\tX)$, hence $W_{-Z_K,0}=\pic^{-Z_K}(\tX)=\C^2$.
In fact, from the point of view of Proposition \ref{prop:restrict} the situation is also trivial:
$p_g(\tX_i)=0$, hence $\pic^{-Z_K}(\tX)$ consists of a unique stratum , namely $W_{-Z_K,0}$.

On the other hand, we will see that the `fixed component' stratification is not trivial.

First notice that $\overline { \im (c^{l'})}$ is 2--dimensional, hence it is
$\pic^{l'}$, and along it $h^1=0$, hence $\overline { \im (c^{l'})}=W_{l',0}$. To find the $F$--stratification  we have to find the solutions for
$l\in L_{\geq 0}$ of the system $Z_K+l\in\calS$ and
$\chi(l)-(l,Z_K)=2-\dim V(I(-Z_K-l))$ ($\dag$).

One solution is $l=0$ which provides $\im (c^{l'})$. The other solution is $l=E_1$ (see Appendix for notation).
In this case $Z_K+E_1=2Z_{min}\in \calS$, and $\chi(l)-(l,Z_K)=\dim V(I(2Z_{min}))=1$, hence
($\dag$) is satisfied. We will show that these are the only solutions. Indeed, assume that $l>0$ is such solution.
Then $\chi(l)-(l,Z_K)=\chi(-l)>0$ (see the third paragraph in \ref{bek:A7}),
hence $\dim V(I(-Z_K-l))\leq 1$. But, since $Z_K+l>0$ and the singularity is Gorenstein,
$\dim V(I(-Z_K-l))\geq  1$ too.
On the other hand, if $\dim V(I(-Z_K-l))= 1$ then $Z_K+l=nZ_{min}$ for some $n\geq 2$.
Hence, $\chi(l)-(l,Z_K)=\chi(-l)=\chi(Z_K-nZ_{min})=\chi(nZ_{min})=n(n-1)/2$.
This shows that necessarily $n=2$.

In conclusion, $\C=\im (\widetilde{c}^{-2Z_{min}})=F^0_{-Z_K,0}(E_1)$ and
$\C^2\setminus \C=\im (\widetilde{c}^{-Z_K})=F^0_{-Z_K,0}(0)$.
\end{example}
\begin{remark}\label{rem:geometry}
Though the set of subspaces of type $\{\overline{ F^0_{l',k}(l)} \}_l$ is in bijection with
$\mathcal{I}_{l', k}$ (completely defined/described above),
sometimes, in order  to reduce the the possible candidate  solutions
of (\ref{eq:compat})  we can use some additional geometric restrictions as well
(which, by Theorem \ref{th:structure},  are automatically satisfied, but this fact  might  not be  so
transparent from (\ref{eq:compat})).
E.g., if $l$ is a solution, hence $\{\overline{ F^0_{l',k}(l)} \}_l$ is a non--empty stratum,
then necessarily
$\dim V(I(l'-l))=\dim\, \im (c^{l'-l})\leq  \dim \overline { W_{l',k}}$, and equality
$\dim V(I(l'-l))<\dim \overline { W_{l',k}}$ whenever
 $\overline{\im (c^{l'-l})}$ is a proper subspace of $\overline { W_{l',k}}$.
 See the next Example for such an argument.
\end{remark}
\begin{example}\label{ex:trivi}
Consider the minimal resolution graph of a Gorenstein elliptic singularity, and fix
 $v\in B_0\setminus B_1$.  Then $\im (c^{-E^*_v})$ has dimension 1 (cf. Theorem \ref{th:Vk}).
 Recall that $\calL\in \pic ^{-E^*_v}$ belongs to $\im (c^{-E^*_v})$ if and only if it has no fixed components. We show that $\overline {\im (c^{-E^*_v})}=\im (c^{-E^*_v})$, that is,
 if $\calL\in \overline {\im (c^{-E^*_v})}$, then $\calL$ has no fixed components. Since
 $h^1$ along $\overline {\im (c^{-E^*_v})}$ is $p_g-1$ (cf. Theorem \ref{th:Vk}{\it (8)}), this fact reads also as
 $W_{-E^*_v,p_g-1}=\overline{ W_{-E^*_v,p_g-1}}=F_{-E^*_v,p_g-1}(0)$. (This fact will be used in the sequel,
 e.g. in the proof of Lemma
 \ref{lem:nobar}.)

 Indeed, assume that $\calL\in \overline {\im (c^{-E^*_v})}\setminus \im (c^{-E^*_v})$, and
 let $l\in L_{\geq 0}$ be the cycle of fixed components of $\calL$. Then from the exact sequence
 $0\to \calL(-l)\to \calL\to \calL|_l\to 0$ (or from (\ref{eq:compat})) we get that $l+E^*_v\in\calS'$ and
 $p_g-1+\chi(l)+(l,-E^*_v)=p_g-\dim V(I(-E^*_v-l))$.
 Since $\chi(l)\geq 0$, $(l,-E^*_v)\geq 0$ and $\dim V(I(-E^*_v-l))$ is
 necessarily at least 1 ($E^*+l>0$ and $(X,o)$ is Gorenstein), we get that
 \begin{equation}\label{eq:*}
 l\in L_{\geq 0},\
 \chi(l)=0, \ (l,-E^*_v)=0, \ l+E^*_v\in \calS', \ \mbox{and $E^*$--support of $E^*_v+l$} \subset B_0\setminus B_1\end{equation}
 We claim that the only solution of (\ref{eq:*}) is $l=0$.
 To verify from the combinatorics of the graph that no $l>0$ can be a solution
 of (\ref{eq:*}) might be tedious. But geometrically we see it as follows: this $F$--strata
 is zero dimensional, hence the corresponding
 $\dim V(I(l'-l))=\dim\, \im (c^{l'-l})$ must be  zero,
 but we already know that $\dim V(I(l'-l))$ is 1.

\end{example}

\begin{example}\label{ex:continu} {\bf (Continuation of Example \ref{ex:NOECC})} \
Here we exemplify how the condition (\ref{eq:compat}){\it (iii)} might enter
in the picture. Consider the situation from Example \ref{ex:NOECC}, and set $l'=-E^*_4$.

From Theorem \ref{prop:vanishing}  we get for any $\calL\in \pic^{-E^*_4}$ one has
$h^1(\calL)=h^1(\calL|_{\tX_1})\leq 1$. In particular $\cali_{-E^*_4,k}=\emptyset$
for $k\geq  2$. Next, $W_{-E^*_4,1}=\overline { \im (c^{-E^*_4})}$ is irreducible of dimension 1. Moreover, by Example \ref{ex:trivi},
$\overline { \im (c^{-E^*_4})}= \im (c^{-E^*_4})$, hence $l=0$ along
 $\overline{W_{-E^*_4,1}}=W_{-E^*_4,1}$.

 Next, for $k=0$ we consider the system (\ref{eq:compat}). Hence
 $E^*_4+l\in\calS' $, $\chi(l)+(l, -E^*_4)=2-\dim V(I(-E^*_4-l))$. Here, for any solution,
 $\dim V\geq 1$ since the support of $E^*_4+l$ is non--trivial and the germ is Gorenstein.

 First assume that $\dim V=2$. Then $\chi(l)=(l,E^*_4)=0$ and $E^*_4+l\in \calS'$. This shows that $|l|\subset \calv\setminus v_4$ and $l\in \calS(\calv\setminus v_4)$ ($\dag$)
 with $\chi(l)=0$. We claim that the only solution is $Z_0=Z_{min}(\calv\setminus v_4)$.
 [Indeed, by ($\dag$), $l=Z_0+l_1$, $l_1\geq 0$, $\chi(l_1)=(l_1, Z_0)=0$, hence $l_1$ is supported on the $E_8$ rational subgraph with $\chi=0$, hence $l_1=0$.]
Hence $\overline { F^0_{-E^*_4,0}(Z_0)} = \overline { \im ( \widetilde{c}^{-E^*_4-Z_0})}$,
it has dimension 2, hence it is $\overline{ W^0_{-E^*_4,0}}=\pic^0$.

Next assume that $\dim V=1$ and $\chi(l)+(l, -E^*_4)=1$. $\dim V=1$ happens
exactly when $0\not=E^*_4+l=nE^*_4+mE^*_3$ for some $n,m\in \Z_{\geq 0}$. We invite the reader to verify that the only solution is $l=Z_{min}$.
However, $\overline { \im (\widetilde{c}^{-E^*_4-Z_{min}})} =
\overline { \im (\widetilde{c}^{-E^*_4})}$, and this last one is a member of $\cali_{-E^*_4,1}$. Hence $\overline { \im (\widetilde{c}^{-E^*_4-Z_{min}})}
\not \in\cali_{-E^*_4,0}$, and along $W_{-E^*_4,0}$ one has $l=Z_0=Z_{min}-E_4$.

Recall that $\calO_{\tX}(-E^*_4) \not\in \overline { \im (c^{-E^*_4})}=W_{-E^*_4,1}$. Hence
$h^1(\calO_{\tX}(-E^*_4))=0$. 
 In particular, the cycle of fixed components of $\calO_{\tX}(-E^*_4)$ is  $Z_{min}-E_4$.
\end{example}

\section{Elliptic singularities with WECC}\label{saturated}

\subsection{WECC for arbitrary singularities}\label{ss:saturated}
Recall, that by definition,
a minimal resolution $\tX$  of a normal surface singularity satisfies WECC if and only if
 $E^*_v\in \overline { \calS'_{im} }$ for any end--vertex, that is,
 if for some $n>0$ one has
 $\calO_{\tX}(-nE^*_v)\in \im (c^{-nE^*_v})$, see also \ref{ss:images}.
 By \cite[Prop. 9.2.2]{NNI} this happens exactly when $\overline { \calS'_{im} }=\calS'$,
 that is, for any $l'\in \calS'$ there is some $n>0$ so that $\calO_{\tX}(-nl')\in \im (c^{-nl'})$.
(Recall, cf. \ref{ss:AbelMap}, that $\calO_{\tX}(l')\in \im(c^{l'})$ means that
$\calO_{Z}(l')\in \im(c^{l'})$ for $Z\gg0$.)

 As a comparison, EEC for $\tX$, by definition, is given by the condition
 $\calO_{\tX}(-E^*_v)\in \im (c^{-E^*_v})$
 for any end--vertex $v$.  (This can also be compared with the criterion
 (\ref{eq:firstchar}) valid for elliptic singularities.)
  It is known (see e.g. \cite[(2.15)]{Ok} or \cite[5.27]{NCL})
 that ECC is closed by taking `sub-singularities'. This means the following. For any connected union
 $E_{I}:=\cup_{w\in I}E_w$, where $I\subset \calv$, take $\tX_I$ a convenient
small neighbourhood of $E_I$ in $\tX$, then $\tX_I$ --- as the resolution of
 $(X_I, o_I)$ --- satisfies ECC as well. The very same proof gives the following.

 \begin{lemma}\label{lem:WECCrestr}
 Fix any (not necessarily elliptic)  singularity and one of its resolutions $\tX$. Then
 {\rm WECC} of $\tX$ is closed by taking `sub-singularities'.
 \end{lemma}
We will use the same notations even if $E_I$ is not connected, in such cases
 $(X_I,o_I)$ is  a multigerm. We set $\{I_j\}_j$ for the connected components of $I$.

 \bekezdes\label{bek:ujszam}
  Before we state the next result, we warn the reader that, in general,
 the restriction to some $\tX_I$ of a natural line bundle of $\tX$ is not natural
 (that is, $\calO_{\tX}(l')|_{\tX_I}\not= \calO_{\tX_I}(R(l'))$, where $R$ is the Chern class restriction).
By the next statements we prove that an analytic structure is {\it free from this pathology
if and only if it satisfies  WECC}.

In order to test the fact that the restriction of any natural line bundle it is natural
is enough to verify that $\calO_{\tX}(E_v)|_{\tX_{I_j}}$ is natural for any vertex $v$ and $j$,
where $I:=\calv\setminus v$.  Indeed, first note that it is enough to test only integral cycles.
Next,   $\calO_{\tX}(E_u)|_{\tX_{I_j}}=\calO_{\tX_{I_j}}(E_u)$ for any $u\not=v$,
hence by additivity applied for $\calO_{\tX}(l)$ ($l\in L$) the claim follows. Amazingly,
this property fits perfectly with the WECC.

 \begin{lemma}\label{lem:WEECnatural} As above, consider any
   singularity and one of its resolutions $\tX$.
 Fix any vertex $v\in\calv$ and set $I:=\calv\setminus v$, $I=\cup_jI_j$. Then
 $$
 E^*_v \in \overline{  \calS'_{im}}
  \ \ \ \mbox{if and only if}
 \  \ \
 \calO_{\tX}(E_v)|_{\tX_{I_j}} \ \ \ \mbox{is natural for all $j$}.$$
 \end{lemma}
 \begin{proof} `$\Rightarrow$' \ Take $n\gg 0$ so that $nE^*_v\in L$ and write
 $nE^*_v=\sum_j l_j+m_vE_v$ ($\dag$), where $l_j\in L(\tX_{I_j})_{>0}$. The assumption guarantees that
 $\calO_{\tX}(-nE^*_v)\in \im(c^{-nE^*_v})$, hence the existence of a divisor $D\in \eca^{-nE^*_v}(\tX)$ with
 $\calO_{\tX}(-nE^*_v)=\calO_{\tX}(D)$. Therefore,
  $\calO_{\tX}(-nE^*_v)|_{\tX_{I_j}}=\calO_{\tX}(D)|_{\tX_{I_j}}=\calO_{\tX_{I_j}}$.
  This reads as
  $\calO_{\tX}(m_vE_v)|_{\tX_{I_j}}=\calO_{\tX_{I_j}}(-l_j)$, hence
  $\calO_{\tX}(E_v)|_{\tX_{I_j}}$ is natural.

`$\Leftarrow$' \ If $\calO_{\tX}(E_v)|_{\tX_{I_j}}$ is natural then
  $\calO_{\tX}(m_vE_v)|_{\tX_{I_j}}$ has the form $\calO_{\tX_{I_j}}(-l_j)$
  for some $m_v\gg 0$. This means that
  for a convenient large $n$, such that $nE^*_v$ has the form ($\dag$), we get
  that  $\calO_{\tX}(-nE_v^*)|_{\tX_{I}}$ is trivial.  Consider next the
  restriction  $\pi_I : \pic^{-nE^*_v}(\tX)\to \pic^0(\tX_I)$. Then $(\pi_I)^{-1}(0)$
  is an affine space of dimension $h^1(\calO_{\tX})-h^1(\calO_{\tX_I})$. The same is true for
  $\im (c^{-nE^*_v})$ for $n\gg 0$. Since  $\im (c^{-nE^*_v})\subset (\pi_I)^{-1}(0)$,
  the two spaces should agree. Thus, $\calO_{\tX}(-nE^*_v)\in
  \im (c^{-nE^*_v})$ for $n\gg 0$.
 \end{proof}
 \begin{corollary}\label{cor:WECC} {\bf (Characterization of WECC for arbitrary singularity)}
 Under the condition of Lemma
 \ref{lem:WEECnatural} the following facts are equivalent:

(a)  WECC for $\tX$;

(b)  $\calO_{\tX}(E_v)|_{\tX_{\calv\setminus v}}$ is natural for any end vertex $v$;

(c)  $\calO_{\tX}(E_v)|_{\tX_{I_j}}$ is natural for any $v\in \calv$ and $j$;

(d) $\calO_{\tX}(-l)|_{\tX_{\calv\setminus I}}\in \pic^0(\tX_{\calv\setminus I})$
is trivial for any $l\in \calS$ with $E^*$--support $I$;

(e) The restriction to any $\tX_I$ of any natural line bundle of $\tX$ is natural.

 \end{corollary}
 \begin{proof}
 Use Lemma \ref{lem:WEECnatural} and its proof and the comment from \ref{bek:ujszam}.
 \end{proof}

 \subsection{WECC for elliptic singularities. First consequences.}

In the elliptic case, cf. Theorem \ref{th:Vk},
 since $\im (c^{nl'})=\overline { \im (c^{nl'})}$ and
 $\calO_{\tX}(nl')\in \overline { \im ( c^{nl'})}\
 \Leftrightarrow \ 0 \in \overline { \im ( \widetilde{c}^{nl'})}\
  \Leftrightarrow \ 0 \in \overline { \im ( \widetilde{c}^{l'})}\
  \Leftrightarrow \ \calO_{\tX}(l')\in \overline { \im ( c^{l'})}$
 for $n\gg0$, we get the following.

 \begin{corollary}\label{cor:first}
 {\bf (First analytic   characterization of WECC  for  elliptic $\tX$)}
 \begin{equation}\label{eq:firstchar}
 {\rm WEEC}\ \Leftrightarrow\ \calO_{\tX}(l')\in \overline { \im (c^{l'})} \
 \mbox{for any $l'\in -\calS'$}. 
 \end{equation}
\end{corollary}

We invite the reader to review the definition of the
{\it analytic multivariable Poincar\'e series} $P(\bt)$, associated with a fixed resolution of a
normal surface singularity  e.g. from \cite{NPS,NCL}, or  \cite{CDGPs,CDGEq}, see also \cite[2.3.6]{NNI}.
Usually $P(\bt)$ is not topological. However, for singularities, which satisfy ECC $P(\bt)$ equals
the topological series $Z(\bt)$, cf. \cite{NCL}.

\begin{corollary}\label{cor:WECCGor}
(a) If  $\tX$ is elliptic, numerically Gorenstein and it satisfies WECC  then
it is Gorenstein too.  More generally,
if  $\tX$ is elliptic, non--numerically Gorenstein and it satisfies WECC  then
$(X_0,o_0)$  is Gorenstein.

(b) If  $\tX$ is elliptic   and it satisfies WECC  then
the analytic Poincar\'e  series $P(\bt)$ is determined by the resolution graph.
(The identity $P(\bt)=Z(\bt)$ will be proved later.)
\end{corollary}
\begin{proof} {\it (a)}
By part {\it (d)}  of Corollary  \ref{cor:WECC} $\calO_{\tX}(-Z_{B_0})|_{B_1}$ is trivial.
Using induction and Lemma \ref{lem:WECCrestr} we get that in fact
$\calO_{\tX}(-Z_{B_j})|_{B_{j+1}}$ are  trivial for all $0\leq j\leq m-1$. Then apply part
{\it (d')$\Leftrightarrow$(f)} of Theorem \ref{e220}. For the second part use the first part and Lemma
 \ref{lem:WECCrestr} again.

 {\it (b)}
First note that $(X_0,o_0)$ satisfies WECC (cf. Lemma \ref{lem:WECCrestr}), hence it is Gorenstein
 (by part (a)). In particular,  $p_g$ is topological. Furthermore,
it is known, see e.g. \cite[4.2]{NPS}, that $P(\bt)$ can be recovered from the dual resolution graph
of $\tX$ combined with the knowledge of the cohomology groups
$\{h^1(\calO_{\tX}(l'))\}_{l'\in -\calS'}$ of natural line bundles indexed by $-\calS'$.
However,  for each such $l'$ one has
$\calO_{\tX}(l')\in \overline {\im (c^{l'})}$
and $h^1$ along $\overline {\im (c^{l'})}$ is topological (apply Theorem \ref{th:Vk} for  a
Gorenstein singularity).
\end{proof}

\begin{remark}
Corollary \ref{cor:WECCGor}(a)  can be compared with the following statement.
Assume that the link of a singularity $(X,o)$ is a rational homology sphere. Then,
if $(X,o)$ is numerically Gorenstein and $\Q$--Gorenstein,   then it is Gorenstein.
(Recall that  weighted homogeneous singularities, or those which satisfy ECC, are
$\Q$--Gorenstein.)
\end{remark}
\begin{lemma}\label{lem:nobar}
(a) Assume that $(X,o)$ is elliptic, numerically Gorenstein and $\tX$ satisfies WECC
 (hence it is automatically Gorenstein,  cf. Corollary  \ref{cor:WECCGor}). If $v\in B_0\setminus B_1$
then $E^*_v\in \calS'_{im}$ (that is,
  $\calO_{\tX}(-E^*_v)\in \im (c^{-E^*_v})$).

(b) Assume that $(X,o)$ is elliptic, non numerically Gorenstein and $\tX$ satisfies WECC.
If $v\in B_{-1}\setminus B_0$
then $E^*_v\in \calS'_{im}$. Moreover, $\calO_{\tX}(-E^*_v)$ has no base  point.
\end{lemma}
\begin{proof}  In both cases, from  (\ref{eq:firstchar}) $\calO_{\tX}(-E^*_v)\in
\overline{ \im (c^{-E^*_v})}$. In case (a),  by Corollary \ref{cor:WECCGor} the singularity is Gorenstein, hence by Example \ref{ex:trivi}
$\overline{ \im (c^{-E^*_v})}= \im (c^{-E^*_v})$. In case (b), by Theorem
\ref{th:Vk}, $\im (c^{-E^*_v})$ is  a point, hence
$\overline{ \im (c^{-E^*_v})}= \im (c^{-E^*_v})$.
Therefore, in both cases $\calO_{\tX}(-E^*_v)\in \im (c^{-E^*_v})$.
In case {\it (b)} use again that  $\im (c^{-E^*_v})$ is  a point, hence any `moved'
effective Cartier divisor from $\eca^{-E^*_v}$ is the zero--set of a section.
\end{proof}
\begin{example}\label{ex:nonobar}
Example \ref{ex:NOECC} shows that the WECC in the above Lemma \ref{lem:nobar}
is necessary.
\end{example}
The previous lemma guarantees that in the elliptic WECC Gorenstein case
$\calO_{\tX}(-E^*_v)$ $(v\in B_0\setminus B_1)$
has no fixed components. The next result focuses on the possible base points.

\begin{theorem}\label{th:extension} {\bf (Elliptic
Gorenstein/WECC extension obstruction)} \
Assume that $(X,o)$ is elliptic and  Gorenstein with minimal resolution $\tX$.
Let us denote the dual graph by $\Gamma$ and we fix a vertex   $v\in B_0\setminus B_1$.

(a) Assume that $\Gamma$ can be extended to a larger elliptic graph $\Gamma'$ by adding a new vertex
 $w$ connected to $\Gamma$ by an edge $(v,w)$.
  Then $\calO_{\tX}(-E_v^*)$ does not admit  $E_v$ as its fixed component (though any other $E_u$, $u\not=v$,
  might be fixed, cf. Example \ref{ex:continu}), and  along $E_v$ it has  a unique   base point.

(b) Assume that the elliptic  $\Gamma'$ (obtained from $\Gamma$ as in (a))
is the dual graph of a resolution $\tX'$,
and $\tX$ can be identified with a small neighbourhood of\, $\cup_{v\in \calv(\Gamma)}E_v$ in
$\tX'$. If $\tX'$ satisfies WECC then  $\calO_{\tX}(-E_v^*)$  has
a unique base point (along $E$), which  is exactly   $p:=E_v\cap E_w$.

In particular, $\tX$ cannot be embedded as a subsingularity in an elliptic WECC
 $\tX''$ which has two additional irreducible exceptional curves intersecting
 $E_v$ transversally in two different points $p_1,p_2\in E_v\setminus \cup_{u\in \calv(\Gamma)\setminus v}E_u$.
\end{theorem}
\begin{proof} {\it (a)}
We consider the cohomological exact sequence associated with $0\to \calO_{\tX}(-E^*_v-E_v)\to \calO_{\tX}(-E^*_v)\to \calO_{E_v}(-E^*_v)\to 0$. Then $\calO_{E_v}(-E^*_v)\simeq \calO_{{\mathbb P}^1}(1)$, hence
$ H^0(\calO_{E_v}(-E^*_v)) =\C^2$ and $ H^1(\calO_{E_v}(-E^*_v)) =0$.
We will show that
$h^1( \calO_{\tX}(-E^*_v-E_v))=h^1( \calO_{\tX}(-E^*_v)) +1$.
Using the cohomological exact sequence, this is equivalent with the fact that
 the dimension of the image of $\rho: H^0(\calO_{\tX}(-E^*_v))\to H^0(\calO_{E_v}(-E^*_v))=\C^2$  is 1.
This shows that $E_v$ is not fixed (since $H^0(\calO_{\tX}(-E^*_v-E_v))\hookrightarrow H^0(\calO_{\tX}(-E^*_v))$
is not onto), $E_v$ supports a base point (since $\dim(\im\rho)=1$) and the base point is unique
(since $(E_v,-E^*_v)=1$).

First we handle $h^1( \calO_{\tX}(-E^*_v-E_v))$. Note that
$v$ is an end--vertex by Lemma \ref{lem:glue},
and $m_{E_v}(Z_{min})=m_{E_v}(Z_K)=1$. On the other hand, it is known,
cf. \cite[Proposition 4.3.3]{trieste},
that for any $l'\in L'$ there exists a unique $s(l')\in \calS'$
with $s(l')-l'\in L_{\geq  0}$, and
minimal with these two properties.
Then $h^1(\calO_{\tX}(-l'))$ and $h^1(\calO_{\tX}(-s(l')))$ can be compared topologically.
Indeed, there exists a Laufer type computation sequence $\{z_i\}_{i=0}^t$ so that $z_0=l'$,
$z_t=s(l')$, and $z_{i+1}=z_i+E_{v(i)}$ such that
$(z_i,E_{v(i)})>0$ for $i<t$. Then, from a long exact sequence we get
$h^1(\calO_{\tX}(-z_{i+1}))=h^1(\calO_{\tX}(-z_{i}))- h^1(\calO_{\mathbb {P}^1}(-(z_i,E_{v(i)})))$.

Now, in our case, one sees that $s(E^*_v+E_v)=E^*_v+Z_{min}$, and the above computation sequence is in fact
$E^*_v+z_i'$, where $\{z_i'\}_i$ is the computation sequence connecting $E_v$ with $Z_{min}$ (compare the
Laufer algorithms for the two cases and use the fact that $m_{E_v}(Z_{min})=1$). Hence, from ellipticity,  there exists exactly one step when
$(z_i,E_{v(i)})=2$ and in all other steps it is 1. Hence $h^1(\calO_{\tX}(-E^*-E_v))=h^1(\calO_{\tX}(-E^*_v-Z_{min}))+1$.

Next, in the case of both line bundles $\calO_{\tX}(-E^*_v)$ and $\calO_{\tX}(-E^*_v-Z_{min})$ the $E^*$--support is included in $B_0\setminus B_1$ (use the assumption and (\ref{eq:orthogonal})). Hence, by Theorem
\ref{prop:vanishing}
$h^1(\calO_{\tX}(-E^*_v))=h^1(\calO_{\tX}(-E^*_v)|_{\tX_1})$ and
$h^1(\calO_{\tX}(-E^*_v-Z_{min}))=h^1(\calO_{\tX}(-E^*_v-Z_{min})|_{\tX_1})$. But
$\calO_{\tX}(-Z_{min})|_{\tX_1}$ is trivial by the Gorenstein property, cf. Theorem \ref{e220}
(note that $\pic(\tX_1)=\pic(C'_1)$).
Hence $h^1(\calO_{\tX}(-E^*_v-Z_{min})=h^1(\calO_{\tX}(-E^*_v))$.

{\it (b)}
 Next, assume that $\tX'$ (hence, by Lemma \ref{lem:WECCrestr} $\tX$ too) satisfies WECC.
Then $\calO_{\tX}(-E^*_v)$ has no fixed components at all (cf. \ref{lem:nobar}).
Since $(E_u,-E^*_v)=\delta_{uv}$, it can have a base point
only along $E_v$, where it really has one by {\it (a)}.
Let the disc $E_w\cap \tX$ be $D_w$. Having the WEEC  for $\tX'$, we can choose a divisor
$D\in \eca(\tX')$, which intersect $E(\tX')$ only along $E_w\setminus E_v$,  and some integer $m$ such that
$\calO_{\tX'}(m(\iota(E^*_v)+E_w)+D)$ is trivial. (Here $\iota$ is the embedding
 $L(\tX)\otimes \Q\to
L(\tX')\otimes\Q$.)  Then
$\calO_{\tX'}(m(\iota(E^*_v)+E_w))|_{\tX}=\calO_{\tX}(E^*_v+D_w)^{\otimes m}$
is trivial too. Since $\pic(\tX)$ has no torsion, $\calO_{\tX}(E^*_v+D_w)$ is also trivial.
Hence there is a section of $\calO_{\tX}(-E^*_v)$ which vanishes  along $D_w$, hence  at $p$.
\end{proof}
\begin{remark}
An elliptic Gorenstein analytic structure does  not satisfy  necessarily WECC.
Indeed, take e.g.
the germ from Example \ref{ex:NOECC}, or the graph from the right hand side from
Example \ref{ex:nagygraf}: they do not satisfy the above WECC extension property.
On the other hand, the graph from the left hand side in  Example \ref{ex:nagygraf}
carries a WECC analytic structure by the next Theorem \ref{th:WECCCrit}.

Note also that the identity $P(\bt)=Z(\bt)$ characterizes ECC,
cf. \cite[Theorem 7.2.1]{ICM}. Therefore, for these Gorenstein but not WECC
singularities   $P(\bt)=Z(\bt)$ also fails.
\end{remark}

\subsection{First topological characterization of the existence of WECC structure}
\label{ss:Top1WECC}
In the following we will give a topological characterization
in terms of the combinatorics of
the  minimal resolution graph $\Gamma$ for the existence of a
 WECC analytic type supported on $\Gamma$.

\begin{theorem}\label{th:WECCCrit} {\bf (Extension Criterion of the elliptic sequence)}
Fix an elliptic minimal graph $\Gamma$ with elliptic sequence $B_{-1},
 \ldots , B_m$.  Then there exists a  singularity  with minimal resolution
 $\tX$ with dual  graph $\Gamma$, and which satisfies WECC, if and only if
for every $0\leq i \leq m$ and for any vertex $v \in B_i \setminus B_{i+1}$,
 $v$ has  at most one  neighbour in $B_{i-1}$.
\end{theorem}

\begin{proof}
The extension obstruction from Theorem \ref{th:extension} (applied via  Corollary \ref{cor:WECCGor})
shows that the combinatorial
restriction is necessary. Now, we fix a graph $\Gamma$, which satisfies the gluing obstruction,
and we wish to construct a WECC analytic type supported on it.
The construction builds a resolution space $\tX$  by analytic plumbing based on induction on $m$.
If $m=0$ then the graph is minimally elliptic, hence any analytic realization satisfies ECC \cite{NWsq},
hence WECC too.

Next, we assume that $\tX_i$ was already constructed, and it satisfies WECC.
Fix $v\in B_i$, which has a neighbour $w$ in $B_{i-1}$. By assumption, $v$ admits only one such $w$.
 By Lemma \ref{lem:glue} $v\not\in B_{i+1}$, and by Lemma
\ref{lem:nobar} $\calO_{\tX_i}(-E^*_v)\in \im (c^{-E^*_v})$ (here all invariants are
associated with $\tX_i$). Hence, there exists a divisor $D_w$ of $\tX_i$ such that
$\calO_{\tX_i}(-E^*_v)=\calO_{\tX_i}(D_w)$. The  Chern class shows that
$D_w$ is smooth and it intersects $E_v$ transversally and it intersects no other exceptional curve.
[The `Extension Theorem' \ref{th:extension}  and its proof show
that $D_w\cap E_v$ is uniquely determined  by the analytic type of $\tX_i$,
it is the base point of $\calO_{\tX_i}(-E^*_v)$.]
Then let $T_w$ be an  analytic disc bundle over $E_w$ with Chern number $E_w^2$,
and we analytically glue $T_w$ to $\tX_i$ in such a way that $\tX_i\cap E_w=D_w$.
We proceed similarly for all other such $w\in B_{i-1}\setminus B_i$ vertices,
 which have a neighbour
in $B_{i}$.  (We call such $w$  a contact vertex.)
The other disc bundles (corresponding to vertices $w\in B_{i-1}\setminus B_i$,
which have no neighbours in  $B_i$) are glued arbitrarily. The obtained resolution
space will be denoted by $\tX_{i-1}$.

We claim that $\tX_{i-1}$ supports a singularity with WECC. In the proof we use
Corollary \ref{cor:WECC}{\it (c)$\Rightarrow$(a)}. According to this, we need
to verify that
\begin{equation}\label{eq:need}\mbox{
$\calO_{\tX_{i-1}}(E_u)|_{\tX_{B_{i-1}\setminus u}}$ is natural
for any  vertex $u\in B_{i-1}$. }\end{equation}

First we prove a lemma. In order to formulate it, let us fix a connected
subgraph supported on $\overline{B}$ with $B_i\subset \overline{B}\varsubsetneq B_{i-1}$.
Note that the maximal
numerical Gorenstein support in $\overline{B}$ is $B_i$. [Indeed, $\overline{B}$ has a unique maximal numerically
Gorenstein subgraph with length $m+1-i$ by Remark \ref{rem:univ}, but $B_i$ satisfies
this requirement.]
\begin{lemma}\label{lem:trivii} Fix  $\calL\in \pic(\tX_{\overline{B}})$. Then
$\calL\in \pic(\tX_{\overline{B}})$ is natural  \ $\Leftrightarrow$ \
$\calL|_{\tX_i}\in \pic(\tX_i)$ is natural.
\end{lemma}
\begin{proof}
`$\Rightarrow$' \ Fix $n\gg 0$ so that $\calL^{\otimes n}=\calO_{\tX_{\overline{B}}}
(\sum n_uE_u)$ with $n_u\in\Z$. If $E_u\cap E_{B_i}=\emptyset$
then $\calO_{\tX_{\overline{B}}} (E_u)|_{\tX_i}$ is trivial,  if $u\subset B_i$
then $\calO_{\tX_{\overline{B}}} (E_u)|_{\tX_i}$  is obviously natural, and if  $u=w$ is a contact vertex then
$\calO_{\tX_{\overline{B}}} (E_w)|_{\tX_i}=\calO_{\tX_i}(D_w)$ is natural by construction.

`$\Leftarrow$' \ Note that the restriction $\pic^{l'}(\tX_{\overline{B}})
\to \pic^{R(l')}(\tX_i)$ is an isomorphism (here $l'\in L'(\tX_{\overline{B}})$,
and $R(l')$ is its  restriction). Now, if the restriction of $\calL\in \pic(\tX_{\overline{B}})$ is  natural,
then $\calL^n|_{\tX_i}=\calO_{\tX_i}(l)$ for some $n\in\Z$ and $l\in L(\tX_i)$.
Consider $\calO_{\tX_{\overline{B}}}(\iota(l))$, where $i:L(B_i)\to L(\overline{B})$ is the lattice embedding.
Then  $\calO_{\tX_{\overline{B}}}(\iota(l))|_{\tX_i}=\calO_{\tX_i}(l)=\calL^n|_{\tX_i}$, hence by the
injectivity of the restriction $\calL^n=\calO_{\tX_{\overline{B}}}(\iota(l))$.
\end{proof}
Now we verify (\ref{eq:need}). If $u\in B_{i-1}\setminus B_i$ then
$B_{i-1}\setminus u$ has a connected component $\overline{B}$ with
$B_i\subset \overline{B}\varsubsetneq B_{i-1}$, and maybe some other components, all of them
supporting rational graphs. Along the rational components any bundle is automatically natural. Then
$\calO_{\tX_{i-1}}(E_u)|_{\tX_{B_{i-1}\setminus u}}$ is natural by
Lemma \ref{lem:trivii}, since
its restriction to $\tX_i$ is natural (this last statement can be proved as the part
`$\Rightarrow$' of Lemma \ref{lem:trivii}).

Next, assume that $u\in B_i$. Let $j$ (where $m+1\geq j>i$)
be maximal so that $u\in B_{j-1}$. Then, similarly as in the previous case,
$\calO_{\tX_{i-1}}(E_u)|_{\tX_{B_{i-1}\setminus u}}$ is natural whenever its restriction
to $\tX_j$ is natural. (For $j=m+1$ this reads as follows: all the components are
 rational, hence the restricted bundle is natural.)
This follows from the WECC of $\tX_i$.
\end{proof}

\subsection{Further topological/analitical  characterizations of the  WECC structure}
\label{ss:Top2WECC}
In this subsection we will prove the following two statements:
if a minimal elliptic graph supports an analytic structure with WECC then it necessarily supports also one with
ECC. Even more, any analytic structure with WECC satisfies in fact ECC too.

We wish to separate sharply these two statements by the following reason.  Recall that the existence of an analytic structure with ECC is topological: it exists if and only if the graph either satisfies the semigroup and congruence
conditions of Neumann--Wahl \cite{NWsq}, or the monomial condition of Okuma \cite{Ok}. In this article we will
use the monomial condition (for definition see below). Hence, the first statement basically
is equivalent with the fact that a WECC elliptic singularity necessarily must satisfy the combinatorial
monomial condition. (The other direction is already in the literature: the monomial condition assures
the existence of a splice quotient analytic type \cite{Ok}, while splice quotients by their construction
satisfies ECC, hence WECC too.)

An immediate consequence of this is that the `old' combinatorial criterions, namely the semigroup--congruence condition, or the monomial condition,  for elliptic graph are equivalent with the existence of the gluing property from Theorem \ref{th:WECCCrit} (which is much easier to test!).

The second part is analytical in nature, it says that in the elliptic case
for {\it any} analytic structure already the WECC guarantees ECC. Recall that by the `End Curve Theorem'
\cite{NWECTh,OECTh}
the ECC is equivalent with splice quotient analytic type. Hence, a consequence of our next theorem is that
in the elliptic case the three notions --- splice quotient, WECC, ECC -- are equivalent.

\begin{definition}\label{def:mC} \cite{Ok} $\Gamma$ satisfies the {\it monomial condition} (MC) if for any node (rupture vertex) $v$ and any connected full subgraph $\Gamma_i$ of $\Gamma\setminus v$ there exists an effective cycle $C_i$ supported on $\Gamma_i$ such that $(E^*_v+C_i,E_u)=0$ for any $u\in \calv(\Gamma_i)\cup \{v\}$, which is not an end--vertex of $\Gamma$ sitting in $\calv(\Gamma_i)$.

[In fact, below, we will use only the `melody' of this definition: MC is satisfied iff
any node $v$ and any  $\Gamma_i$  satisfy some combinatorial property, which not necessarily should be specified.]
\end{definition}

\begin{theorem}\label{th:9.4.2} Fix an elliptic  minimal resolution graph $\Gamma$.

(1)  {\bf (Second topological characterization of the existence of an analytic structure with WECC)}
$\Gamma$ supports an analytic structure with WECC if and only if it satisfies MC.

(2) {\bf (Second analytic characterization of an analytic structure with WECC)}
Assume  that $\Gamma$ supports an analytic structure with WECC.
Then any such structure satisfies ECC too.
\end{theorem}

\begin{proof}
We will prove the two statements by simultaneous induction on the number of vertices $|\calv|$.  For minimally elliptic or rational graphs the statements are true,
 because any minimally elliptic or rational singularity is  splice quotient.
Thus,  assume that the  statements are valid for graphs with less than $k$ vertices, and assume, that $|\calv| = k$.

We claim that it is enough to prove {\it (1)}, because {\it (1)} implies  {\it (2)}.
Indeed,  if  $P(\bt) = \sum_{l' \in S'} p(l') \bt^{l'}$
is the analytic multivariable  Poincar\'e series then an analytic structure satisfies ECC if and only if
$p(E_v^*) = 1$ for every end vertex $v$ (this follows basically from the definition of $P$).
On the other  hand, if a WECC analytic structure exists, then all of them have the same $P(\bt)$ determined
topologically, cf. Corollary \ref{cor:WECCGor}{\it (b)}. By part {\it (1)} a structure with ECC also exists, for which
$p(E_v^*) = 1$. Since ECC is WECC too, for all WECC structures $p(E_v^*) = 1$. Hence any WECC is ECC.

In the sequel we focus on part {\it (1)},
 where $\Gamma$ is an elliptic minimal resolution graph  with $|\calv| = k$.
We assume the existence of an analytic structure $\tX$ with WECC (on $\Gamma$)
and we wish to prove MC.

Assume that MC fails at some node $v$ and branch $\Gamma_1$ of $\Gamma\setminus v$. Denote by $v_1,\ldots, v_\delta$
the adjacent vertices of $v$ in $\Gamma$ with $v_1\in \calv(\Gamma_1)$, $\delta\geq 3$. [In fact, by the inductive step,
we can
even assume that $\calv(\Gamma)=\calv(\Gamma_1)\cup \{v, v_2,v_3\}$, otherwise we take the subgraph with these
vertices, it is WECC by restriction, cf. \ref{lem:WECCrestr},
 it is ECC by induction, hence it satisfies MC at $\Gamma_1$, a contradiction.
But this reduction does not really help in the next proof.]
Besides $\Gamma_1$ we will consider several graphs. $\Gamma_1^v$ denotes the full subgraph $\Gamma_1\cup \{v\}$
of $\Gamma$.  $\Gamma_1^m$ is obtained from $\Gamma_1^v$ by modifying the decoration of $v$ by a very negative
integer $N\ll 0$. Furthermore, the `extended--modified'
$\Gamma_1^{me}$ is obtained from the full subgraph $\Gamma_1\cup\{v, v_2,\ldots, v_\delta\}$ of $\Gamma$ by replacing all decorations of $\{v, v_2,\ldots, v_\delta\}$ by $N$.

We claim that $\Gamma_1^{me}$ is elliptic. Indeed, if we take subgraphs or we decrease decorations of an elliptic graph
we get an elliptic or rational graph. However, $\Gamma_1^{me}$ has a node and branch for which MC fails,
so it cannot be rational (since rational singularities are splice quotient \cite{NWsq}).

\begin{lemma}\label{lem:proof}
(a)   Let $B_0=B_0(\Gamma^{me}_1)$ be the support of the maximal numerically Gorenstein subgraph in
$\Gamma^{me}_1$  (cf. Remark \ref{rem:univ}). Then $\{v_2,\ldots, v_\delta\}\cap B_0= \emptyset.$

(b) Both $\Gamma_1^m$ and $\Gamma_1^v$ are elliptic graphs.
\end{lemma}
\begin{proof}
{\it (a)}
Assume $v_2\in B_0$. Then by the uniqueness of the maximal numerically Gorenstein subgraph and by symmetry
we get $\{v_2,\ldots, v_\delta\}\subset B_0$, and  by the connectedness of $B_0$ we get $v\in B_0$ too.
By decreasing $N\ll 0$ the  fundamental cycle of any subgraph is non--increasing, hence $Z_K(B_0)=\sum _{i\geq 0}
Z_{min}(B_i)\in L$ is non--increasing too. In particular, it must stabilize to an
integral  cycle independent of $N$.
Let the coefficients of $E_v$ and $E_{v_j}$ be $m_v$ and $m_{v_j}$. Then by the adjunction formula applied for $v_j$
 we get that necessarily $m_{v_j}=1$ and $m_v=2$ ($j\geq 2$). But then the adjunction formula for
$v$ gives an $N$--dependent relation, hence a contradiction.

{\it (b)} Since $\Gamma_1^{me}$ is elliptic, part {\it (a)} shows that
$\Gamma_1^m=\Gamma_1^{me}\setminus \{v_2,\ldots, v_\delta\}$ contains the elliptic cycle,
hence $\Gamma_1^m$ is elliptic. $\Gamma_1^v$ being a subgraph of $\Gamma$ is either
elliptic  or rational; but it cannot be rational since then $\Gamma_1^m$ would  also
be rational.
\end{proof}

Starting from the analytic type $\tX$ we construct an analytic type supported on $\Gamma_1^{me}$ with ECC.
This will contradict the fact that for $\Gamma_1^{me}$ MC fails. The construction has several steps.

{\em Step 1.} Consider a tubular neighbourhood $\tX^v_1$ of
exceptional divisors indexed by $\Gamma^v_1$ in $\tX$. It can be considered as a resolution space. By Lemma
 \ref{lem:WECCrestr}  it satisfies WECC. Since $|\calv(\Gamma^v_1)|<k$, by the inductive step it satisfies ECC too.
In particular, since $v$ is an end--vertex of $\Gamma_1^v$,
$\calO_{\tX^v_1}(-E^*_v)$ has no fixed components.
Moreover, by WECC extension Theorem  \ref{th:extension}, $v\not\in B_0(\Gamma^v_1)$.
[Indeed,  $B_0(\Gamma_1^v)=\Gamma_1^v$ implies that $\Gamma_1^v$ is numerically Gorenstein, being numerically Gorenstein and WECC it is Gorenstein by \ref{cor:WECCGor}, this contradicts  \ref{th:extension} since graph can be extended
as WECC by more than two vertices $v_{j}$]. Therefore, by Lemma \ref{lem:nobar}
 $\calO_{\tX^v_1}(-E^*_v)$ has no base points either.

{\em Step 2.}  We claim that $v_0\not\in B_0(\Gamma_1^m)$.
Indeed, otherwise $\Gamma_1^m$ is numerically Gorenstein. Similarly as in the proof of Lemma \ref{lem:proof}, by decreasing $N\ll 0$ the integral cycle $Z_K(\Gamma_1^m)$ stabilizes, and by adjunction formula its $E_v$--multiplicity is 1, independently of $N$.
But then the very same integral cycle might serve as the  canonical cycle of $\Gamma_1^v$ too, which contradicts the fact that $\Gamma_1^v$ is not numerically Gorenstein (see
{\it Step 1}).

{\em Step 3.}  Fix a (generic)  section  $s\in H^0(\calO_{\tX^v_1}(-E^*_v))$, set ${\rm div}(s)=C$, which is a transversal
 smooth cut of $E_v$ in $\tX^v_1$. Blow up the infinitesimal close point $C\cap E_v$ several times (that is,
 blow up $C\cap E_v$ by creating the new exceptional curve $E_{new}$, then blow up the intersection of
 $E_{new}$ with the strict transform of $E_v$, etc.). We blow up so many times that
 in the created total space $Bl(\tX^v_1)$
 the strict transform of $E_v$ (still denoted by $E_v$) will have Euler number $N$.
 Identify the subgraph determined by the strict transforms of exceptional curves  by $\Gamma_1^m$, and let
 $\tX^m_1$ a tubular neighbourhood of them in $Bl(\tX^v_1)$.
 Since we have blown up $C\cap E_v$, where $C$ is the zero of an end curve function, $Bl(\tX^v_1)$
 remains ECC.
 Then, by restriction, $\tX^m_1$ is also ECC, cf. \ref{ss:saturated}, in particular
 $\calO_{\tX^m_1}(-E^*_v)$ has no fixed components.
  We claim that  along $E_v$ it has  no base point either.
  Indeed, since  $\calO_{\tX^v_1}(-E^*_v)$ has no base points (cf. {\it Step 1}),
  all the sections of this bundle lifted to $Bl(\tX^v_1)$ and restricted to
  $\tX^m_1$ provide sections of $\calO_{\tX^m_1}(-E^*_v)$.
  [The fact that $\calO_{\tX^m_1}(-E^*_v)$ has no base points also shows that
  $\Gamma_1^m$ is not numerically Gorenstein. Indeed, otherwise, being WEEC too,
  it would be Gorenstein with elliptic extension, hence with a base point of
  $\calO_{\tX^m_1}(-E^*_v)$  by \ref{th:extension}{\it (a)}.]

 {\em Step 4.} We fix (generic) sections $s_i\in H^0(\calO_{\tX^m_1}(-E^*_v))$,
 $2\leq i\leq \delta$. Write
 ${\rm div}(s_i)=C_i$, they are transversal cuts of $E_v$ in $\tX^m_1$ with
 $\calO_{\tX^m_1}(C_i+E^*_v)=\calO_{\tX^m_1}$ in $\pic(\tX^m_1)$.
 Then we make along each $C_i$ an analytic plumbing by gluing a
 disc bundle over $E_{v_i}$ with Euler number $N$ to $\tX^m_1$ such that $E_{v_i}\cap \tX^m_1=C_i$.
 In this way we get a resolution space $\tX^{me}_1$ associated with $\Gamma_1^{me}$.

 In fact, for these analytic gluings, any transversal cut $C'$  works (instead of $C_i$'s considered above).
 Indeed, by {\it Step 2}
  we have
  $v\not\in B_0(\tX^m_1)$. In particular, the image of the Abel map $c^{-E^*_v}$ (associated with a large cycle $Z\subset L(\tX^m_1)$) is a point, hence any $C'\in \eca^{-E^*_v}(\tX^m_1)$ satisfies
 $\calO_{\tX^m_1}(C'+E^*_v)\simeq\calO_{\tX^m_1}$. This together with the construction from the proof of Theorem
 \ref{th:WECCCrit} also shows that $\tX^{me}_1$ satisfies WECC.
Therefore Corollary \ref{cor:WECC} applies and \begin{equation}\label{eq:restrtilde}
\calO_{\tX^{me}_1}(l'))|_{\tX^m_1}=\calO_{\tX^m_1}(R(l'))) \ \ \ \ (\mbox{$R$ is the restriction}).\end{equation}

 \vspace{1mm}

 We claim that  $\tX^{me}_1$  satisfies ECC as well.
 Let us focus first on an end vertex $u$ of $\Gamma^{me}_1$ different than
  $v_{i}$ ($i\geq 2$). By Lemma \ref{lem:proof}
 we know that the restriction $H^1(\calO_{\tX^{me}_1})\to H^1(\calO_{\tX^{m}_1})$ is an isomorphism,
 hence we have  a bijection
 $\pic^{-E^*_u}(\tX^{me}_1)\to \pic^{E^*_u}(\tX^{m}_1)$ too. This bijection together with
 $\calO_{\tX^m_1}(-E^*_u)\in \im (c^{-E^*_u}(\tX^m_1))$ (ECC for $\tX^m_1$)  and (\ref{eq:restrtilde})
 gives that  $\calO_{\tX^{me}_1}(-E^*_u)\in \im (c^{-E^*_u}(\tX^{me}_1))$, hence ECC for $\tX^{me}_1$ and vertex $u$.

 Finally take the end vertex $u=v_i$ ($i\geq 2$) of $\Gamma^{me}_1$.  Since $\tX^{me}_1$ satisfied WECC and
 $u\not \in B_0(\Gamma^{me}_1)$, by Lemma \ref{lem:nobar} we get  $\calO_{\tX^{me}_1}(-E^*_u)\in \im (c^{-E^*_u}(\tX^{me}_1))$ again.
\end{proof}

\begin{example} \ \cite[3.1.4]{Sell}, \cite[Ex. 10.4]{NWECTh}
Consider the weighted homogeneous
hypersurface elliptic singularity
$\{z^2=x^4+y^9\}$. It has $m=1$ hence $p_g=2$.  Its graph is

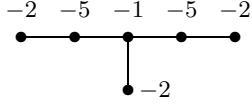
\begin{picture}(200,50)(-100,0)
\put(50,30){\circle*{4}}
\put(-30,30){\circle*{4}}
\put(-30,40){\makebox(0,0){\small{$-2$}}}\put(-10,40){\makebox(0,0){\small{$-5$}}}
\put(10,40){\makebox(0,0){\small{$-1$}}}
\put(30,40){\makebox(0,0){\small{$-5$}}}\put(50,40){\makebox(0,0){\small{$-2$}}}
\put(20,10){\makebox(0,0){\small{$-2$}}}
\put(-10,30){\circle*{4}}
\put(10,30){\circle*{4}}
\put(30,30){\circle*{4}}
\put(10,10){\circle*{4}}
\put(-30,30){\line(1,0){80}}\put(10,10){\line(0,1){20}}
\end{picture}

This analytic structure (e.g., since it is weighted homogeneous) satsifies ECC.
 The point is that it has a (positive weight, $p_g$--constant) hypersurface deformation
$\{z^2=x^4+y^9+txy^7\}$, which has even non--degenerate Newton principal part, but which is not a splice quotient deformation.
In other words, this hypersurface/Gorenstein analytic type does not satisfy ECC.
This is verified (see [loc.cit.]) by checking  that the deformation monomial cannot be realized by
splice quotient equations. By our results, this analytic type does not satisfy  WECC  either.
However, we do not know any `elementary verification' of  this statement, valid even only
for this particular example. (This shows that usually the `direct'  verification of WECC can be very hard.)
\end{example}

\section{Appendix}

\subsection{} Let us fix the following minimal resolution graph $\Gamma$
(where the $(-2)$--vertices are unmarked).

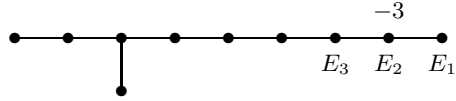
\begin{picture}(200,50)(-20,0)
\put(170,30){\circle*{4}}\put(190,30){\circle*{4}}\put(210,30){\circle*{4}}\put(230,30){\circle*{4}}
\put(150,30){\circle*{4}}
\put(70,30){\circle*{4}}
\put(210,40){\makebox(0,0){\small{$-3$}}}
\put(230,20){\makebox(0,0){\small{$E_1$}}}
\put(210,20){\makebox(0,0){\small{$E_2$}}}\put(190,20){\makebox(0,0){\small{$E_3$}}}
\put(90,30){\circle*{4}}
\put(110,30){\circle*{4}}
\put(130,30){\circle*{4}}
\put(110,10){\circle*{4}}
\put(70,30){\line(1,0){160}}\put(110,10){\line(0,1){20}}
\end{picture}

It is a numerically Gorenstein graph with $L'=L$ and $m=1$,
hence $Z_K=Z_{min}+C$ ($C=Z_{B_1}$). $B_1$ is obtained by deleting
$E_1$; $Z_{min}=E_1^*$, $Z_K=E_2^*$. Note that $C^2=-1$.

We show that any non--Gorenstein
 analytic type supported by this topological type admits a special unique  line bundle
$\calL\in\pic^0(\tX)$ such that the cycle of fixed components $l$ of $\calL$ is
$2Z_{min}$, which is $>Z_K$.
(In any other situation $l\leq Z_K$, hence $l\in \{0,Z_{min}, Z_K\}$, cf. Lemma \ref{lem:Ci}.)

We will brake the discussion in several steps.

\bekezdes \label{bek:A1} {\bf The starting point.} The cycle $l$ of fixed components is zero if and only if  $\calL\simeq\calO_{\tX}$. Otherwise,
since $l\in\calS$, necessarily $l\geq Z_{min}$. In the sequel we assume $l\geq Z_{min}$.

\bekezdes \label{bek:A2} {\bf Inequalities for $l$.} {\it We claim that {\it (a)} if $l>Z_{min}$ then $l\geq Z_K$, and {\it (b)} if $l>Z_K$ then $l\geq 2 Z_{min}$.}
(Here the only needed property of $l$ is $l\in\calS$.)

For {\it (a)} use Lemma \ref{lem:cs2}. According to this algorithm, if $E_1\subset |l-Z_{min}|$ then $l\geq 2Z_{min}$, and if $E_v\subset |l-Z_{min}|$ $(E_v\not=E_1)$ then $l\geq Z_K$. For {\it (b)}
let us denote by $\Gamma_8$ the $E_8$--subgraph of $\Gamma$ (obtained from
$\Gamma$ by deleting $E_1$ and $E_2$). Assume that $l>Z_K$ but $l\not\geq 2Z_{min}$.
Then $l=Z_K+x$ with $x>0$ and $x$ supported on the $\Gamma_8$ subgraph.
Then for any $v\in\calv(\Gamma_8)$ one has
$(x,E_v)=(l,E_v)\leq 0$, hence $x\in \calS(\Gamma_8)\setminus \{0\}$, hence $x\geq Z_{min}(\Gamma_8)$.
In particular, the coefficient  of $x$ at $E_3$ is $\geq 2$. But then $(l, E_2)\leq 0$ fails, which is a contradiction.

\bekezdes \label{bek:A3} Using the exact sequence $0\to \calL(-Z_{min})\to \calL\to \calL|_{Z_{min}}\to 0$ and
$H^0(\calL(-Z_{min}))=H^0(\calL)$ and $\chi(Z_{min})=0$ we get that
$h^1(\calL(-Z_{min}))=h^1(\calL)$.

\bekezdes \label{bek:A4} {\bf Characterization of $l\not=Z_{min}$.} {\it We claim that
$l>Z_{min}$ if and only if $h^1(\calL(-Z_{min}))=0$. }

Consider the exact sequence $0\to \calL(-Z_K)\to \calL(-Z_{min})\to \calL(-Z_{min})|_{C}\to 0$.
Here $h^1(\calL(-Z_K))=0$ and $\chi(\calL(-Z_{min})|_{C})=0$. Therefore,
(use also \ref{bek:A2}) $l>Z_{min} \, \Leftrightarrow\, l\geq Z_K \, \Leftrightarrow\,
H^0(\calL(-Z_K))=H^0(\calL(-Z_{min}))\, \Leftrightarrow\, h^1(\calL(-Z_{min}))=0$.

\bekezdes \label{bek:A5} {\bf Characterization of fixed components of $\calL(-Z_K)$.}
Note first that $Z_K+E_1=2Z_{min}$. Using \ref{bek:A2}{\it (b)} one obtains that $\calL(-Z_K)$ has a nontrivial fixed component if and only if $E_1$ is a fixed component.
Then from the exact sequence
$0\to \calL(-2Z_{min})\to \calL(-Z_K)\to \calL(-Z_K)|_{E_1}\to 0$ one gets that $E_1$ is a fixed component of $\calL(-Z_K)$ if and only if $h^1(\calL(-2Z_{min}))\not=0$ if and only if
$h^1(\calL(-2Z_{min}))=1$.

\bekezdes \label{bek:A6} By \ref{bek:A2} $l$ is either $Z_{min}$, or $Z_K$,  or it is $>Z_K$. We claim that in the Gorenstein case $l>Z_K$ cannot happen. Indeed, if $l>Z_K$ then
$h^1(\calL(-Z_{min}))=0$ (by \ref{bek:A4}) and $h^1(\calL(-2Z_{min}))\not=0$ (by \ref{bek:A5}).
On the other hand, $(X,o)$ is Gorenstein if and only if $Z_{min}=Z_{max}$ (cf. Theorem \ref{e220} and
\ref{bek:discussion}--\ref{bek:discussion2}). Hence $\calO_{\tX}(-Z_{min})$ has no fixed components,
let $s$ be a generic section of it (that is, $s$ is the generic linear section).
Then consider the exact sequence $0\to \calL(-Z_{min})\stackrel{\cdot s}
\longrightarrow \calL(-2Z_{min})\to \calC\to 0$
where $\cdot s$ is multiplication by $s$ and $\calC$ is a Stein cut of $E_1$ with $h^1(\calC)=0$.
Hence $h^1(\calL(-Z_{min}))\geq h^1(\calL(-2Z_{min}))$, a contradiction.

\bekezdes \label{bek:A7} Next assume that $l>Z_K$. By the above discussion  this means that
$l\geq 2Z_{min}$,
$(X,o)$ is not Gorenstein and it has  $p_g=1$,
$h^1(\calL(-Z_{min}))=h^1(\calL)=0$ (cf. \ref{bek:A3}--\ref{bek:A4}),
$h^1(\calL(-2Z_{min}))=1$  (cf. \ref{bek:A5}).

From the exact sequence $0\to \calL(-l)\to \calL\to \calL|_l\to 0$, $H^0 (\calL(-l))=H^0(\calL)$, and
$h^1(\calL)=0$ we get that necessarily $\chi(l)=h^1(\calL(-l))$ ($\dag$). On the other hand,
from the definition of $l$ we have that $\calL(-l)\in \im (c^{-l})$, hence by Theorem \ref{th:Vk}{\it (8)} and
Theorem \ref{prop:AZ}{\it (c)} $h^1(\calL(-l))=
p_g(X_{\calv\setminus I},o_{\calv\setminus I})$, where $I$ is the $E^*$--support of $l$.

Clearly $I\not=\emptyset$. We consider two cases. If $I\not =\{E_1\}$, then
 $(X_{\calv\setminus I},o_{\calv\setminus I})$ is necessarily
rational with $p_g(X_{\calv\setminus I},o_{\calv\setminus I})=0$,
hence $\chi(l)=0$ too. We claim that this cannot happen, since
$l>Z_K$ implies $\chi(l)>0$. Indeed, consider $x:=Z_K-l\in L_{<0}$, and the Laufer sequence
from Lemma \ref{lem:cs2}
connecting $x$ with $s(x)=0$. Along the sequence $\chi$ is non--increasing and in the very last step
before $z_t=0$ we have $z_{t-1}=-E_v$ for some $v$. But $\chi(-E_v)>0$ hence $\chi(l)>0$ too.

(The fact that $l\in\calS$ and $\chi(l)=0$ imply  $l\in \{C_i\}_i$ can be deduced also from
\cite[Th. 6.3]{Tomari85}, or also from  the
structure of the graded root associated with elliptic singularities, cf. \cite{NOSZ}.)

In particular, the only remaining possibility is the second case $I=\{E_1\}$. This means that
$l=nE_1^*=nZ_{min}$ for some $n\geq 2$. In this case $(X_{\calv\setminus I},o_{\calv\setminus I})$
is the minimally elliptic singularity $(X_1,o_1)$ with $p_g(X_1,o_1)=1$, hence
form ($\dag$) we have $\chi(l)=1$. Since $\chi(nZ_{min})=n(n-1)/2$, we get that $n=2$ is the
unique possibility.

Summarized, if $l>Z_K$ then necessarily $l=2Z_{min}$ (and $(X,o)$ must satisfy
all the cohomological restrictions listed at the beginning of this subsection).

\bekezdes \label{bek:A8} We show that $l=2Z_{min}$ can be realized for some special $\calL$ indeed.

Fix any non--Gorenstein analytic type $(X,o)$  and its resolution
$\tX$ with dual graph $\Gamma$.

First we consider the Abel map $c^{-Z_{min}}$. Since the $E^*$--support $I$ of $Z_{min}=E_1^*$
is $E_1$, $p_g=1$ (cf. \ref{e220})
and this $p_g$ is already supported on $C$, from Theorem \ref{prop:AZ} it follows that $\dim(V(I))=0$.
Hence $\im (c^{-Z_{min}})$ is a
point, say $\calb_1\in \pic^{-Z_{min}}$.
Since $Z_{min}\not=Z_{max}$ (the non--Gorenstein property, see again Theorem \ref{e220}),
 $\calO_{\tX}(-Z_{min})$ has nontrivial
 fixed components, that is, $\calO_{\tX}(-Z_{min})\not \in \im(c^{-Z_{min}})$. In other words,
 $\calL_1:= \calb_1(Z_{min})= \im (\widetilde{c}^{-Z_{min}})\not=0$ in $\pic^0$.

By additivity, cf. \ref{bek:addDiv},
 $\im (c^{-2Z_{min}})$ is a
point too, say $\calb_2\in \pic^{-2Z_{min}}$, and set  $\calL_2:= \calb_2(2Z_{min})
= \im (\widetilde{c}^{-2Z_{min}})\in\pic^0$.
By additivity again, $\calL_2=\calL_1+\calL_1$ (using additive notation of the group structure of
$\pic^0=H^1(\calO_{\tX})=\C$), hence $\calL_2\not=0$ as well.

We set $\calL:=\calL_2\in\pic^0$. Then $\calL(-2Z_{min})=\calb_2= \im(c^{-2Z_{min}})$, hence
Theorem \ref{th:Vk}{\it (8)} applies and  we get $h^1(\calL(-2Z_{min}))=1$.

Consider next the bundle $\calL(-Z_{min})=\calb_2(Z_{min})$. Its restriction to
$C'_1=C$ is $\calO_{C'_1}(Z_{min})$ (Indeed, the restriction of $\eca^{-Z_{min}}$ to
$C'_1$ is the empty divisor, hence the restriction of $\calb_2$ to $C'_1$ is the  trivial bundle).
Furthermore, by Theorem \ref{e220}{\it (d)} $\calO_{C'_1}(Z_{min})$
 is not the trivial bundle in $\pic^0(C'_1)$.
By Theorem \ref{prop:vanishing}{\it (a)} we get that
$h^1(\calL(-Z_{min}))=h^1(C'_1, \calO_{C'_1}(Z_{min}))$. However, since $\calO_{C'_1}(Z_{min})$
is nontrivial, $h^1(C'_1, \calO_{C'_1}(Z_{min}))=0$ by \cite[\S 3]{weakly}. Therefore,
$h^1(\calL(-Z_{min}))=0$. This combined with \ref{bek:A3} shows that $h^1(\calL)=0$ too.

Finally, use again
 $0\to \calL(-2Z_{min})\to \calL\to \calL|_{2Z_{min}}\to 0$. Since $h^1(\calL)=0$ we get that
 $h^1(\calL|_{2Z_{min}})=0$ too. Therefore,  from $\chi(2Z_{min})=1$ one gets that
 $h^0(\calL|_{2Z_{min}})=1$. Since
 $h^1(\calL(-2Z_{min}))=1$ too, one obtains that $H^0(\calL(-2Z_{min}))\hookrightarrow H^0(\calL) $
 is an isomorphism. This shows that the cycle of fixed components $l$
 of $\calL(-2Z_{min})$ is $\geq 2Z_{min}$.
 But by the previous discussions $l\leq 2Z_{min}$ always. Hence $l=2Z_{min}$.

 \bekezdes \label{bek:A9} $\calL$ constructed in \ref{bek:A8} satisfies another uniqueness property
 as well. Recall that $Z_K=E_2^*$.
 The image of $c^{-Z_K}=c^{-E^*_2}$ is 1--dimensional, and in fact (using the Laufer integration formula
\cite[\S 7]{NNI} applied to  the unique differential form of pole one along $E_2$)
it is the bijective
image of $\eca^{-E^*_2}(Z_{min})=\C^*$ (the moving divisor/point along $E_2\setminus (E_1\cup E_3)$).
Since $Z_{min}$ is the cohomological cycle (or,
for any $Z\geq Z_{min}$ one has $\pic^0(Z)=\pic^0(Z_{min})$),
$\im(c^{-E^*_2}(Z))=\im(c^{-E^*_2}(Z_{min}))$, see also diagram (3.1.1) from \cite{NNI}.
Hence $\im(c^{-Z_K})=\C^*$ in $\pic^{-Z_K} =\C$. In other words, $\pic^{-Z_K}\setminus \im(c^{-Z_K})$
consists of
one point. This is exactly $\calL(-Z_K)$ (since this bundle has nontrivial cycle of fixed components).
In other words,
$\calb_2(E_1)$ is the gap point $\pic^{-Z_K}\setminus \im(c^{-Z_K})$ of $\pic^{-Z_K}$.

This example suggest fully the subtlety of the cycle of fixed components $l$ of a bundle $\calL$
compared with $h^1(\calL)$. In $\pic^{-Z_K}$ any line bundle has $h^1=0$ by the generalized
Grauert--Riemenschneider vanishing. However,
$\pic^{-Z_K}$ might have a nontrivial interesting stratification according to $l$ (and even
the possible values of $l$ are not evident at all).

  \bekezdes \label{bek:A10} Consider the situation from \ref{bek:A8}.
  It is instructive to determine the possible values $l$ for all line bundles
  $\calL_n:=\calb_n(nZ_{min})\in \pic^0$, where $\calb_n=\im(c^{-nZ_{min}})$, $n\geq 0$. Clearly, for $n=0$ we have $l=0$.
If $n=1$ then $\calL_1(-Z_{min})=\calb_1$ is in the image of the Abel map, hence by Theorem
\ref{th:Vk} $h^1(\calL_1(-Z_{min}))=1$. Hence by \ref{bek:A3} $l=Z_{min}$.
If $n=2$ we already know that $l=2Z_{min}$.

Next assume that $n\geq 3$. Then the restrictions to $C'_1$ of both $\calL_n(-Z_{min})=\calb((n-1)Z_{min})$
and $\calL_n(-2Z_{min})=\calb((n-2)Z_{min})$ are nontrivial (the restriction of $\calb_n$ is trivial, while
of $\calO(Z_{min})$ is not), hence $h^1$ of both bundles is zero by \cite[\S 3]{weakly}. Hence by
\ref{bek:A3}  and \ref{bek:A5} one obtains $l=Z_K$.

(The reader is invited to repeat the discussion for all $\calb_n(mZ_{min})$ as well.)

Finally let us provide $h^1(\calL_n)$ for all $n\geq 0$. If $n=0$ then $\calL_0=\calO$,
hence  $h^1(\calL_0)=1$. For $n=1$ one has $h^1(\calL_1)=h^1(\calL_1(-Z_{min}))=h^1(\calb_1)=1$ too, cf. \ref{th:Vk}{\it (8)}.
For $n=2$ by \ref{bek:A8} $h^1(\calL_2)=0$. If $n\geq 3$, in the exact sequence
$0\to \calL_n(-Z_K)\to \calL_n\to \calL_n|_{Z_K}\to 0$ one has
$H^0(\calL_n(-Z_K))=H^0(\calL_n)$ and $\chi(\calL_n|_{Z_K})=0$, hence
$h^1(\calL_n)=h^1(\calL_n(-Z_K))=0$.

\vspace{1mm}

This shows that though for several different Chern classes $l'$ it can happen that they have the same
$I(l')$ and $V(I(l'))$, the corresponding affine spaces $\overline { \im ( \widetilde{c}^{l'})}$
might be different, each individual affine subspace preserves some information about $l'$,
more than just $I(l')$. The above computation shows that even the $h^1$--behaviour  along these subspaces might vary.


\begin{thebibliography}{30}

\bibitem[ACGH85]{ACGH}  Arbarello, E., Cornalba, M., Griffiths, P. A.,
 Harris, J.: Geometry of algebraic curves,
 Vol. I. Grundlehren der Mathematischen Wissenschaften 267, Springer
Verlag, New York, 1985.


\bibitem[A62]{Artin62} Artin, M.:
Some numerical criteria for contractibility of curves on algebraic surfaces.
{\em  Amer. J. of Math.}, {\bf 84}, 485-496, 1962.

\bibitem[A66]{Artin66} Artin, M.:
On isolated rational singularities of surfaces.
{\em Amer. J. of Math.}, {\bf 88}, 129-136, 1966.






\bibitem[CDGZ04]{CDGPs} Campillo, A.,  Delgado, F. and Gusein-Zade, S. M.:
Poincar\'e series of a rational surface singularity, {\em Invent. Math.} {\bf 155} (2004),
no. 1, 41--53.

\bibitem[CDGZ08]{CDGEq}  Campillo, A.,  Delgado, F. and Gusein-Zade, S. M.:
Universal abelian covers of rational
surface singularities and multi-index filtrations,
{\em Funk. Anal. i Prilozhen.} {\bf 42} (2008), no. 2, 3--10.






\bibitem[Du78]{Du}
Durfee, A.H.: The signature of smoothings of complex surface singularities,
  Math. Ann. {\bf 232}, no. 1 (1978), 85-98.



\bibitem[Fl10]{Flamini} Flamini, F.: Lectures on Brill--Noether theory,
 in Proceedings of the workshop "Curves and Jacobians", Eds. J-M Muk, Y. R. Kim, Korea Institute for Advanced Study, (2011), 1-20.


\bibitem[GR62]{GRa} Grauert, H.: \"Uber Modifikationen und exzeptionelle
analytische Mengen,   {\it Math. Ann.} {\bf 146} (1962), 331--368.



\bibitem[GrRie70]{GrRie} Grauert, H. and Riemenschneider, O.:
Verschwindungss\"atze f\"ur analytische
kohomologiegruppen auf komplexen R\"aumen, {\it Inventiones math.}
{\bf 11} (1970), 263--292.

\bibitem[Gro62]{Groth62} Grothendieck, A.: Fondements de la g\'eom\'etrie
alg\'ebrique, [Extraits du S\'eminaire Bourbaki 1957--1962], Secr\'etariat
math\'ematique, Paris 1962.




\bibitem[Kl05]{Kl} Kleiman, St. L.: The Picard scheme, in
`Fundamental Algebraic Geometry: Grothendieck’s FGA Explained',
Mathematical Surveys and Monographs
Volume: 123; 2005, 248--333.

\bibitem[Kl13]{Kleiman2013} Kleiman, St. L.: The Picard Scheme, In `Alexandre
 Grothendieck: A Mathematical Portrait', International Press of Boston, Inc., 2014
 (L. Schneps editor).






\bibitem[L13]{LPhd} L\'aszl\'o, T.: Lattice cohomology and
Seiberg--Witten invariants of normal surface singularities, PhD. thesis,
Central European University, Budapest, 2013.




\bibitem[La71]{Lauferbook} Laufer, H.B.: Normal two--dimensional singularities.
{\em Annals of Math. Studies}, {\bf 71}, Princeton University Press, 1971.


\bibitem[La72]{Laufer72} Laufer, H.B.: On rational singularities,
{\em Amer. J. of Math.}, {\bf 94}, 597-608, 1972.



\bibitem[La77]{Laufer77} Laufer, H.B.: On minimally elliptic singularities,
{\em Amer. J. of Math.} {\bf 99} (1977), 1257--1295.









\bibitem[Mu66]{MumfordCurves} Mumford, D.: Lectures on curves on an algebraic surface,
{\it Ann. of Math. Studies} {\bf 59}, Princeton, 1966.



\bibitem[NN18a]{NNI} Nagy, J., N\'emethi, A.:
The Abel map for surface singularities  I. Generalities and  examples,
arXiv:1809.03737.

\bibitem[NN18b]{NNII} Nagy, J., N\'emethi, A.:
The Abel map for surface singularities  II. Generic analytic structure,
 arXiv:1809.03744.

\bibitem[NN19]{NNtop}  Nagy, J., N\'emethi, A.:
On the topology of elliptic singularities, arXiv:1901.06224.


\bibitem[N99]{weakly} N\'emethi, A.: ``Weakly'' Elliptic Gorenstein
singularities of surfaces,
{\em Inventiones math.},  {\bf 137}, 145-167 (1999).


\bibitem[N99b]{Nfive} N\'emethi, A.: Five lectures on normal surface singularities,
lectures at the Summer School in {\em Low dimensional topology} Budapest,
Hungary, 1998; Bolyai Society Math. Studies {\bf 8} (1999), 269--351.

\bibitem[N05]{NOSZ} N\'emethi, A.:
On the Ozsv\'ath--Szab\'o invariant of negative
definite plumbed 3--manifolds,
{\em Geometry and Topology} {\bf 9} (2005), 991--1042.

\bibitem[N07]{trieste} N\'emethi, A.: Graded roots and singularities,
{\em Singularities in geometry and topology},  World
Sci. Publ., Hackensack, NJ (2007), 394--463.

\bibitem[N08]{NPS} N\'emethi, A.: Poincar\'e series associated with
surface singularities, in Singularities I, 271--297,
{\em Contemp. Math.} {\bf 474}, Amer. Math. Soc., Providence RI, 2008.


\bibitem[N12]{NCL} N\'emethi, A.: The cohomology of line bundles
of splice--quotient singularities,
{\em Advances in Math.} {\bf 229} 4 (2012), 2503--2524.



\bibitem[N18]{ICM}  N\'emethi, A.: Pairs of invariants of surface singularities,
Proceddings of ICM, Rio de Janeiro, 2018.














\bibitem[NW05]{NWsq}
Neumann, W. and Wahl, J.: Complete intersection singularities of splice type as universal abelian covers,
{\em Geom. Topol.} {\bf 9} (2005), 699--755.

\bibitem[NW10]{NWECTh}
W.~D. Neumann and J.~Wahl,
\emph{ The End Curve Theorem for normal complex surface singularities},  J. Eur. Math. Soc. {\bf 12} (2010), 471--503.





\bibitem[O04]{OkumaRat} Okuma, T.: Universal abelian covers of rational surface singularities,
{\it Journal of London Math. Soc.} {\bf 70}(2) (2004), 307-324.

\bibitem[O05]{OkumaEll} Okuma, T.: Numerical Gorenstein elliptic singularities,
Mathematische Zeitschrift 249 (2005), Issue 1, 31--62.


\bibitem[O08]{Ok} Okuma, T.: The geometric genus of splice--quotient singularities,
{\em Trans. Amer. Math. Soc.} {\bf 360} 12 (2008), 6643--6659.

\bibitem[O18]{Okuma18} Okuma, T.: Cohomology of ideals in elliptic surface singularities,
arXiv: 1804.07558v2.

\bibitem[O10]{OECTh}  Okuma, T.:
\emph{Another proof of the end curve theorem for normal surface singularities},
J. Math. Soc. Japan {\bf 62}, Number 1 (2010), 1--11.









\bibitem[PP11]{PPP} Popescu-Pampu, P.: Numerically Gorenstein surface singularities are
homeomorphic to Gorenstain ones,  {\it Duke Math. Journal} {\bf 159} No. 3 (2011), 539-559.







\bibitem[Re97]{MR}  Reid, M.: Chapters on Algebraic Surfaces.
In: Complex Algebraic Geometry,
IAS/Park City Mathematical Series,  Volume {\bf 3}  (J. Koll\'ar editor),
3-159, 1997.


\bibitem[S07]{Sell} Sell, E.: Universal abelian covers of surface singularities $\{z^n=f(x,y)\}$,
UNC PhD Thesis (2007)




\bibitem[To85]{Tomari85} Tomari, M.: A $p_g$--formula and elliptic singularities,
{\it Publ. Res. Inst. Math. Sci. } {\bf 21} (1985), no. 2, 297--354.

\bibitem[To86]{Tomari86} Tomari, M.: Maximal-Ideal-Adic Filtration on $R^1\psi_*\cO_{\tilde{V}}$
for Normal Two-Dimensional Singularities, {\it Advanced Studies in Pure Math.} {\bf 8} (1986),
Complex Analytic Singularities,  633-647.


\bibitem[Wa70]{Wa70} Wagreich, Ph.: Elliptic singularities of surfaces, {\it Amer. J. of Math.},
{\bf 92} (1970), 419--454.

\bibitem[Wa76]{Wa76} Wahl, M. J.: Equisingular deformations of normal surface singularities,
{\it Ann. of Math.} {\bf 104} (1976), 325--356.

\bibitem[Y79]{Yau5} Yau, S. S.-T.: On strongly elliptic singularities,
{\em Amer. J. of Math.}, {\bf 101} (1979), 855-884.

\bibitem[Y80]{Yau1} Yau, S. S.-T.: On maximally elliptic singularities,
{\em Transactions of the AMS}, {\bf 257} Number 2 (1980), 269-329.

\end{thebibliography}
\end{document}